\documentclass[12pt]{article}

\usepackage[utf8]{inputenc}
\usepackage[T1]{fontenc}
\usepackage{amsmath,amssymb,amsthm}
\usepackage{natbib}
\usepackage{graphicx}
\usepackage{multirow}
\usepackage{dsfont}
\usepackage{epstopdf}
\usepackage{booktabs}
\usepackage{textcomp}
\usepackage[normalem]{ulem}
\usepackage{xcolor}
\usepackage{threeparttable}
\usepackage{comment}
\usepackage{float}
\usepackage{indentfirst}
\usepackage{arydshln}
\usepackage{subcaption}
\usepackage{adjustbox}
\usepackage{tabularx}
\usepackage{array}
\usepackage{ragged2e}
\usepackage{soul}
\usepackage{scalerel}
\usepackage[colorlinks,citecolor=blue,urlcolor=blue]{hyperref}
\usepackage{cleveref}

\newtheorem{theorem}{Theorem}
\newtheorem{proposition}{Proposition}
\newtheorem{lemma}{Lemma}

\newtheorem{definition}{Definition}
\newtheorem{corollary}{Corollary}
\newtheorem{condition}{Condition}

\newcommand{\hpfe}{p_\textnormal{fe}^{\textnormal{h}}}
\newcommand{\se}{\textnormal{se}}

\newcommand{\lmm}{\operatorname{lm}}
\newcommand{\sign}{\operatorname{sign}}
\newcommand{\sss}{\textnormal{ss}}
\newcommand{\Lin}{\textnormal{L}}

\newcommand{\Prob}{\mathbb{P}}
\newcommand{\cov}{\operatorname{cov}}
\newcommand{\adj}{\textnormal{adj}}
\newcommand{\diag}{\operatorname{diag}}
\newcommand{\var}{\operatorname{var}}
\newcommand{\op}{o_{\textnormal{p}}}
\newcommand{\Op}{O_{\mathbb{P}}}
\newcommand{\wconv}{\stackrel{d}{\rightarrow}}

\newcommand{\mk}{\mathcal{K}}

\newcommand{\ma}{\mathcal{A}}
\newcommand{\ms}{\mathcal{S}}

\newcommand{\rem}{\textnormal{rem}}

\newcommand{\remss}{\textnormal{rem-ss}}
\newcommand{\repss}{\textnormal{rep-ss}}
\newcommand{\repfe}{\textnormal{rep-fe}}
\newcommand{\rep}{\textnormal{rep}}
\newcommand{\fe}{\textnormal{fe}}

\newcommand{\lm}{\textnormal{lm}}
\newcommand{\bs}{\boldsymbol}
\newcommand{\vtt}{V_{\tau\tau}}

\newcommand{\sx}{\textnormal{sx}}
\newcommand{\sumi}{\sum_{i=1}^n}
\newcommand{\pinf}{\mathbb{P}_\infty}
\newcommand{\hpL}{p_\textnormal{L}^{\textnormal{h}}}

\DeclareMathOperator*{\argmin}{arg\,min}

\newcommand{\lx}[1]{#1}

\def\spacingset#1{\renewcommand{\baselinestretch}{#1}\small\normalsize}
\spacingset{1}

\title{\bf Rerandomization for covariate balance mitigates $p$-hacking caused by strategically selecting covariates in regression adjustment}
\author{Xin Lu, Peng Ding\thanks{Xin Lu, Department of Statistics and Data Science, Washington University in St. Louis, St. Louis, MO 63130, USA. Peng Ding, Department of Statistics, University of California, Berkeley, CA 94720, USA (E-mail: pengdingpku@berkeley.edu).}}
\date{}

\begin{document}

\maketitle
\bigskip

\begin{abstract}
Rerandomization enforces covariate balance across treatment groups in the design stage of experiments. Despite its intuitive appeal, its theoretical justification remains unsatisfying because its benefits of improving efficiency for estimating the average treatment effect diminish if we use regression adjustment in the analysis stage. To strengthen the theory of rerandomization, we show that it mitigates false discoveries resulting from $p$-hacking caused by strategically selecting covariates in regression adjustment to get more significant $p$-values. Moreover, we show that rerandomization with a sufficiently stringent threshold can resolve such  $p$-hacking. As a byproduct, our theory offers guidance for choosing the threshold in rerandomization in practice.
\end{abstract}


\newpage
\spacingset{1.5}

\section{Introduction}

Randomized experiments are the gold standard for estimating the average treatment effect (ATE). They balance observed and unobserved covariates on average. However, chance covariate imbalances are common in realized treatment allocations, which may complicate the interpretation of the estimated ATE.
Rerandomization, termed by \citet{cox1982randomization} and \citet{morgan2012rerandomization}, enforces covariate balance in the design stage by rejecting allocations with covariate imbalances, while regression adjustment addresses covariate imbalances in the analysis stage \citep{fisher1935design,freedman2008regression,Lin2013Agnostic}.

\citet{bruhn2009pursuit} conducted a survey of leading experimental researchers in development economics,
and suggested that rerandomization is commonly used yet often poorly documented. 
 \citet{gerber2012field} recommended rerandomization as a way to approximate blocking. 
Although rerandomization is intuitive and has been widely implemented in practice, its necessity has not been fully justified. Since \citet{morgan2012rerandomization}, most papers on rerandomization have focused on the aspect of efficiency for estimating the ATE. That is, rerandomization decreases the variances of the estimators for the ATE, if the covariates are predictive of the outcomes \citep{branson2016improving,morgan2012rerandomization,Li9157,wang2022rerandomization,wang2021rerandomization,2020Rerandomization}. However, those papers demonstrated that when we use the same set of covariates in rerandomization and regression adjustment, rerandomization cannot further improve efficiency compared with solely regression adjustment. Given regression adjustment, the additional benefits of rerandomization remain unclear if the criterion is the efficiency for estimating the ATE.
 
 We depart from the current literature on efficiency and offer a new perspective. In particular, we demonstrate that rerandomization can mitigate  $p$-hacking caused by strategically selecting covariates in regression adjustment. The term $p$-hacking, coined by \citet{simmons2011false}, refers to the practices that a researcher might use to
generate more significant $p$-values \citep{brodeur2020methods}.  It has become a pervasive and profound issue, undermining the credibility and reproducibility of scientific discoveries \citep{john2012measuring,ioannidis2005most}. 
 Among the practices of $p$-hacking, strategically selecting covariates is especially noteworthy, as it is common for researchers to explore various covariate combinations in the analysis stage \citep{simmons2011false}. While a detailed pre-analysis plan can prevent such a problem, it can be costly and therefore is not required for every study \citep{banerjee2020praise}.

 \lx{In this paper, we focus exclusively on $p$-hacking caused by
strategically selecting covariates in regression adjustment. Throughout this
paper, the term $p$-hacking refers exclusively to this specific practice.}
 Our contributions are twofold. First,  we study two rerandomization schemes: rerandomization with Mahalanobis distance (ReM) \citep{cox1982randomization,morgan2012rerandomization} and rerandomization based on the $p$-values of marginal $t$-tests (ReP) \citep{gerber2012field,zdrep}. We demonstrate that rerandomization can lower the type I error rate due to $p$-hacking, and as rerandomization enforces nearly perfect balancing of the covariates, it can control the type I error rate. The intuition is that when we balance the covariates by rerandomization, the $p$-values from regression adjustments with different covariate combinations tend to be coherent with each other; see \citet{zdrep} for more discussion about rerandomization and coherence. Therefore, $p$-hacking will not distort the $p$-values too much under rerandomization. Second, our theory gives recommendations for the rerandomization thresholds to control the type I error rates.
\lx{To emphasize, although rerandomization can provide a safeguard against
$p$-hacking caused by strategically selecting covariates, we do not
encourage researchers to base their inference on a hacked $p$-value.
Researchers should prespecify their primary analysis whenever possible. }

\textbf{Notation:}  Let $\Prob_{\infty}$ denote the probability measure with respect to the asymptotic distribution. 
For two finite-population arrays $\{\bs{a}_i\}_{i=1}^n$ and $\{\bs{b}_i\}_{i=1}^n$, define the finite-population mean $\bar{\bs{a}} = n^{-1}\sumi \bs{a}_i$ and  covariance $    \bs{S}_{\bs{a}\bs{b}} = (n-1)^{-1}\sum_{i=1}^n (\bs{a}_i-\bar{\bs{a}})(\bs{b}_i-\bar{\bs{b}})^\top.$
Let $\Phi(\cdot)$ be the cumulative distribution function of a standard Gaussian distribution. Let $z_{1-\alpha}$ be the  $\alpha$th upper quantile of a standard Gaussian distribution. Let $\chi^2_{K}$ be the $\chi^2$ distribution with degrees of freedom $K$. Let $\chi_{K,1-\alpha}^2$ be the upper $\alpha$th quantile of $\chi^2_{K}$. For a set of tuples $\{(u_i,\bs{V}_{i1},\ldots,\bs{V}_{iL}):u_i\in\mathbb{R},\bs{V}_{il}\in\mathbb{R}^{K_l},i=1,\ldots,n,l=1,\ldots,L\}$, let $\operatorname{lm}(u_i\sim \bs{V}_{i1}+\cdots+\bs{V}_{iL})$ be the ordinary least squares (OLS) fit of $u_i$ on  $(\bs{V}_{i1},\ldots,\bs{V}_{iL})$. 

\section{Design-based framework for hacked $p$-values}

Consider a completely randomized experiment (CRE) with $n_z$ units, $z=0,1$, assigned to the treatment group $z$, and $n_1+n_0=n$. Let $Y_i(z)$ be the potential outcome of unit $i$ under the treatment group $z$ and $Z_i\in\{0,1\}$ be the treatment indicator of unit $i$. The observed outcome is $Y_i = Y_i(1)Z_i + Y_i(0)(1-Z_i)$. Let $\tau_i = Y_i(1)-Y_i(0)$ be the individual level treatment effect. For unit $i$, we observe $K$ covariates, $\bs{x}_i = ({x}_{i1},\ldots,{x}_{iK})$.  We are interested in estimating the ATE:
 $$\bar{\tau} = n^{-1}\sumi \tau_i.$$  
 
 We focus on the design-based inference framework, which relies solely on the
randomness of treatment assignment, conditional on the potential outcomes and covariates \citep{neyman1923,freedman2008regression,imbens2015causal}. Consequently, conclusions drawn from this framework do not depend on the underlying stochastic process of the potential outcomes and covariates; see \citet{li2017general} for a review.  \citet{fisher1935design} considered testing the sharp null hypothesis:
\[
H_{0\textsc{f}}: Y_i(1)=Y_i(0),\quad i=1,\ldots,n.
\]
Although we observe only one of the two potential outcomes for each unit, the sharp null hypothesis enables us to impute all the missing potential outcomes and implement the Fisher randomization test. Specifically, for any test statistic, we can compute its finite-sample exact $p$-value by randomly permuting the treatment indicators to simulate the randomization distribution. 

\citet{neyman1923} considered the weak null hypothesis: $$H_{0\textsc{n}}: \bar{\tau} = 0.$$
The common choices of the test statistics for the weak null hypothesis are the studentized ATE estimators. For example, the coefficient estimator of $Z_i$ from the following OLS regression is an unbiased estimator of $\bar{\tau}$ \citep[Theorem 1]{freedman2008regression}:
\begin{align}
\label{eq:unadjusted-OLS}
     \lmm(Y_i\sim 1+Z_i),
\end{align}
 which we denote by $\hat{\tau} = n_1^{-1}\sum_{i:Z_i=1}Y_i-n_0^{-1}\sum_{i:Z_i=0}Y_i$.
The corresponding $t$-statistic using the associated Eicker–Huber–White (EHW) standard error is 
asymptotically valid for testing $H_{0\textsc{n}}$ \citep[chap.~4]{ding2024first}.
    
With covariates $\bs{x}_i$, regression adjustment is often used in the analysis stage to reduce the variance of the ATE estimator, therefore improving the power of the $t$-test \citep{fisher1935design,freedman2008regression,Lin2013Agnostic}. 
Assume that the covariates are centered such that $\bar{\bs{x}} = 0$ and that the covariance matrix, $\bs{S}_{\bs{x}\bs{x}}$, is nonsingular to exclude collinearity. \citet{Lin2013Agnostic} proposed the interacted regression for regression adjustment, which adds covariates and the treatment-covariates interaction in the OLS regression:
\begin{align}
\label{eq:formula-Lin}
  \lmm(  Y_i \sim 1 + Z_i +\bs{x}_i + Z_i\bs{x}_i).
\end{align}
 It guarantees the asymptotic variance reduction of the ATE estimator, i.e., its coefficient estimator of $Z_i$ has smaller asymptotic variance than $\hat{\tau}$. Let $\hat{\tau}_{\Lin}$, $\hat{\se}_{\Lin}$, and $T_{\Lin}=\hat{\tau}_{\Lin}/\hat{\se}_{\Lin}$ denote the ATE estimator, the EHW standard error, and the $t$-statistic, respectively, obtained from Lin's interacted regression. \citet{Lin2013Agnostic} showed that the $p$-value of the two-sided $t$-test $p_{\Lin} =  2(1-\Phi(|T_{\Lin}|))$ is valid 
for testing the weak null hypothesis $H_{0\textsc{n}}$ even when the linear model in \eqref{eq:formula-Lin} is misspecified.

However, $p_{\Lin}$ can be susceptible to $p$-hacking caused by strategically selecting covariates, driven by the pressure on researchers to publish statistically significant results. Even if the researchers have no intention of fraud, it is common for them to explore various covariate combinations, and report only favorable $p$-values \citep{simmons2011false}. 

We now formally define the hacked $p$-value caused by strategically selecting covariates. Let $[K] = \{1,2,\ldots,K\}$. Given $\mk\subseteq [K]$, let $\bs{x}_{i\mk} = ({x}_{ik},k\in \mk)$ be the subset of covariates in $\mk$. Let $\hat{\tau}_{\Lin,\mk}$, $\hat{\se}_{\Lin,\mk}$, $T_{\Lin,\mk}$, and $p_{\Lin,\mk}$ denote, respectively, the ATE estimator, the EHW standard error, the $t$-statistic, and the two-sided $p$-value obtained through \eqref{eq:formula-Lin}, replacing $\bs{x}_i$ with $\bs{x}_{i\mk}$.  \eqref{eq:unadjusted-OLS} corresponds to $\mk = \emptyset$ and \eqref{eq:formula-Lin} corresponds to $\mk = [K]$.
There may be many ways for a researcher to strategically select covariates. To illustrate the idea, we consider the scenario in which the researcher might search through all possible 
covariate combinations $\mk$ in $[K]$ to report the most significant $p$-value.  Define the hacked $p$-value for testing $H_{0\textsc{n}}$ as the minimal $p$-value over all possible covariate combinations:
      \[
  \hpL = \min_{\mk\subseteq [K]} p_{\Lin,\mk}.
    \]

Let $r_z = n_z/n$, $z=0,1$, be the proportion of treatment group $z$. Let $e_i(z)$, $i=1,\ldots,n$, be the residual of  $\lmm(Y_i(z)\sim 1+ \bs{x}_{i})$, and $R^2_{\bs{x}}$ be the $R^2$ of $\lmm(r_1^{-1}Y_i(1)+r_0^{-1}Y_i(0)\sim 1+ \bs{x}_{i})$, which are both finite-population quantities. 
We study the asymptotic property of $\hpL$ under the following standard assumption for asymptotic analysis \citep{li2017general}.
\begin{condition}
    \label{a:crt-regularity-condition}
    As $n \rightarrow \infty$, for $z=0,1$, (i) $r_z$ has a positive limit; (ii)  $S_{Y(z)Y(z)}$,  $\bs{S}_{\bs{x}\bs{x}}$, $\bs{S}_{\bs{x} Y(z)}$, $S_{\tau\tau}$ have finite limits, the limit of $S_{r_1^{-1}e(1)+r_0^{-1}e(0),r_1^{-1}e(1)+r_0^{-1}e(0)}$ is positive, and the limit of $\bs{S}_{\bs{x}\bs{x}}$ is nonsingular; and (iii) $\max_{1 \leq i \leq n} |Y_i(z)-\bar{Y}(z)|^2 = o(n)$, $\max_{1 \leq i \leq n} \lVert \bs{x}_i\lVert_{\infty}^2 = o(n)$. 
 \end{condition} 

\lx{\Cref{a:crt-regularity-condition} (i) requires that both the treatment and
control groups have large sample sizes. \Cref{a:crt-regularity-condition} (ii)
requires that the relevant finite-population variances and covariances have
finite limits, that the regression residual variance is nonzero, and that the
covariance matrix of the covariates is nonsingular.
\Cref{a:crt-regularity-condition} (iii) is automatically satisfied when
the covariate entries and potential outcomes are uniformly bounded.}

\begin{proposition}
\label{prop:type I-error-rate-asymptotic}
Under $H_{0\textsc{f}}$, we have $\Prob_{\infty}(\hpL \leq \alpha) \geq \alpha$, with equality if and only if $R^2_{\bs{x}}=0$.
\end{proposition} 

 Proposition \ref{prop:type I-error-rate-asymptotic} demonstrates that under $H_{0\textsc{f}}$, the type I error rate of the hacked $p$-value is greater than the significance level unless $R^2_{\bs{x}}=0$. Under  $H_{0\textsc{n}}$, the problem will be worse: since $H_{0\textsc{f}}$ implies $H_{0\textsc{n}}$, we have $\sup_{H_{0\textsc{n}}}\Prob_{\infty}(\hpL \leq \alpha) \geq \sup_{H_{0\textsc{f}}}\Prob_{\infty}(\hpL \leq \alpha)$.

In the next section, we show that rerandomization can mitigate the type I error rate inflation due to $p$-hacking.

\section{Rerandomization mitigates $p$-hacking}

In the design stage, covariate imbalances may occur, complicating the interpretation of the ATE estimators. Rerandomization has been used in the design stage to avoid covariate imbalances. It draws a desired assignment by rejective sampling until the assignment satisfies a given requirement of the covariate balance measure \citep{cox1982randomization,morgan2012rerandomization}.  Most of the existing papers on rerandomization have focused on its benefits for improving efficiency, namely its ability to reduce the asymptotic variance of the unadjusted ATE estimators; see, for example, \citet{Li9157,wang2021rerandomization}. However, the benefits diminish if we use regression adjustment in the analysis stage \citep{2020Rerandomization}. \citet{zdrep} provided a different angle, demonstrating that rerandomization improves coherence between regression adjusted estimators. 

In this section, we quantify the benefit of rerandomization under a different statistical framework. In particular, we demonstrate that rerandomization can mitigate $p$-hacking caused by strategically selecting covariates in regression adjustment and can even resolve it entirely when the threshold is stringent enough. To illustrate the concept, we focus on two
rerandomization schemes, ReM and ReP, and the theory extends to other rerandomization schemes.

\subsection{ReM mitigates $p$-hacking}

  The difference in covariate means $\hat{\bs{\tau}}_{\bs{x}} = n_1^{-1}\sum_{i:Z_i=1}\bs{x}_i-n_0^{-1}\sum_{i:Z_i=0}\bs{x}_i$ provides an intuitive measure of imbalance. Rerandomization based on Mahalanobis distance accepts an assignment if and only if the Mahalanobis distance of $\hat{\bs{\tau}}_{\bs{x}}$ is below a given threshold $a$ \citep{morgan2012rerandomization}: $$\ma_{\rem}(a) = \{\hat{\bs{\tau}}_{\bs{x}}^\top \cov^{-1}(\hat{\bs{\tau}}_{\bs{x}})\hat{\bs{\tau}}_{\bs{x}}\leq a\}.$$ 

\begin{theorem}
\label{thm:rem-mitigates-p-hacking}
    Under Condition \ref{a:crt-regularity-condition} and $H_{0\textsc{n}}$, we have, for any $a>0 $ and $\alpha\in (0,1) $, (i)
    $\Prob_{\infty}( \hpL \leq \alpha\mid \ma_{\rem}(a)) -\Prob_{\infty}( \hpL \leq \alpha) \leq 0;$
    (ii) $ \lim_{a\rightarrow 0}\Prob_{\infty}( \hpL \leq \alpha \mid \ma_{\rem}(a)) \leq \alpha.$
\end{theorem}

\Cref{thm:rem-mitigates-p-hacking} is our key result. \Cref{thm:rem-mitigates-p-hacking} (i) states that ReM mitigates $p$-hacking for any positive threshold $a$, because the asymptotic type I error rate of the hacked $p$-value under ReM never exceeds that under CRE. \Cref{thm:rem-mitigates-p-hacking} (ii) states that when the threshold $a$ of ReM tends to $0$, ReM resolves $p$-hacking, controlling the type I error rate. 

We explain the intuition behind the results of \Cref{thm:rem-mitigates-p-hacking}. As we will demonstrate in the Supplementary Material: 
\begin{align}
\label{eq:decompose-of-hat-tau-lin}
    \hat{\tau}_{\Lin, \mk} \approx \hat{\tau} - \bs{\beta}_{\Lin,\mk}^\top \hat{\bs{\tau}}_{\mk},
\end{align}
where $\hat{\bs{\tau}}_{\mk}$ is the subvector of $\hat{\bs{\tau}}_{\bs{x}}$ in $\mk$ and $\bs{\beta}_{\Lin,\mk}$ is the regression coefficients of $\bs{x}_{i\mk}$ in $\lmm(r_1^{-1}Y_i(1)+r_0^{-1}Y_i(0)\sim 1+ \bs{x}_{i\mk})$. Rerandomization ensures that $\hat{\bs\tau}_{\bs{x}}$ is small and $\hat{\tau}_{\Lin,\mk}$, $\mk \subseteq [K]$, are close to each other \citep{zdrep}. Therefore, under rerandomization, $p$-hacking will not change the $p$-values too much.

\lx{\Cref{thm:rem-mitigates-p-hacking} (ii) establishes a limiting result as
$a\to0$. In practice, however, the acceptance probability also tends to zero,
so this limiting case is infeasible in randomized experiments. The practical
question is therefore how to choose a positive, finite threshold. We address
this question in the next subsection.}

\subsection{Guidance on selecting rerandomization threshold}
\Cref{thm:rem-mitigates-p-hacking} motivates a natural question: assume that in the design stage, we do not observe 
 $(Y_i(1),Y_i(0))_{i=1}^n$, we only observe $\bs{x}_i$, and we wish to control the type I error rate due to $p$-hacking in the subsequent analysis stage. How should we specify the threshold of ReM $a$ in the design stage? We address this question in this section.

  Let $R^2_{\mk}$ be the $R^2$ of $\lmm(r_1^{-1}Y_i(1)+r_0^{-1}Y_i(0)\sim 1+ \bs{x}_{i\mk})$. By definition, $R^2_{\bs{x}} = R^2_{[K]}$. Let $[-k] = [K]\backslash \{k\}$. For $k\in [K]$, define
 \[
\Delta_k = R^2_{\bs{x}}-R_{[-k]}^2
 \]
as the difference in $R^2$ of covariate sets $[K]$ and $[-k]$. By definition, $\Delta_k/(1-R_{[-k]}^2)$ is the partial $R^2$ of the full model $\lmm(r_1^{-1} Y_i(1)+r_0^{-1} Y_i(0)\sim 1 + \bs{x}_{i})$ versus the reduced model $\lmm(r_1^{-1} Y_i(1)+r_0^{-1} Y_i(0)\sim 1 + \bs{x}_{i[-k]})$. Hence, $\Delta_k$ can be interpreted as the scaled partial $R^2$ of covariate $k$.  Our subsequent discussion depends on $R^2_{\bs{x}}$ and $\underline{\Delta} = \min_{k\in [K]} \Delta_k$.
 
  Assume that in addition, we have prior knowledge about  $\underline{\Delta}$ and $R^2_{\bs{x}}$ in the design stage. Define $\mathbb{M}(R^2_{\bs{x}}, \underline{\Delta})$ as the set of possible potential outcomes $(Y_i(1), Y_i(0))_{i=1}^n$ given $\underline{\Delta}$ and $R^2_{\bs{x}}$. For a meaningful discussion, we only consider $(R^2_{\bs{x}},\underline{\Delta}) \in (0,1)\times (0,1)$ such that $\mathbb{M}(R^2_{\bs{x}}, \underline{\Delta})\ne \emptyset$. We denote the set of such $(R^2_{\bs{x}},\underline{\Delta})$ by $\mathcal{B} \subseteq (0,1)\times (0,1)$.
  \lx{The condition $\underline{\Delta} >0$ is equivalent to a
nondegenerate covariance matrix of the covariates and that no covariate be
irrelevant, in the sense that its coefficient in the full regression is
nonzero.}

  \Cref{thm:iff-for-rem_control-TIE} gives a necessary and sufficient condition to universally control the type I error rates for all $(Y_i(1),Y_i(0))_{i=1}^n \in \mathbb{M}(R^2_{\bs{x}}, \underline{\Delta})$.

\begin{theorem}
\label{thm:iff-for-rem_control-TIE}
Under \Cref{a:crt-regularity-condition} and $H_{0\textsc{n}}$, for any $(R^2_{\bs{x}},\underline{\Delta})\in \mathcal{B}$, $\alpha \in (0,1)$, we have
\begin{align*}
\max_{(Y_i(1),Y_i(0))_{i=1}^n\in\mathbb{M}(R^2_{\bs{x}},\underline{\Delta})}\Prob_{\infty}( \hpL \leq \alpha\mid \ma_{\rem}(a)) \leq \alpha,
\end{align*}
if and only if $a \leq \bar{a}_{\rem}(\alpha,R^2_{\bs{x}}, \underline{\Delta})$, where
    \begin{align}
    \label{eq:bound-for-ReM}
     \bar{a}_{\rem}(\alpha,R^2_{\bs{x}}, \underline{\Delta}) =  z_{1-\alpha/2}^2\frac{\underline{\Delta}}{\Big(\sqrt{1-R^2_{\bs{x}}}+\sqrt{1-R^2_{\bs{x}}+ \underline{\Delta}}  \Big)^2}.
    \end{align}
\end{theorem}

By \Cref{thm:iff-for-rem_control-TIE}, $\bar{a}_{\rem}(\alpha, R^2_{\bs{x}},\underline{\Delta})$ is the largest allowable threshold to resolve $p$-hacking given $(\alpha, R^2_{\bs{x}},\underline{\Delta})$ and we can specify $\bar{a}_{\rem}$ as the rerandomization threshold. It is intuitive that the threshold $\bar{a}_{\rem}$ should depend on $R^2_{\bs{x}}$, since the magnitude of $\bs{\beta}_{\Lin,\mk}$ in \eqref{eq:decompose-of-hat-tau-lin} is influenced by $R^2_{\bs{x}}$.  While $\hpL$ is not invariant to non-degenerate linear transformations of $\bs{x}_i$, $R^2_{\bs{x}}$ is. Therefore, besides $R^2_{\bs{x}}$, we also need to know $\underline{\Delta}$ to resolve $p$-hacking.

If we only have prior knowledge about the range of $(R^2_{\bs{x}}, \underline{\Delta})$, that is, $(R^2_{\bs{x}}, \underline{\Delta}) \in \mathcal{B}_0$, then to ensure that the threshold is valid for resolving $p$-hacking for all $(R^2_{\bs{x}}, \underline{\Delta}) \in \mathcal{B}_0$, we can specify the threshold as $\inf_{(R^2_{\bs{x}}, \underline{\Delta}) \in \mathcal{B}_0} \bar{a}_{\rem}(\alpha,R^2_{\bs{x}}, \underline{\Delta})$.
 
We can show that $\bar{a}_{\rem}$ is an increasing function of $\underline{\Delta}$ and $R^2_{\bs{x}}$, and a decreasing function of $\alpha$. A small value of $\bar{a}_{\rem}$ suggests that resolving $p$-hacking is challenging. The implications are threefold. First, it is easier to resolve $p$-hacking for a test requiring a lower significance level $\alpha$. Second, it is easier to resolve $p$-hacking when the covariates are predictive of the potential outcomes. Third, it is challenging to resolve $p$-hacking when $\underline{\Delta}$ is small, such as when there are redundant covariates or highly correlated covariates.

Given $R^2_{\bs{x}}$, the $\underline{\Delta}$ can be arbitrarily close to $0$ if there are redundant covariates. 
Consequently, $\bar{a}_{\rem}$ can also be arbitrarily close to $0$. Therefore, when we only know $R^2_{\bs{x}}$, we need additional assumptions to get a nontrivial threshold $\bar{a}_{\rem}\ne 0$.  We consider a set of orthogonal and equally important covariates summarised by the following Condition.
 \begin{condition}
 \label{a:orthogonal-covariates-equally-importance}
     $\bs{S}_{\bs{x}\bs{x}}$ is a diagonal matrix and $R^2_{k} = R^2_1$ for $k\in [K]$.
 \end{condition}

Under \Cref{a:orthogonal-covariates-equally-importance}, we have  $\underline{\Delta} = R^2_{\bs{x}}/K$. By \Cref{thm:iff-for-rem_control-TIE}, we require $a\leq \bar{a}_\rem(\alpha,R^2_{\bs{x}},R^2_{\bs{x}}/K)$ to resolve $p$-hacking, which requires only specification of $R^2_{\bs{x}}$ and $\alpha$. We can check that $\bar{a}_\rem(\alpha,R^2_{\bs{x}},R^2_{\bs{x}}/K)$ is an increasing function of $R^2_{\bs{x}}$ and a decreasing function of $\alpha$, suggesting that ReM more easily resolves  $p$-hacking if $\alpha$ is small and $R_{\bs{x}}^2$ is large. Consider $K=5$, if we specify $\alpha=0.05$, $R^2_{\bs{x}} = 0.6$, we require $a \leq 0.252 \approx \chi_{5,0.15\%}^2$. Namely, we require an asymptotic acceptance probability of $0.15\%$, which is less stringent than the recommended acceptance probability of $0.1\%$ by \citet{Li9157}.

By the finite-population central limit theorem \citep{li2017general}, the asymptotic acceptance probability of $\ma_\rem(a)$ is $\Prob_\infty(\ma_\rem(a)) = \Prob(\chi^2_K\leq a)$.  \Cref{fig:R2_prob} plots the 
asymptotic acceptance probability of $\ma_\rem(a)$ with $a = \bar{a}_{\rem}(\alpha,R^2_{\bs{x}},R^2_{\bs{x}}/K)$ under $\alpha\in \{0.05,0.1\}$, $R_{\bs{x}}^2\in [0.1,1]$ and $K\in \{1,2,3,4,5,6\}$. The asymptotic acceptance probability decreases dramatically as $R^2_{\bs{x}}$ decreases. For example, when $K=5$ and $\alpha = 0.05$, as $R^2_{\bs{x}}$ decreases from $0.7$ to $0.1$, the asymptotic acceptance probability decreases from about $4\times 10^{-3}$ to an extremely stringent threshold of $3\times 10^{-6}$. So resolving $p$-hacking is challenging when the covariates are not predictive of the potential outcomes.

When we know neither $R^2_{\bs{x}}$ nor $\underline{\Delta}$, for any threshold of ReM $a\ne 0$, there always exists type I error inflation for the worst case of $(Y_i(1),Y_i(0))_{i=1}^n$. However, to choose a threshold  $a\ne 0$ of ReM, it is useful to obtain an upper bound on the type I error rate to assess the potential inflation under this threshold. 

\begin{theorem}
\label{thm:ReM-bound-of-I-error-rate-inflation}
Under \Cref{a:crt-regularity-condition} and $H_{0\textsc{n}}$, we have
    \begin{align*}
        \Prob_{\infty}( \hpL \leq \alpha\mid \ma_{\rem}(a)) \leq \Prob\Big((\varepsilon^{2}+V)^{1/2} \geq z_{1-\alpha/2}~\Big | ~ V \leq a\Big),
    \end{align*}
    where $\varepsilon$ and $V$ are independent standard Gaussian and $\chi^2_{K}$ random variables, respectively.
\end{theorem}

The right-hand side of the inequality in \Cref{thm:ReM-bound-of-I-error-rate-inflation} is the upper bound for the type I error rate. The upper bound in \Cref{thm:ReM-bound-of-I-error-rate-inflation} does not depend on $R^2_{\bs{x}}$ and $\underline{\Delta}$. It is also independent of the covariance structure of the covariates, owing to the invariance of ReM under non-degenerate linear transformation of $\bs{x}$. As $a$ approaches $0$, the upper bound tends to $ \Prob(|\varepsilon| \geq z_{1-\alpha/2})=\alpha$, echoing \Cref{thm:rem-mitigates-p-hacking} (ii). 
 For any $a>0$, the upper bound is always greater than $\alpha$. For example, when $K=5$, $\alpha=0.05$ and $a  = \chi^2_{5,0.01}$, the upper bound is approximately $0.063$,  reflecting an inflation of $0.013$ over the significance level of $0.05$. Therefore, we cannot control the type I error rate under $\alpha$ for a test which rejects $H_{0\textsc{n}}$ if $\hpL \leq \alpha$ by specifying the rerandomization threshold.  However, given the upper bound, we can instead specify $\gamma$ and reject $H_{0\textsc{n}}$ if $\hpL \leq \alpha-\gamma$, so that the type I error is controlled under $\alpha$, where $\gamma$ is the solution of 
\[
\Prob\Big((\varepsilon^{2}+V)^{1/2} \geq z_{1-\alpha/2+\gamma/2}~\Big | ~ V \leq a\Big) = \alpha.
\]
Continuing the previous example with $K=5$, $\alpha=0.05$ and $a  = \chi^2_{5,0.01}$, the corresponding $\gamma\approx 0.01$. \lx{This calibration can be substantially less
conservative than Bonferroni or Holm corrections, each of which rejects only if $\hpL\leq\alpha/2^K$, because the researcher searches over $2^K$
covariate specifications. This is because the upper bound exploits the asymptotic distribution under the
rerandomization design to provide a worst-case correction. However, it still remains
conservative because it does not use information about the potential
outcomes. }

\lx{The results in this section yield a practical hierarchy for selecting the ReM
threshold. When credible prior information about outcome--covariate
relationships is available, it can be used to select a design-stage
threshold with uniform type I error control. When such information is
unavailable, the type I error bound instead supports a conservative
analysis-stage correction. \Cref{tab:practical-guidance-rem} summarizes
these recommendations.}

\begin{figure}
    \centering
    \includegraphics[width = \linewidth]{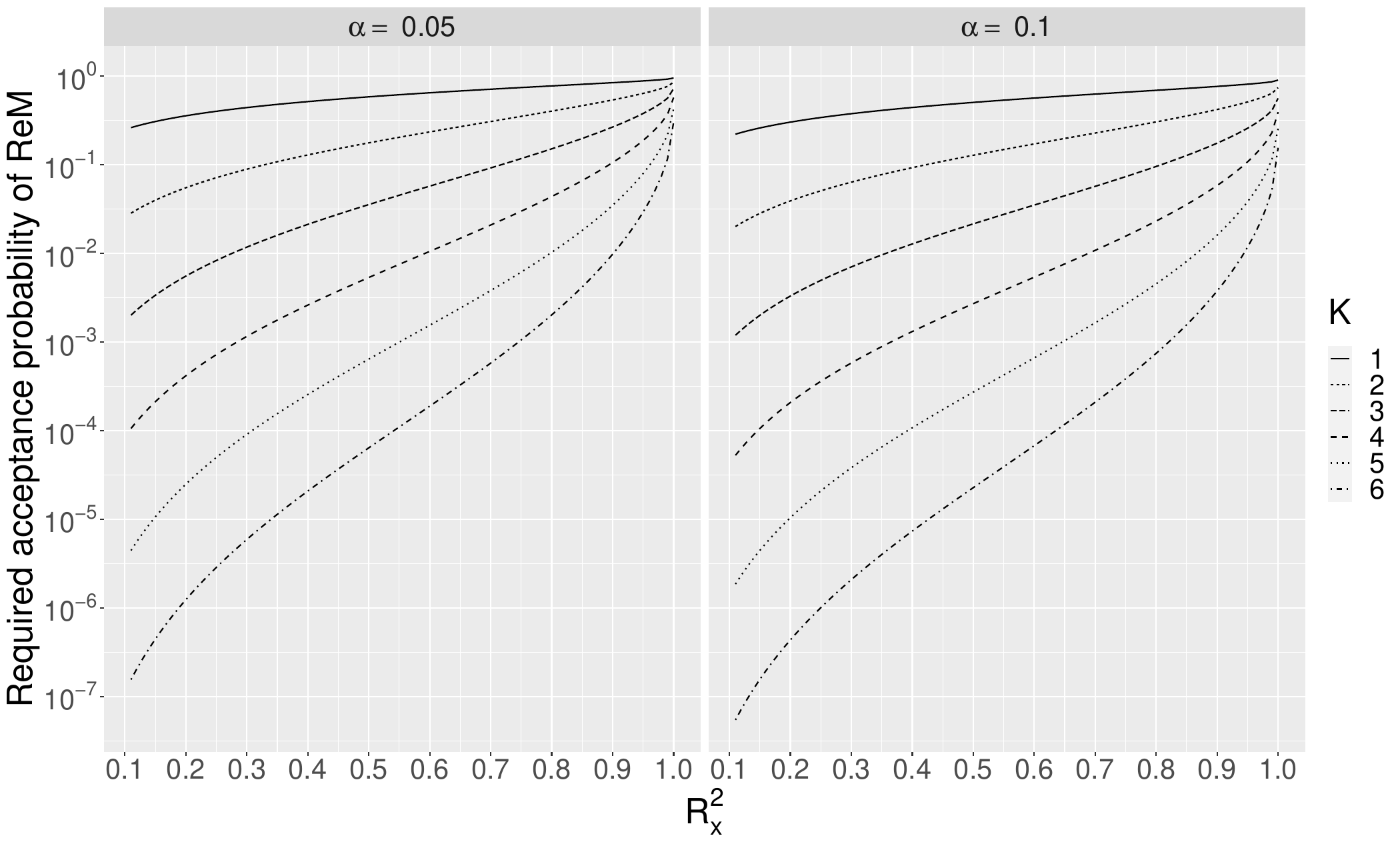}
    \caption{Required acceptance probability of ReM to resolve $p$-hacking under different $\alpha$, $R_{\bs{x}}^2$ and $K$ when \Cref{a:orthogonal-covariates-equally-importance} holds. The acceptance probability is plotted under the $\log_{10}$ transformation.}
    \label{fig:R2_prob}
\end{figure}

\begin{table}[p]
\centering
\caption{Practical guidance for selecting the threshold and analyzing ReM.
Here, $\mathcal{B}_0$ denotes a credible range for
$(R^2_{\bs{x}},\underline{\Delta})$.}
\label{tab:practical-guidance-rem}

\small
\renewcommand{\arraystretch}{1.18}
\begin{tabular}{@{}p{0.30\textwidth}p{0.65\textwidth}@{}}
\toprule
\textbf{Available information}
&
\textbf{Recommendation}
\\
\midrule

Know $R^2_{\bs{x}}$ and $\underline{\Delta}$
&
Set
$a\leq\bar{a}_{\rem}(\alpha,R^2_{\bs{x}},\underline{\Delta})$,
where $\bar{a}_{\rem}(\alpha,R^2_{\bs{x}},\underline{\Delta})$ is the
largest allowable threshold that guarantees
$\Prob_{\infty}\{\hpL\leq\alpha\mid\ma_{\rem}(a)\}\leq\alpha$
uniformly over potential outcomes compatible with the specified
information; see
\Cref{thm:iff-for-rem_control-TIE}.
\\
\addlinespace[0.6em]

Know only that
$(R^2_{\bs{x}},\underline{\Delta})\in\mathcal{B}_0$
&
Use
$a\leq\inf_{(R^2_{\bs{x}},\underline{\Delta})\in\mathcal{B}_0}
\bar{a}_{\rem}(\alpha,R^2_{\bs{x}},\underline{\Delta})$,
which guarantees control throughout $\mathcal{B}_0$.
\\
\addlinespace[0.6em]

Know $R^2_{\bs{x}}$ and assume orthogonal, equally important covariates
&
Under Condition~\ref{a:orthogonal-covariates-equally-importance},
set $\underline{\Delta}=R^2_{\bs{x}}/K$ and use
$a\leq\bar{a}_{\rem}(\alpha,R^2_{\bs{x}},R^2_{\bs{x}}/K)$.
Without additional structure, knowledge of $R^2_{\bs{x}}$ alone is
insufficient to obtain a positive uniformly valid threshold.
\\
\addlinespace[0.6em]

Have no prior information about
$R^2_{\bs{x}}$ or $\underline{\Delta}$
&
At the design stage, choose a feasible positive threshold based on balance
and the acceptance rate. At the analysis stage, use the worst-case bound in
\Cref{thm:ReM-bound-of-I-error-rate-inflation} to construct a conservative
test. See \Cref{sec:discussion} for more discussion on analysis-stage correction.
\\
\bottomrule
\end{tabular}
\end{table}

\subsection{ReP mitigates $p$-hacking}
\label{sec:rep}

Another popular way of checking the balance is to perform a two-sample $t$-test for each covariate and report all the associated $p$-values \citep{zdrep,bruhn2009pursuit,gerber2012field}. Let $p_{t,k}$ be the $p$-value of the two-sided two-sample $t$-test using EHW standard error based on data $\{x_{ik},Z_i\}_{i=1}^n$, $k\in [K]$. To balance the covariates, an intuitive and possibly widely used method is to perform rejective sampling until every $p$-value passes the balance check: $\min_{k\in [K]} p_{t,k}\geq \alpha_t$, for some prespecified threshold $\alpha_t$. Denote this rerandomization scheme by $$\ma_\rep(\alpha_t) = \left\{\min_{k\in [K]} p_{t,k}\geq \alpha_t \right\}.$$  For a matrix $\bs{V}\in \mathbb{R}^{S\times S}$, let $\bs{\sigma}(\bs{V}) = \diag(V_{ss}^{1/2})_{s=1}^S$ and $\bs{D}(\bs{V}) = \bs{\sigma}(\bs{V})^{-1}\bs{V}\bs{\sigma}(\bs{V})^{-1}$.

\Cref{thm:rep-mitigates-p-hacking} parallels \Cref{thm:rem-mitigates-p-hacking}.
\begin{theorem}
    \label{thm:rep-mitigates-p-hacking}
    Under \Cref{a:crt-regularity-condition} and $H_{0\textsc{n}}$, for any $\alpha_t \in (0,1)$ and $\alpha \in (0,1)$, we have
    \begin{itemize}
        \item[(i)]  $\pinf( \hpL \leq \alpha\mid \ma_\rep(\alpha_t)) -\pinf( \hpL \leq \alpha) \leq 0;$
        \item[(ii)]  $\lim_{\alpha_t\rightarrow 1}\Prob_{\infty}(\hpL \leq \alpha  \mid \ma_{\rep}(\alpha_t)) \leq \alpha.$ 
    \end{itemize}
\end{theorem}

\Cref{thm:rep-mitigates-p-hacking} (i) states that ReP mitigates $p$-hacking for any $\alpha_t$ in the interval of $(0, 1)$. \Cref{thm:rep-mitigates-p-hacking} (ii) states that when $\alpha_t\rightarrow 1$, ReP resolves $p$-hacking.

Define $\lVert\bs{V}\lVert_{\infty,2} = \max_{\lVert\bs{u}\lVert_\infty=1}\lVert\bs{V}\bs{u}\lVert_2$. 
ReP restricts the $\ell_\infty$ norm of $\hat{\bs{\tau}}_{\bs{x}}$ while ReM restricts the $\ell_2$ norm of $\hat{\bs{\tau}}_{\bs{x}}$. Since $\hat{\bs{\tau}}_{\bs{x}}$ in $\ma_{\rem}(\bar{a}_{\rem})$ is balanced enough to resolve $p$-hacking, we can use $\ma_{\rep}(\alpha_t)$ to resolve $p$-hacking  if $\ma_{\rep}(\alpha_t)\subseteq \ma_{\rem}(\bar{a}_{\rem})$. \Cref{cor:bound-of-rep-derived-by-direct-inequality}, which serves
both as a corollary of \Cref{thm:iff-for-rem_control-TIE} and a parallel result to it, gives such a threshold of ReP to resolve $p$-hacking depending on $(\alpha,R^2_{\bs{x}}, \underline{\Delta})$. 

\begin{corollary}
\label{cor:bound-of-rep-derived-by-direct-inequality}
 Under \Cref{a:crt-regularity-condition} and $H_{0\textsc{n}}$, for any $(R^2_{\bs{x}},\underline{\Delta})\in \mathcal{B}$, $\alpha \in (0,1)$, if $\alpha_t$ is large enough to satisfy
\begin{align}
\label{eq:rep-threshold}
    z_{1-\alpha_t/2} \leq  \lVert\bs{D}(\bs{S}_{\bs{x}\bs{x}})^{-1/2}\lVert^{-1}_{\infty,2} \Big\{\bar{a}_{\rem}(\alpha,R^2_{\bs{x}}, \underline{\Delta})\Big\}^{1/2},
\end{align}
we have
\begin{align*}    \max_{(Y_i(1),Y_i(0))_{i=1}^n\in\mathbb{M}(R^2_{\bs{x}}, \underline{\Delta})} \Prob_{\infty}(\hpL \leq \alpha  \mid \ma_{\rep}(\alpha_t))\leq \alpha.
    \end{align*}
\end{corollary}

The condition \eqref{eq:rep-threshold} is motivated by the condition that $\ma_{\rep}(\alpha_t)\subseteq \ma_{\rem}(\bar{a}_{\rem})$.
  By \Cref{cor:bound-of-rep-derived-by-direct-inequality}, similar to the discussion of \Cref{thm:iff-for-rem_control-TIE}, the threshold is more stringent with a larger $\alpha$, a smaller $\underline{\Delta}$ and a smaller $R^2_{\bs{x}}$. Therefore, our discussion for ReM under \Cref{thm:iff-for-rem_control-TIE} can apply to the setting of ReP. 

\Cref{cor:bound-of-rep-derived-by-direct-inequality} is a sufficient condition to resolve $p$-hacking, which can be overly stringent. With \Cref{a:orthogonal-covariates-equally-importance}, we can obtain a less stringent threshold. Since $\underline{\Delta} = R^2_{\bs{x}}/K$ under \Cref{a:orthogonal-covariates-equally-importance}, the set $\mathbb{M}(R^2_{\bs{x}},\underline{\Delta})$ is determined solely by $R^2_{\bs{x}}$. Hence, we denote $\mathbb{M}(R^2_{\bs{x}},\underline{\Delta})$ simply by $\mathbb{M}(R^2_{\bs{x}})$.

\begin{theorem}
\label{thm:rep-orthogonal-covariates-equally-importance}
  Let $g(\alpha) = z_{1-\alpha/2}$, for $\alpha \in (0,1)$. Under \Cref{a:crt-regularity-condition},\ref{a:orthogonal-covariates-equally-importance} and $H_{0\textsc{n}}$, for $R^2_{\bs{x}} \in (0,1)$, we have
    \begin{align*}
 \max_{(Y_i(1),Y_i(0))_{i=1}^n\in\mathbb{M}(R^2_{\bs{x}})}\Prob_{\infty}( \hpL \leq \alpha\mid \ma_{\rep}(\alpha_t)) \leq \alpha,
\end{align*}
if and only if $\alpha_t$ is large enough to satisfy $\alpha_t \geq \underline{\alpha}_{t}(R^2_{\bs{x}},\alpha)$, where
    \begin{align*}
        \underline{\alpha}_{t}(R^2_{\bs{x}},\alpha)= g^{-1}(c_{\rep}(R^2_{\bs{x}})g(\alpha)),\quad c_{\rep}(R^2_{\bs{x}}) = \frac{(R^2_{\bs{x}}/K)^{1/2}}{1+(1-R^2_{\bs{x}})^{1/2}}.
    \end{align*}
\end{theorem}

 We can use $\underline{\alpha}_{t}(R^2_{\bs{x}},\alpha)$ as the threshold of ReP if \Cref{a:orthogonal-covariates-equally-importance} holds. We can see that $c_{\rep}(R^2_{\bs{x}}) \leq 1$ and the equality holds if and only if $R^2_{\bs{x}}=1$ and $K=1$. Since $g(\cdot)$ is a decreasing function, we have
 \[
 \underline{\alpha}_{t}(R^2_{\bs{x}},\alpha) = g^{-1}(c_{\rep}(R^2_{\bs{x}})g(\alpha)) \geq g^{-1}(g(\alpha)) = \alpha.  
 \]
 Therefore, the common practice of setting $\alpha_t = \alpha$ is generally insufficient to resolve $p$-hacking. Moreover, $\underline{\alpha}_{t}(R^2_{\bs{x}},\alpha)$ is more stringent for a smaller $c_{\rep}(R^2_{\bs{x}})$.
 Since $c_\rep(R^2_{\bs{x}})$ is an increasing function of $R^2_{\bs{x}}$, the threshold is more stringent with a smaller $R^2_{\bs{x}}$, similar to ReM.

When we know neither $R^2_{\bs{x}}$ nor $\underline{\Delta}$, to estimate the type I error rate inflation under ReP, we have the following parallel result of \Cref{thm:ReM-bound-of-I-error-rate-inflation}.

\begin{theorem}
\label{thm:ReP-bound-of-I-error-rate-inflation}
Under \Cref{a:crt-regularity-condition} and $H_{0\textsc{n}}$, we have
    \begin{align*}
        &\Prob_{\infty}( \hpL \leq \alpha\mid \ma_{\rep}(\alpha_t))
        \leq  \Prob\Big\{(\varepsilon^{2}+\lVert\bs{\xi}\lVert_2^2)^{1/2} \geq z_{1-\alpha/2}~\Big |~ \lVert\bs{D}(\bs{S}_{\bs{x}\bs{x}})^{1/2} \bs{\xi}\lVert_{\infty} \leq z_{1-\alpha_t/2}\Big\},
    \end{align*}
    where $\varepsilon \in \mathbb{R}$ and $\bs{\xi} \in \mathbb{R}^K$ are independent standard Gaussian random variables.
\end{theorem}

Similar to \Cref{thm:ReM-bound-of-I-error-rate-inflation}, the upper bound in \Cref{thm:ReP-bound-of-I-error-rate-inflation} does not depend on $R^2_{\bs{x}}$ or $\underline{\Delta}$. However, it depends on the covariance structure of the covariates through $\bs{D}(\bs{S}_{\bs{x}\bs{x}})$. This is because ReP is not invariant to linear transformations of the covariates. As $\alpha_t\rightarrow 1$, the upper bound tends to $\Prob(|\varepsilon| \geq z_{1-\alpha/2})=\alpha$. We can apply this upper bound by specifying $\bs{D}(\bs{S}_{\bs{x}\bs{x}})$, $\alpha$ and $\alpha_t$. For example, with $K=5$, $\bs{D}(\bs{S}_{\bs{x}\bs{x}}) = 0.2\bs{I}_K + 0.8\bs{1}_K\bs{1}_K^\top$, $\alpha = 0.05$, and $\alpha_t$ equal to the $1\%$ upper quantile of the asymptotic distribution of $\min_{k\in [K]}p_{t,k}$,  \Cref{thm:ReP-bound-of-I-error-rate-inflation} yields an upper bound about $0.076$. This reflects an inflation of $0.026$. Similar to the discussion of \Cref{thm:ReM-bound-of-I-error-rate-inflation}, we can specify a conservative test that controls the type I error rate based on the error bound in \Cref{thm:ReP-bound-of-I-error-rate-inflation}. 

\section{Extension to stratified randomized experiments}

Stratified randomized experiments (SRE) are frequently used in the design stage to balance the discrete covariates such as gender and age \citep{imbens2015causal,ding2024first}. Rerandomization can be paired with SRE to further balance the continuous covariates. \citet{wang2021rerandomization} proved the benefits of rerandomization for improving the efficiency of ATE estimators in SRE. The benefits also diminish if we use regression adjustment in the analysis stage. In this section, we show that rerandomization also mitigates $p$-hacking in SRE.

\label{sec:stratified-CRE}
\subsection{Hacked $p$-values under stratified randomized experiments}
We now extend the results to SRE. Consider $n$ units in $H$ strata. Let $n_{ h }$ be the size of stratum $h$, $h=1,\ldots,H$. An SRE conducts independent CREs across strata. In stratum $h$, $n_{ h z}$ units are assigned to the treatment group $z$, $z=0,1$, with $n_{ h 1}+n_{ h 0} = n_{ h }$ and $r_{ h z} = n_{ h z}/n_{ h }$. Let $\ms_{ h } $ be the set of units in stratum $h$. Let $\pi_{ h } = n_{ h }/n$ be the proportion of units in stratum $h$. Let $\bar{\tau}_{ h } = n_{ h }^{-1}\sum_{i\in  \ms_{ h } }\tau_i$ be the ATE of stratum $h$. The ATE for SRE is defined as $$\bar{\tau} = \sum_{h=1}^H \pi_{ h }\bar{\tau}_{ h }.$$  The goal is to test the weak null hypothesis for the SRE: $$H_{0\textsc{n}}: \bar{\tau} = 0.$$ An unbiased estimator of $\bar{\tau}$  is $\hat{\tau} = \sum_{h=1}^H \pi_{ h }\hat{\tau}_{ h }$ with $\hat{\tau}_{ h } = n_{h1}^{-1}\sum_{i:i\in \ms_h,Z_i=1} Y_i - $ $ n_{h0}^{-1}\sum_{i:i\in \ms_h,Z_i=0} Y_i$. Let $h=1$ be the reference stratum and $\bs{S}_{i} = (I(i\in \ms_{ h })-\pi_{ h })_{h=2}^H$  be the vector of centered strata indicators, where $I(\cdot)$ is the indicator function. $\hat{\tau}$ is equal to the coefficient estimator of $Z_i$ in the regression of $Y_i$ on $Z_i$, $\bs{S}_i$, and their interaction terms \citep{Lin2013Agnostic,miratrix2013adjusting,liu2020regression}:
\begin{align}
\label{eq:formul-stra}
    \lmm(Y_i \sim 1+ Z_i+\bs{S}_i + Z_i \bs{S}_i).
\end{align}
We can use the corresponding studentized two-sided $p$-value of the coefficient estimator of $Z_i$ to test $H_{0\textsc{n}}$.

Let $[H] = \{1,\dots,H\}$. For two finite-population arrays $\{\bs{a}_i\}_{i=1}^n$, $\{\bs{b}_i\}_{i=1}^n$, for $h\in [H]$, define $\bar{\bs{a}}_{ h } = n_{ h }^{-1}\sum_{i\in \ms_{ h }} \bs{a}_i$ as the stratum-specific mean and $\bs{S}_{ h, \bs{a}\bs{b}} = (n_{ h }-1)^{-1}\sum_{i\in \ms_{ h }} (\bs{a}_i-\bar{\bs{a}}_{ h })(\bs{b}_i-\bar{\bs{b}}_{ h })^\top,$
as their stratum-specific covariance.  Assume the covariates are centered within each stratum such that $\bar{\bs{x}}_h=\bs{0}$. Assume that $\bs{S}_{h,\bs{x}\bs{x}}$, $h\in [H]$, are nonsingular. Let $\otimes$ be the Kronecker product.
When there are covariates, Lin's regression further adds covariates and their interactions with $Z_i$, $\bs{S}_i$ and $Z_i\bs{S}_i$ in the regression,
\begin{align}
\label{eq:formula-Lin-stratified}
  \lmm(  Y_i \sim 1+ Z_i+\bs{S}_i + Z_i \bs{S}_i + \bs{x}_i+ Z_i\bs{x}_i+\bs{S}_i\otimes\bs{x}_i + Z_i (\bs{S}_i\otimes \bs{x}_i)).
\end{align}

Let $\hat{\tau}_{\Lin,\mk}$, $\hat{\se}_{\Lin,\mk}$, $T_{\Lin,\mk}=\hat{\tau}_{\Lin,\mk}/\hat{\se}_{\Lin,\mk}$, and $ p_{\Lin,\mk} = 2\big\{1-\Phi(|T_{\Lin, \mk}|)\big\}$ denote, respectively, the ATE estimator, the EHW standard error, the $t$-statistic, and the two-sided $p$-value obtained through \eqref{eq:formula-Lin-stratified}, replacing $\bs{x}_i$ with $\bs{x}_{i\mk}$.
Define the hacked $p$-value as $$\hpL = \min_{\mk \subseteq [K]} p_{\Lin, \mk}.$$

Another popular ATE estimator is obtained through the fixed-effects regression model:
\begin{align}
\label{eq:formula-stra-no-interaction}
    \lmm(Y_i \sim 1+ Z_i+\bs{S}_i + \bs{x}_i).
\end{align}
The regression in \eqref{eq:formula-stra-no-interaction} excludes all interaction terms in the regression \eqref{eq:formula-Lin-stratified} to avoid overparametrization \citep{imbens2015causal,duflo2007using,angrist2014opportunity}.  When $\bs{x}_i = \emptyset$, we denote the estimated coefficient of $Z_i$ in this fixed-effects regression by $\hat{\tau}_{\omega}$.  \citet{ding2021frisch} showed
$\hat{\tau}_{\omega} = \sum_{h=1}^H \omega_{ h } \hat{\tau}_{h}$, where $ \omega_h = \pi_{ h } r_{ h 1}r_{ h 0}/\sum_{h^\prime=1}^H \pi_{h^\prime} r_{h^\prime 1}r_{h^\prime 0}$. Therefore $\hat{\tau}_{\omega}$ is an unbiased estimator of $$\bar{\tau}_{\omega} = \sum_{h=1}^H \omega_{ h }  \bar{\tau}_{ h }.$$

Hence, we use \eqref{eq:formula-stra-no-interaction} to test $$H_{0\omega}: \bar{\tau}_{\omega} = 0.$$ Let $\hat{\tau}_{\fe,\mk}$, $\hat{\se}_{\fe,\mk}$, $T_{\fe,\mk}=\hat{\tau}_{\fe,\mk}/\hat{\se}_{\fe,\mk}$, and $ p_{\fe,\mk} = 2\big\{1-\Phi(|T_{\fe,\mk}|)\big\}$ denote, respectively, the ATE estimator, the EHW standard error, the $t$-statistic, and the two-sided $p$-value obtained through \eqref{eq:formula-stra-no-interaction}, replacing $\bs{x}_i$ with $\bs{x}_{i\mk}$.
Define the hacked $p$-value under the fixed-effects regression as $$\hpfe = \min_{\mk \subseteq [K]} p_{\fe, \mk}.$$

\subsection{Rerandomization under SRE}
Since \eqref{eq:formula-Lin-stratified} is equivalent to performing Lin's regression separately within each stratum and then combining the adjusted estimators from each stratum, weighted by the stratum proportion $\pi_h$, a natural way to extend $\ma_{\rem}(a)$ and $\ma_{\rep}(\alpha_t)$ is to perform them separately within each stratum. We term these schemes stratum-specific ReM (ReM-SS) and stratum-specific ReP (ReP-SS). These two rerandomization schemes are studied in \citet{wang2021rerandomization}, \citet{johansson2022rerandomization}, and \citet{zdrep}, and are shown to reduce the variance of the unadjusted ATE estimator.

Let $\bs{V}_{ h,  \bs{x}\bs{x}} =  (r_{ h 1}r_{ h 0})^{-1}\bs{S}_{ h,  \bs{x}\bs{x}}$ be the covariance matrix of $n_h^{1/2}\hat{\bs{\tau}}_{ h, \bs{x}}$ where $\hat{\bs{\tau}}_{ h, \bs{x}} = n_{h1}^{-1}$ $\sum_{i: i\in\ms_{ h }, Z_i=1}\bs{x}_i-n_{h0}^{-1}\sum_{i:i\in\ms_{ h },Z_i=0}\bs{x}_i$. 
By the finite-population central limit theorem, $n_{ h }\hat{\bs{\tau}}_{ h,  \bs{x}}^\top\bs{V}_{ h,  \bs{x}\bs{x}}^{-1}\hat{\bs{\tau}}_{ h,  \bs{x}}$, $h\in  [H]$, converge in distribution to $\chi^2_{K}$.  We define ReM-SS as
\begin{align*}
    \ma_{\remss}(a) =\{ n_{ h }\hat{\bs{\tau}}_{ h,  \bs{x}}^\top\bs{V}_{ h,  \bs{x}\bs{x}}^{-1}\hat{\bs{\tau}}_{ h,  \bs{x}}\leq a, h\in   [H]  \}.
\end{align*}
 
Let $p_{t, hk}$ be the $p$-value of the two-sample $t$-test under stratum $h$ for covariate $x_{ik}$. We define the ReP-SS as 
 $$\ma_{\repss}(\alpha_t) = \{p_{t, hk}\geq \alpha_t,h\in  [H] , k\in [K]\}.$$   

 The published articles, which use \eqref{eq:formula-stra-no-interaction} to estimate the ATEs, commonly present balance tables showing the two-sided $p$-values obtained from $t$ tests for the coefficient estimators of $Z_i$ in the fixed-effects regression:
\begin{align}
\label{eq:balance-test-additive}
    \lmm(x_{ik} \sim 1+ Z_i+\bs{S}_i),\quad k\in [K].
\end{align}
A large $p$-value from \eqref{eq:balance-test-additive} suggests that the covariate is well balanced. A commonly used rerandomization scheme is to rerandomize until the $p$-value for each covariate passes the balance check \citep{firpo2020balancing}. We refer to this rerandomization scheme as ReP-FE, where FE represents fixed-effects regression adjustment.
 Define $p_{\fe,k}^{x}$ as the two-sided  $p$-value of the $t$-statistic for the $k$th covariate. Define ReP-FE as
  $$
  \ma_{\repfe}(\alpha_t) = \Big\{\min_{k\in [K]} p_{\fe,k}^{x} \geq \alpha_t\Big\}.
  $$

\subsection{Hacked $p$-values under rerandomization in SRE}

We consider the type I error rate of $\hpL$ and $\hpfe$ under rerandomization in SRE. Let $\sigma^2_{h,\adj} = V_{h,\tau\tau}- \bs{V}_{h,\tau \bs{x}} \bs{V}_{h,\bs{x}\bs{x}}^{-1} \bs{V}_{h,\bs{x}\tau}$ with $\bs{V}_{ h,  \bs{x}\tau}^\top = \bs{V}_{ h,  \tau\bs{x}} = n_{ h }\cov(\hat{\tau}_{h},\hat{\bs{\tau}}_{h,\bs{x}})$, $V_{h,\tau\tau} = n_{ h }\var(\hat{\tau}_{h})$ and $\bs{V}_{h,\bs{x}\bs{x}} = n_{ h }\cov(\hat{\bs{\tau}}_{h, \bs{x}})$. Let $\tilde{\se}^2_{\fe }$ be the asymptotic limit of $\hat{\se}_{\fe}^2$. We relegate the expression of $\tilde{\se}^2_{\fe }$ to the Supplementary Material, since its exact form is not required for our theory.

 Condition \ref{a:CLT-sre-for-rerandomization} in SRE extends Condition \ref{a:crt-regularity-condition}.

\begin{condition}
\label{a:CLT-sre-for-rerandomization}
 Assume that $H$ is fixed. For $h \in [H]$, $z=0,1$ (i) $\pi_h$ and $r_{h1}$ have limits in $(0,1)$; (ii)  $S_{h,Y(z)Y(z)}$,  $\bs{S}_{h,\bs{x}\bs{x}}$, $\bs{S}_{h,\bs{x} Y(z)}$, $S_{h,\tau\tau}$ have finite limits, $\sigma^2_{h,\adj}$ has a positive limit, and the limit of $\bs{S}_{h,\bs{x}\bs{x}}$ is positive-definite; (iii) $\max_{h\in[H]}\max_{i\in \ms_h} |Y_i(z)-\bar{Y}_h(z)|^2 = o(n)$, $\max_{h\in[H]}\max_{i\in \ms_h} \lVert\bs{x}_i\lVert_{\infty}^2 = o(n)$; (iv) $n\tilde{\se}^2_{\fe }$ has a positive limit. 
\end{condition}

\lx{\Cref{a:CLT-sre-for-rerandomization} (i) requires that both treatment and control groups in every stratum have large sample sizes. \Cref{a:CLT-sre-for-rerandomization} (ii) and (iii) are direct extensions of \Cref{a:crt-regularity-condition} (ii) and (iii), respectively, requiring that the corresponding conditions should hold for each stratum. \Cref{a:CLT-sre-for-rerandomization} (iv) requires that the asymptotic
variance associated with the fixed-effects estimator be nondegenerate.}

\begin{theorem}
\label{thm:re-mitigates-p-hacking-sre}
Under SRE, assume Condition \ref{a:CLT-sre-for-rerandomization} holds. (i) Theorem \ref{thm:rem-mitigates-p-hacking} holds with $\ma_{\rem}(a)$ replaced by $\ma_{\remss}(a)$. (ii) Theorem \ref{thm:rep-mitigates-p-hacking} holds with $\ma_{\rep}(\alpha_t)$ replaced by $\ma_{\repss}(\alpha_t)$; Theorem \ref{thm:rep-mitigates-p-hacking} holds with $H_{0\textsc{n}}$, $\hpL$, $\ma_{\rep}(\alpha_t)$ replaced by $H_{0\omega}$, $\hpfe$, $\ma_{\repfe}(\alpha_t)$, respectively.
\end{theorem}

Parallel to Theorem \ref{thm:rem-mitigates-p-hacking}, Theorem \ref{thm:re-mitigates-p-hacking-sre} demonstrates that the aforementioned rerandomization schemes mitigate $p$-hacking under SRE. As $a\rightarrow 0$ and $\alpha_t\rightarrow 1$, ReM-SS and ReP-SS resolve $p$-hacking for testing $H_{0\textsc{n}}$ while ReP-FE resolves $p$-hacking for testing $H_{0\omega}$. To establish Theorem \ref{thm:re-mitigates-p-hacking-sre}, we need to obtain asymptotic limits of the $t$-statistics obtained from Lin's regression and the fixed-effects regression. \citet{zdrep} has shown the consistency of the estimator of the fixed-effects regression. However, there is no current result for its asymptotic normality under the design-based inference framework.  We fill this gap in the Supplementary Material.

We study how to resolve $p$-hacking by rerandomization. Define the stratum-specific $R^2_{\bs{x}}$ for stratum $h$ as $R^2_{ h, \bs{x}} = \bs{V}_{ h, \tau \bs{x}}\bs{V}_{ h,  \bs{x}\bs{x}}^{-1}\bs{V}_{ h,  \bs{x}\tau}/V_{h,\tau\tau}$. Define $R^2_{\bs{x}} = \sum_{h=1}^H R^2_{ h, \bs{x}} \pi_{ h } V_{h,\tau\tau}/\{\sum_{h=1}^H \pi_{ h } V_{h,\tau\tau}\}$ as the $R^2$ of $\hat{\tau}$ on $(\hat{\bs{\tau}}_{h, \bs{x}})_{h\in [H]}$. 
$R^2_{\bs{x}}$ measures the extent to which $\hat{\tau}$ can be explained by $(\hat{\bs{\tau}}_{h, \bs{x}})_{h\in [H]}$. When $H=1$, $R^2_{\bs{x}}$ is equal to the $R^2$ of $\lmm(r_1^{-1}Y_i(1) + r_0^{-1}Y_i(0)\sim 1+ \bs{x}_i)$. Define
$\mathbb{M}_{\sss}(R^2_{\bs{x}})$ as the set of all possible potential outcomes given $R^2_{\bs{x}}$. We now obtain 
rerandomization thresholds for ReM-SS and ReP-SS to universally control the type I error rates for $(Y_i(1),Y_i(0))_{i=1}^n\in\mathbb{M}_{\sss}(R^2_{\bs{x}})$.

\begin{theorem}
 \label{thm:sufficient-complete-condition-for-orthogonal-equal-importance}
 Recall $g(\alpha)$ defined in \Cref{thm:rep-orthogonal-covariates-equally-importance}. Under SRE, assume Condition \ref{a:CLT-sre-for-rerandomization} holds. If Condition \ref{a:orthogonal-covariates-equally-importance} holds for every stratum, then we have, under $H_{0\textsc{n}}$, for $R^2_{\bs{x}} \in (0,1)$,
 \begin{align*}
 (i) \max_{(Y_i(1),Y_i(0))_{i=1}^n\in\mathbb{M}_{\sss}(R^2_{\bs{x}})}&\Prob_{\infty}( \hpL \leq \alpha\mid \ma_{\remss}(a))  \leq \alpha,~~~
 \text{if and only if}~ a \leq H^{-1}\bar{a}_\rem(\alpha,R^2_{\bs{x}}, R^2_{\bs{x}}/K);\\
(ii)\max_{(Y_i(1),Y_i(0))_{i=1}^n\in\mathbb{M}_{\sss}(R^2_{\bs{x}})}&\Prob_{\infty}( \hpL \leq \alpha\mid \ma_{\repss}(\alpha_t)) \leq \alpha, ~~~
 \text{if and only if}~ \alpha_t \geq g^{-1}\Big(H^{-1/2}c_{\rep}(R^2_{\bs{x}})g(\alpha)\Big).
    \end{align*}
\end{theorem}

Theorem \ref{thm:sufficient-complete-condition-for-orthogonal-equal-importance} derives the required thresholds for ReM-SS and ReP-SS to resolve $p$-hacking of Lin's regression given $R^2_{\bs{x}}$ when Condition \ref{a:orthogonal-covariates-equally-importance} holds for every stratum. Compared with the thresholds under CRE in Theorem \ref{thm:iff-for-rem_control-TIE} and \ref{thm:rep-orthogonal-covariates-equally-importance}, the thresholds required for  Lin's regression under SRE are more stringent, as shown by the factor $H$ in the expressions. When Condition \ref{a:orthogonal-covariates-equally-importance} is violated, we derive a threshold for ReM-SS, in the proof of \Cref{thm:sufficient-complete-condition-for-orthogonal-equal-importance}, in terms of $R^2_{\bs{x}}$ and $\underline{\Delta}$ (where $\underline{\Delta}$ is redefined under SRE). Following the same approach as in \Cref{cor:bound-of-rep-derived-by-direct-inequality}, we can obtain a threshold for ReP-SS in terms of $R^2_{\bs{x}}$ and $\underline{\Delta}$.
  
  It is more challenging to resolve $p$-hacking of the fixed-effects regression by rerandomization. Let $\hat{\bs{\tau}}_{\omega, \bs{x}} = \sum_{h=1}^H \omega_{ h } \hat{\tau}_{h,\bs{x}}$.  Let $\bs{V}_{\omega,\bs{x}\bs{x}} = n\cov( \hat{\bs{\tau}}_{\omega, \bs{x}})$, $\bs{V}_{\omega,\tau\bs{x}} = \bs{V}_{\omega,\bs{x}\tau}^\top = n\cov( \hat{\bs{\tau}}_{\omega, \bs{x}},\hat{\tau}_{\omega})$, ${V}_{\omega,\tau\tau} = n\var(\hat{\tau}_{\omega}).$ Define $R^2_{\omega, \bs{x}} =  \bs{V}_{\omega,\tau\bs{x}}\bs{V}_{\omega,\bs{x}\bs{x}}^{-1}\bs{V}_{\omega,\bs{x}\tau}/{V}_{\omega,\tau\tau}$ as the $R^2$ of $\hat{\tau}_{\omega}$ on $\hat{\bs{\tau}}_{\omega, \bs{x}}$. Let $\hat{\tau}_{\fe}$ be the coefficient estimator of $Z_i$ by \eqref{eq:formula-stra-no-interaction}. We show in the Supplementary Material that
\begin{align*}
     \hat{\tau}_{\fe} \approx \hat{\tau}_{\omega} - \bs{\beta}_{\fe,\bs{x}}^\top \hat{\bs{\tau}}_{\omega,\bs{x}},\quad \bs{\beta}_{\fe,\bs{x}} = \Big(\sum_{h=1}^H \pi_{ h }\bs{S}_{h, \bs{x}\bs{x}}\Big)^{-1}\Big\{\sum_{h=1}^H \pi_{ h }\big(r_{h1}\bs{S}_{h,\bs{x} Y(1)}+r_{h0}\bs{S}_{h,\bs{x} Y(0)}\big)\Big\}.
\end{align*}
Define the optimal projection coefficients as $\bs{\beta}_{\omega,\bs{x}} = \cov(\hat{\bs{\tau}}_{\omega,\bs{x}})^{-1}\cov(\hat{\bs{\tau}}_{\omega,\bs{x}},\hat{\tau}_{\omega})$. We can decompose
\[
\hat{\tau}_{\fe} \approx (\hat{\tau}_{\omega} - \bs{\beta}_{\omega,\bs{x}}^\top \hat{\bs{\tau}}_{\omega,\bs{x}}) + (\bs{\beta}_{\omega,\bs{x}} - \bs{\beta}_{\fe,\bs{x}})^\top \hat{\bs{\tau}}_{\omega,\bs{x}}.
\]
The first term and the second term are asymptotically independent. Rerandomization only affects the second term. The variance of the first term is proportional to $(1-R^2_{\omega,\bs{x}})$ while the variance of the second term depends on $\bs{\beta}_{\omega,\bs{x}}-\bs{\beta}_{\fe,\bs{x}}$.  However, since $\bs{\beta}_{\omega,\bs{x}}-\bs{\beta}_{\fe,\bs{x}}$ is typically unknown before the experiment, we lack the information to assess how rerandomization might influence $\hat{\tau}_{\fe}$, let alone its effect on the hacked $p$-value. Therefore, it is challenging to obtain a rerandomization threshold to resolve the $p$-hacking of the fixed-effects regression by rerandomization.

Theorem \ref{thm:inflation-bound-sre} gives the type I error bound, parallel to Theorem \ref{thm:ReM-bound-of-I-error-rate-inflation} and Theorem \ref{thm:ReP-bound-of-I-error-rate-inflation}.

\begin{theorem}
\label{thm:inflation-bound-sre}
Under SRE, assume Condition \ref{a:CLT-sre-for-rerandomization} holds. We have, under $H_{0\textsc{n}}$
    \begin{align*}
     (i)\quad   &\Prob_{\infty}( \hpL \leq \alpha~ | ~\ma_{\remss}(a))   \leq \Prob\Big\{\big(\varepsilon^{2}+\sum_{h=1}^H \lVert\bs{\xi}_h\lVert_2^2 \big)^{1/2} \geq z_{1-\alpha/2}~\Big | ~ \max_{h\in  [H] } \lVert\bs{\xi}_h\lVert_2^2 \leq a\Big\};\\
      (ii)\quad   &\Prob_{\infty}( \hpL \leq \alpha\mid \ma_{\repss}(\alpha_t)) \leq   \Prob \Big\{\big(\varepsilon^{2}+\sum_{h=1}^H \lVert\bs{\xi}_h\lVert_2^2 \big)^{1/2} \geq z_{1-\alpha/2} ~ \Big | \\
      &\qquad \qquad \qquad \qquad \qquad \qquad \qquad \qquad ~ \max_{h\in  [H] } \lVert\bs{D}(\bs{V}_{h,\bs{x}\bs{x}})^{1/2}\bs{\xi}_h\lVert_{\infty} \leq z_{1-\alpha_t/2}\Big\};
    \end{align*}
    and  under $H_{0\omega}$,
    \begin{align*}
             (iii)\quad    & \Prob_{\infty}( \hpfe \leq \alpha\mid \ma_{\repfe}(\alpha_t)) \leq 
 \Prob\Big\{(\varepsilon^{2}+ \lVert\bs{\xi}_0\lVert_2^2 )^{1/2} \geq z_{1-\alpha/2}~ \big | ~ \lVert\bs{D}(\bs{V}_{\omega,\bs{x}\bs{x}})^{1/2}\bs{\xi}_0\lVert_{\infty} \leq z_{1-\alpha_t/2} \Big\};
    \end{align*}
    where $\varepsilon \in \mathbb{R}$ and $\bs{\xi}_{h} \in \mathbb{R}^K$, $h=0,1,\ldots,H$ are independent standard Gaussian random variables.
\end{theorem}

 Similar to Theorem \ref{thm:ReM-bound-of-I-error-rate-inflation} and Theorem \ref{thm:ReP-bound-of-I-error-rate-inflation}, Theorem \ref{thm:inflation-bound-sre} holds for all $(Y_i(1),Y_i(0))_{i=1}^n \in \mathbb{R}^{n\times 2}$. As $a\rightarrow 0$ and $\alpha_t \rightarrow 1$, the right-hand sides of the inequalities all tend to $\Prob(|\varepsilon|>z_{1-\alpha/2}) = \alpha$, thus echoing Theorem \ref{thm:re-mitigates-p-hacking-sre}. We can use Theorem \ref{thm:inflation-bound-sre} to obtain a type I error bound for a given rerandomization threshold and specify a conservative test.

\section{Simulation study}
\label{sec:simulation}

We conduct a simulation study to evaluate  $p$-hacking with and without rerandomization. The study considers $3$ designs: CRE, ReM and ReP. For $\{a_i\}_{i=1}^n$, we define $\textnormal{Scale}(a_i) = a_i/S_{aa}^{1/2}$. The sharp null hypothesis is the most challenging case for $p$-hacking, because the $p$-values of $t$ tests under the sharp null hypothesis are generally smaller than in other cases: \citet{zhao2021covariate} 
showed that we have $\Prob_{\infty}(p_{\Lin,\mk}\leq \alpha)\leq \alpha$ for all $\mk\subseteq [K]$ under the weak null hypothesis, and $\Prob_{\infty}(p_{\Lin,\mk}\leq \alpha) = \alpha$ 
holds for all $\mk\subseteq [K]$ only under the sharp null hypothesis. Therefore, we generate the potential outcomes satisfying the sharp null hypothesis with $n=1000$:
\[
Y_i(1) = Y_i(0)  = R_{\bs{x}} \cdot \textnormal{Scale}(\bs{x}_i^\top \bs{\beta}) + (1-R_{\bs{x}}^2)^{1/2}\textnormal{Scale}(\varepsilon_i),
\]
where  $\varepsilon_i$, $i=1,\ldots,n$, are i.i.d. generated from a $t$ distribution with degrees of freedom $3$. $\bs{x}_i$ are i.i.d. generated from $\mathcal{N}(\bs{0}_5, (1-\rho)\bs{I}_5 + \rho \bs{1}_5\bs{1}_5^\top)$. We conduct ReM and ReP with an asymptotic acceptance probability of $0.01$.

Once generated, $\{Y_i(1),Y_i(0),\bs{x}_i\}_{i=1}^n$ are fixed in the simulation. We set  $r_1=0.5$, draw random assignments under CRE, ReM and ReP, and compute the hacked $p$-values for each realization of treatment assignment. Each design is repeated $B=10^5$ times to approximate the distribution of the hacked $p$-values. We consider combinations of $\rho \in \{0, 0.8\}$, $R_{\bs{x}}^2 \in \{0.1, 0.2, \ldots,0.9\}$ and $\bs{\beta} \in \{\bs{1}_5,(1,1,0.3,0.3,0.3)\}$. $\rho$ and $\bs{\beta}$ represent the correlation structure of the covariates and the variability in covariate importance. In particular, $\rho = 0$ and $\bs{\beta} = \bs{1}_5$ correspond to the setting in which \Cref{a:orthogonal-covariates-equally-importance} holds.

We focus on the empirical type I error rate of hacked $p$-values at level $\alpha=0.05$: $\sum_{j=1}^B I(p_{\textnormal{L},(j)}^{\textnormal{h}} \leq \alpha)/B$, where $p_{\textnormal{L},(j)}^{\textnormal{h}}$ is the hacked $p$-value under the $j$th treatment assignment. 

\Cref{fig:R2-error-plot-different-rho} shows the empirical type I error rates of different designs under different $R_{\bs{x}}^2$, $\rho$ and $\bs{\beta}$. The empirical type I error rates under ReP and ReM consistently remain lower than those under CRE.  By \Cref{fig:R2-error-plot-different-rho}, when $\rho = 0$, ReP and ReM perform similarly, both resolve $p$-hacking for large $R_{\bs{x}}^2$. When $\rho = 0$, it is easier to resolve $p$-hacking for $\bs{\beta} = \bs{1}_5$ compared with $\bs{\beta} = (1,1,0.3,0.3,0.3)$: for $\bs{\beta} = \bs{1}_5$, ReM and ReP both resolve $p$-hacking if $R_{\bs{x}}^2 \geq 0.7$ while for $\bs{\beta} = (1,1,0.3,0.3,0.3)$, we need $R_{\bs{x}}^2 \geq 0.9$. When $\rho = 0.8$, neither ReM nor ReP resolves $p$-hacking, and the inflation of empirical type I error rates remains substantial. These results align with \Cref{thm:iff-for-rem_control-TIE}, demonstrating that correlated covariates and varying covariate importance present greater challenges in resolving $p$-hacking due to a small $\underline{\Delta}$. 

When $\bs{\beta}=\bs{1}_5$ and $\rho = 0$, i.e., \Cref{a:orthogonal-covariates-equally-importance} holds, the performance of rerandomization is better in finite samples than in asymptotic theory. Although $\chi^2_{5,0.01} \approx \bar{a}_{\rem}(0.05,0.8,0.8/5)$ and $0.01 \approx \{1-\underline{\alpha}_t(0.8,0.05)\}^5$ suggest that an acceptance probability of $0.01$ resolves $p$-hacking if $R_{\bs{x}}^2 \geq 0.8$ for both ReP and ReM, \Cref{fig:R2-error-plot-different-rho} suggests that this acceptance probability can resolve $p$-hacking if $R_{\bs{x}}^2 \geq 0.7$. 

 By Theorem \ref{thm:iff-for-rem_control-TIE}, a larger $R^2_{\bs{x}}$ is more favorable for resolving $p$-hacking under ReP and ReM. However, as shown in \Cref{fig:R2-error-plot-different-rho}, it is also the setting where $p$-hacking is more problematic under CRE. This is because a higher value of $R^2_{\bs{x}}$ is associated with larger $\ell_2$ norms of $\bs{\beta}_{\Lin,\mk}$, $\mk \subseteq [K]$, in \eqref{eq:decompose-of-hat-tau-lin}. Larger $\bs{\beta}_{\Lin,\mk}$ lead to greater divergence among $\hat{\tau}_{\Lin,\mk}$ under CRE, thereby exacerbating the severity of $p$-hacking.  This suggests the importance of balancing informative covariates by rerandomization.

We present the theoretical bounds of the type I error rates obtained by Theorems \ref{thm:ReM-bound-of-I-error-rate-inflation} and \ref{thm:ReP-bound-of-I-error-rate-inflation} in \Cref{fig:R2-error-plot-different-rho}. By \Cref{fig:R2-error-plot-different-rho}, the empirical type I error rates of ReM and ReP consistently fall below their respective bounds, consistent with \Cref{thm:ReM-bound-of-I-error-rate-inflation} and \Cref{thm:ReP-bound-of-I-error-rate-inflation}. These bounds are conservative but informative, and therefore provide useful guidance for selecting rerandomization thresholds.

Notably, by \Cref{fig:R2-error-plot-different-rho}, ReP exhibits higher empirical type I error rates than ReM when $\rho = 0.8$, as well as a higher theoretical bound, suggesting that ReM may be better than ReP at mitigating  $p$-hacking when the covariates are correlated.

\begin{figure}
    \centering
    \includegraphics[width=\linewidth]{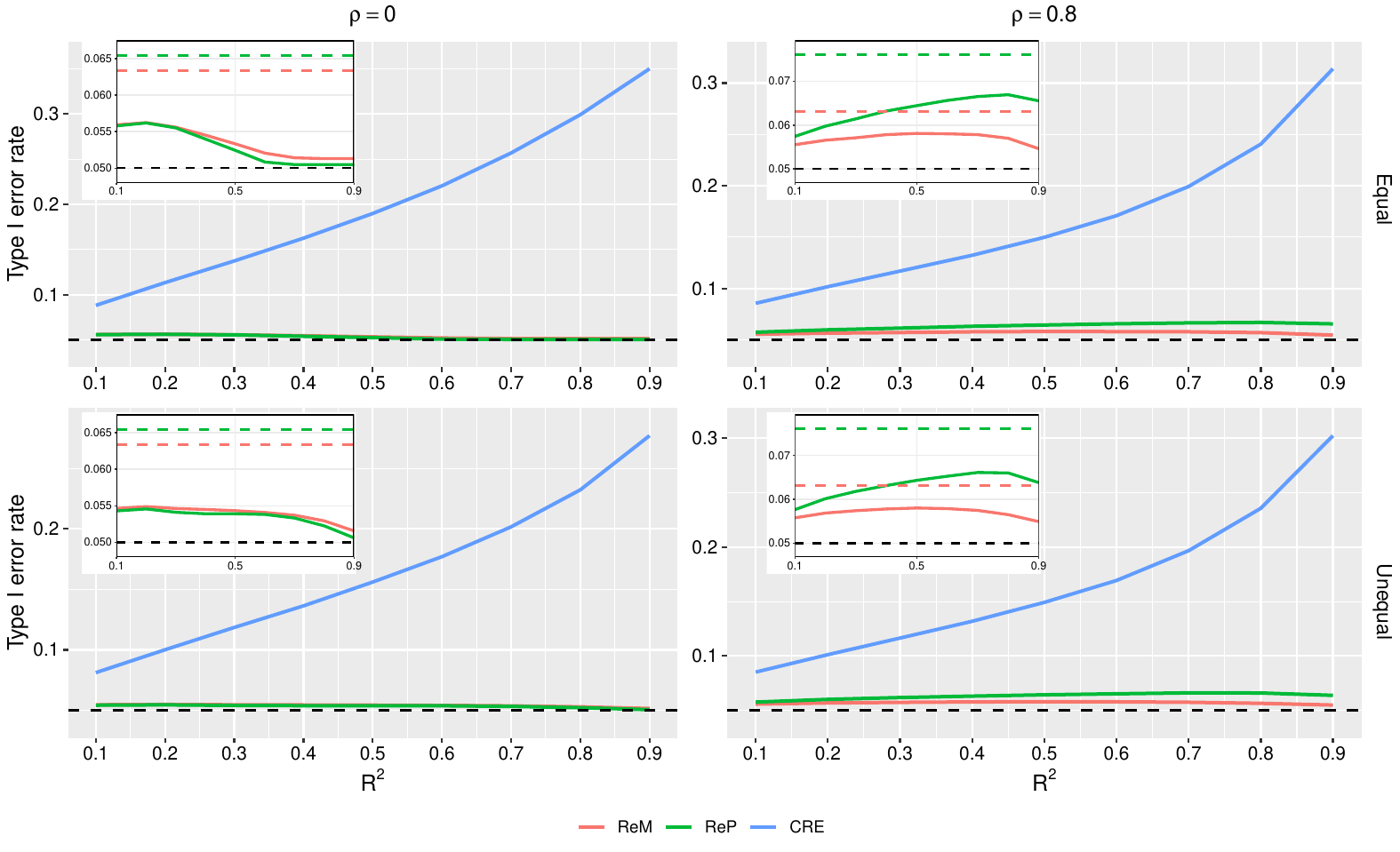}
    \caption{$R^2_{\bs{x}}$ versus empirical type I error rate computed by $10^5$ treatment assignments for
     different $\rho$ and $\bs{\beta}$ under different designs. $\bs{\beta} = \bs{1}_5$ is labelled as ``Equal'' and 
     $\bs{\beta} = (1,1,0.3,0.3,0.3)$ is labelled as ``Unequal'', respectively. The black dashed line signifies the significance level $\alpha=0.05$.
      The type I error bounds under different rerandomization schemes, obtained through \Cref{thm:ReM-bound-of-I-error-rate-inflation} and \Cref{thm:ReP-bound-of-I-error-rate-inflation}, 
      are shown as dashed lines in colors of the corresponding designs. \lx{The upper-left zoom-in highlights the region near the significance level to facilitate comparison between the ReM and
ReP error rates.}}
    \label{fig:R2-error-plot-different-rho}
\end{figure}

\subsection{\lx{Additional simulation design varying $K$ and $\underline{\Delta}$}}
To further investigate how the severity of $p$-hacking scales with
$K$ and how $\underline{\Delta}$ affects the ability of ReM to mitigate
$p$-hacking, we conduct an additional simulation study under the same sharp-null
outcome-generating mechanism as in the previous simulation, namely
\[
Y_i(1) = Y_i(0)  = R_{\bs{x}} \cdot \textnormal{Scale}(\bs{x}_i^\top \bs{\beta}) + (1-R_{\bs{x}}^2)^{1/2}\textnormal{Scale}(\varepsilon_i),
\]
where  $\varepsilon_i$, $i=1,\ldots,n$, are i.i.d. generated from a $t$ distribution with degrees of freedom $3$. $\bs{x}_i$ are i.i.d. generated from $\mathcal{N}(\bs{0}_K, (1-\rho)\bs{I}_K + \rho \bs{1}_K\bs{1}_K^\top)$. 
 We set
$n=1000$, $r_1=0.5$, $\alpha=0.05$, and
$R^2_{\bs{x}}=0.6$. For each configuration, we generate one finite
population, hold it fixed, draw $B=10^5$ treatment assignments under
CRE and ReM, and compute the empirical type I error rate of $\hpL$. We consider covariate dimension $K\in\{5,7,9,11\}$ and ReM with asymptotic
acceptance probabilities
$p_a\in\{0.10,0.05,0.025,0.01\}$.

We consider five configurations: an independent, equal-importance
baseline with $\rho=0$ and $\bs{\beta}=\bs{1}_K$; two collinearity
settings with $\rho=0.5$ and $\rho=0.8$ and
$\bs{\beta}=\bs{1}_K$; and two heterogeneous covariate-importance
settings with $\rho=0$ and coefficient vectors
\[
(\underbrace{1,\ldots,1}_{\lfloor K/2\rfloor},
\underbrace{0.3,\ldots,0.3}_{\lceil K/2\rceil})^\top
\quad\text{and}\quad
(\underbrace{1,\ldots,1}_{\lfloor K/2\rfloor},
\underbrace{0.1,\ldots,0.1}_{\lceil K/2\rceil})^\top,
\]
respectively. The collinearity and covariate-importance settings provide two distinct
ways to generate a smaller $\underline{\Delta}$.

\paragraph{$p$-hacking with a smaller $\underline{\Delta}$}
\begin{figure}
    \centering
    \includegraphics[width=\linewidth]{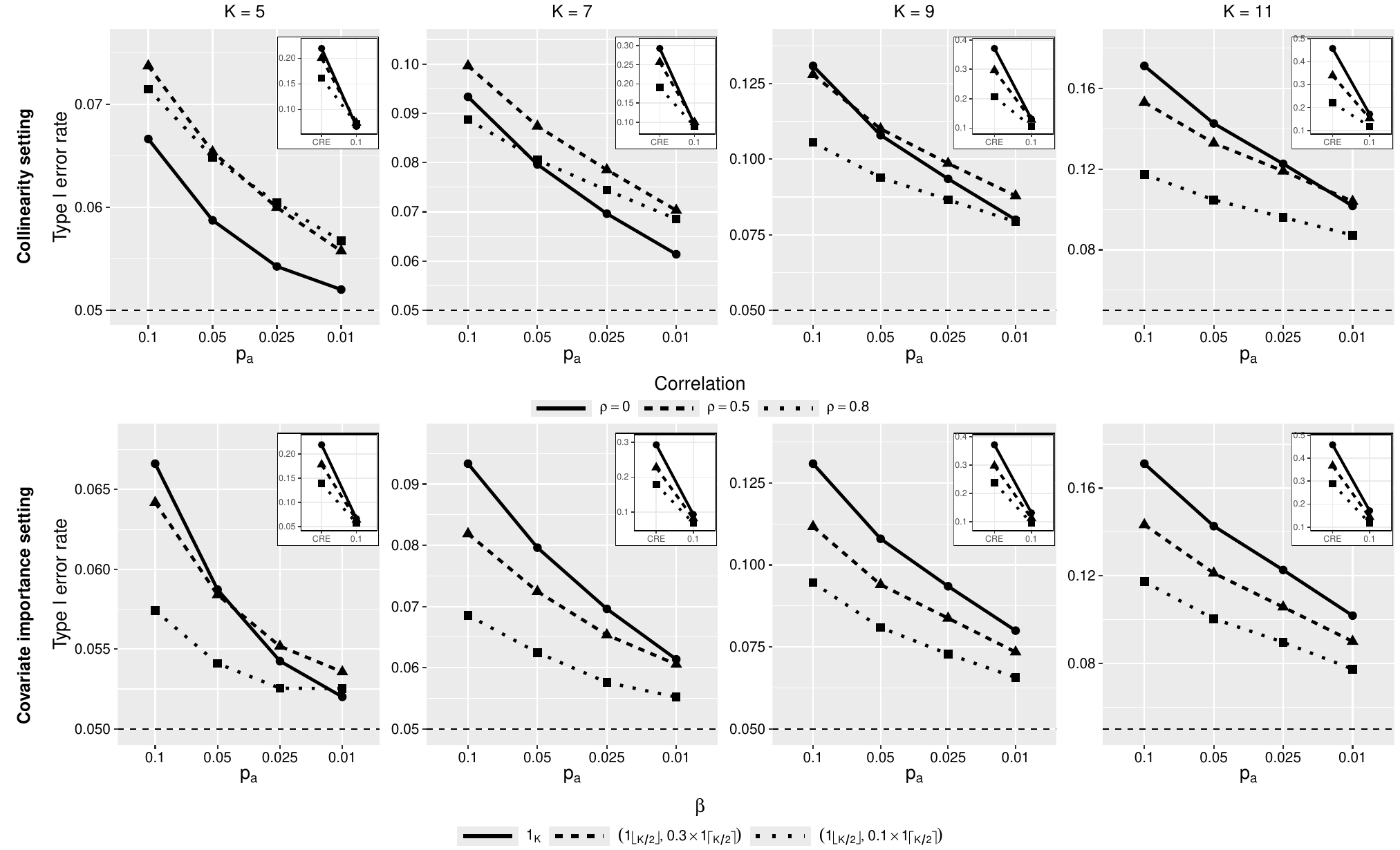}
    \caption{Empirical type I error rates of the hacked $p$-value under ReM
with different asymptotic acceptance probabilities $p_a$. Each column
corresponds to a value of $K$. Moving from left to right of the x-axis corresponds to
increasingly stringent rerandomization. The upper panels compare equally
important covariates with $\rho=0$, $\rho=0.5$, and $\rho=0.8$. The lower
panels compare independent covariates with equal importance and unequal
importance, with the weaker coefficients equal to $0.3$ and $0.1$,
respectively. The upper-right zoom-in in each panel compares CRE with ReM
at $p_a=0.1$. The horizontal dashed line indicates the significance
level $\alpha=0.05$.}
    \label{fig:decreasing_p_a}
\end{figure}

\Cref{fig:decreasing_p_a} shows the empirical type I error rates of the
hacked $p$-value as the ReM acceptance probability $p_a$ decreases. For
every value of $K$, decreasing $p_a$ reduces the type I error rate under all configurations. Thus, more stringent rerandomization more effectively mitigates
the type I error inflation.

The effect of increasingly stringent rerandomization varies across
configurations. In the configurations considered, higher covariate correlation
and the presence of weakly predictive covariates correspond to smaller values of
$\underline{\Delta}$, and the corresponding curves generally decline more
slowly as $p_a$ decreases. Moreover, the upper-right zoom-ins show that,
under CRE or at relatively large values of $p_a$, the
smaller-$\underline{\Delta}$ configurations tend to have lower type I error
rates. As $p_a$ decreases, this initial advantage narrows and, in some
configurations, reverses.

Visually, whenever both curves reach the significance-level line,
the smaller-$\underline{\Delta}$ curve would do so at a smaller value
of $p_a$ than the corresponding baseline curve. This pattern is consistent
with \Cref{thm:iff-for-rem_control-TIE}: a smaller
$\underline{\Delta}$ requires a smaller acceptance probability $p_a$ to
guarantee uniform type I error control.

A related trade-off arises for $R^2_{\bs{x}}$. A larger $R^2_{\bs{x}}$ can
lead to more severe type I error inflation under CRE, as illustrated in
\Cref{fig:R2-error-plot-different-rho}, while allowing a less stringent
ReM threshold to resolve $p$-hacking; see \Cref{fig:R2_prob}. Conversely,
a smaller $R^2_{\bs{x}}$ may yield less severe $p$-hacking but
requires more stringent rerandomization for uniform error control.

\paragraph{$p$-hacking with increasing $K$}
\begin{figure}[p]
    \centering
    \textbf{(a) Collinearity setting}\par
    \vspace{0.2em}
    \includegraphics[width=0.70\linewidth]
    {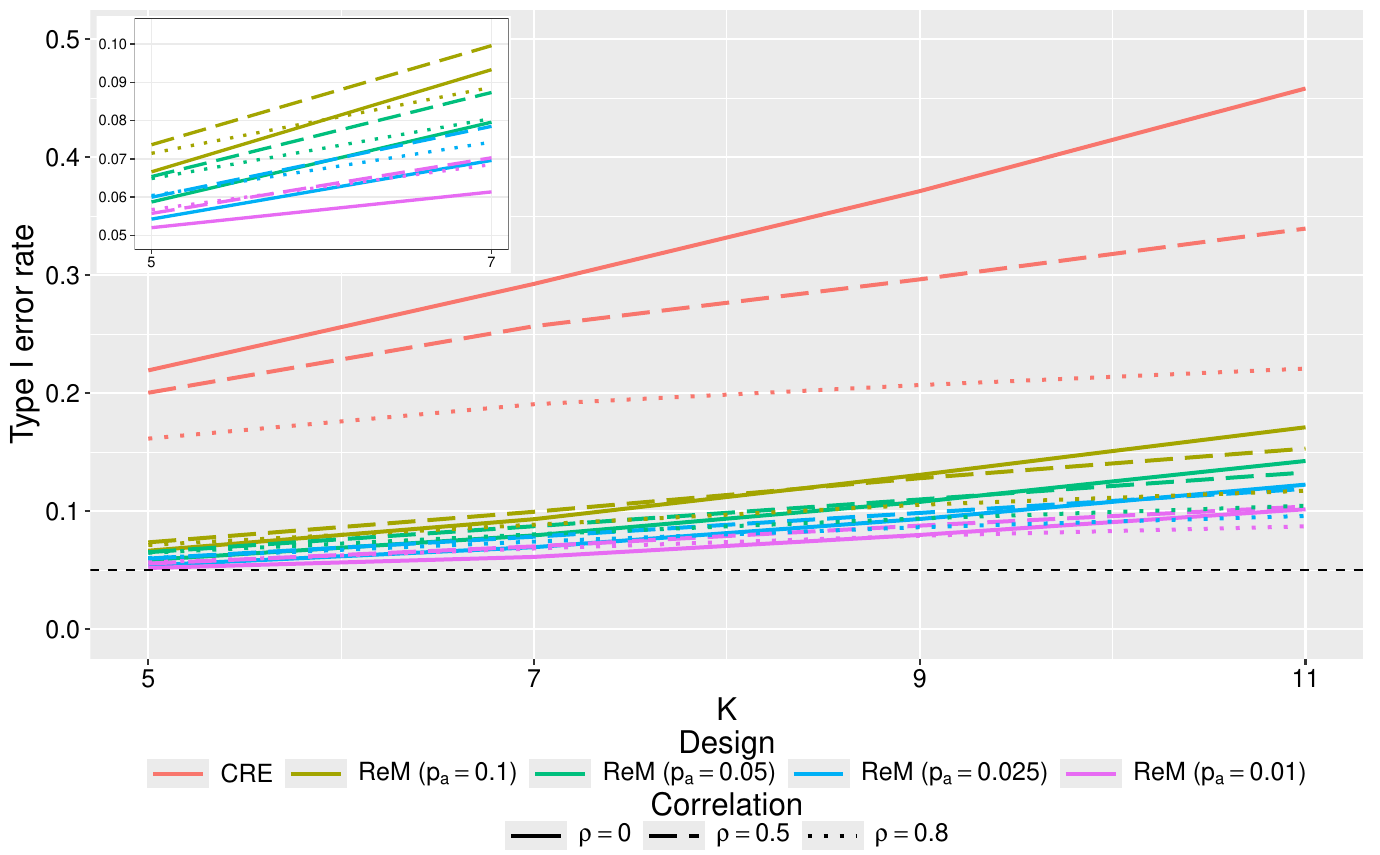}

    \vspace{0.8em}

    \textbf{(b) Covariate-importance setting}\par
    \vspace{0.2em}
    \includegraphics[width=0.70\linewidth]
    {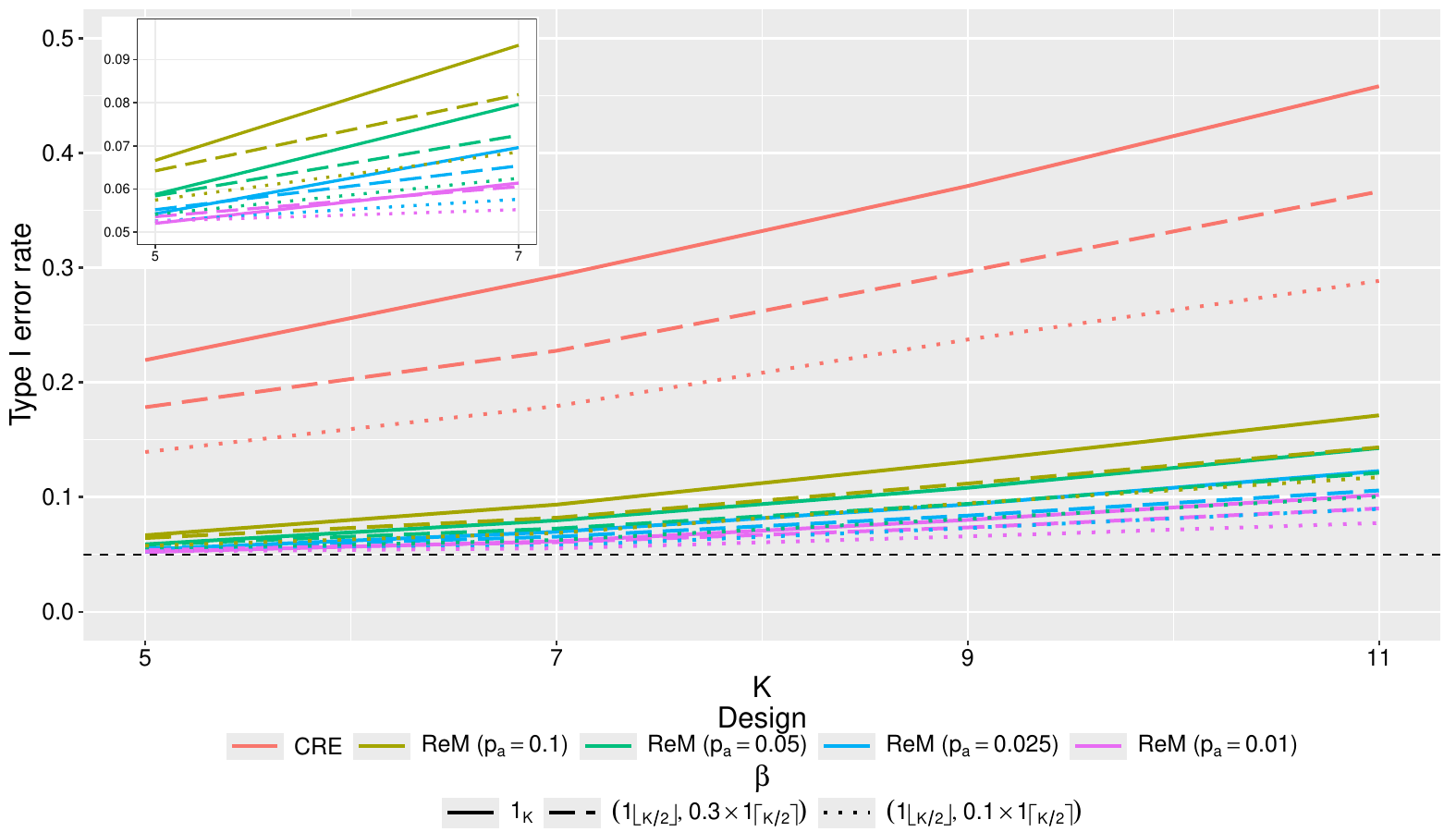}

    \caption{Empirical type I error rates of the hacked $p$-value at
significance level $\alpha=0.05$ as the number of covariates $K$ increases.
Colors distinguish CRE from ReM with different acceptance
probabilities $p_a$. In panel (a), solid, long-dashed, and dotted lines
correspond to $\rho=0$, $\rho=0.5$, and $\rho=0.8$, respectively. In panel
(b), they correspond to equal importance and unequal importance with weak
coefficients equal to $0.3$ and $0.1$, respectively. The upper-left zoom-in
focuses on the ReM results for $K=5$ and $K=7$. The horizontal dashed line
indicates the significance level.}
    \label{fig:growing-k}
\end{figure}
\Cref{fig:growing-k} displays the empirical type I error rates of the
hacked $p$-value as the number of covariates $K$ increases. Under
all configurations, the type I error rate increases sharply with $K$ in both panels.

The dashed and dotted curves, which correspond to the smaller-$\underline{\Delta}$ configurations, in both panels increase more slowly with $K$ than the
corresponding solid curves. 

\subsection{\lx{Summary of the simulation}}

Across all configurations considered, the empirical type I error rate of
the hacked $p$-value increases as the number of covariates $K$ increases
and decreases as the ReM acceptance probability $p_a$ decreases. Thus,
$p$-hacking becomes more severe as the researcher searches over more
covariate specifications, whereas more stringent rerandomization
more effectively mitigates this inflation.

The roles of $R^2_{\bs{x}}$ and $\underline{\Delta}$ involve an important
trade-off. A larger $R^2_{\bs{x}}$ leads to more severe type I error
inflation under CRE, but it also makes rerandomization more effective:
a less stringent ReM criterion can be sufficient to control the inflation.
Likewise, in the configurations considered, a smaller
$\underline{\Delta}$ is associated with a slower growth of the unadjusted
type I error rate as $K$ increases, but the type I error rate declines more
slowly as $p_a$ decreases. Consequently, a smaller
$\underline{\Delta}$ requires more stringent rerandomization to guarantee type I error control. These findings illustrate that the effects of $R^2_{\bs{x}}$ and
$\underline{\Delta}$ on the severity of $p$-hacking and on the difficulty
of mitigating it through rerandomization do not move in the same
direction.

\section{\lx{Real-data illustration}}
\label{sec:real-data}

We use a real dataset to examine whether rerandomization mitigates
$p$-hacking in a realistic finite population. The data come from the
Opportunity Knocks randomized experiment studied by
\citet{angrist2014opportunity}. The experiment evaluated a program that offered
first- and second-year college students grade-based financial awards and
access to peer advising at a Canadian commuter college. The anonymized
dataset contains 1,271 students partitioned into 16 strata
defined by sex, year of study (first or second year), and baseline grade
quartile. The dataset records students' treatment assignments, academic
outcomes, demographic characteristics, and baseline academic and
financial information.

We focus on the four strata consisting of female second-year students, with one stratum corresponding to each of the four
baseline grade quartiles. These four strata contain a total of 377
students. We use the average course grade in Fall 2008 as the outcome. The five
baseline covariates are high school grade, GPA in the previous academic
year, age, an indicator for living at home, and an indicator for having
a high level of concern about financial resources. 

We conduct a complete-case analysis, retaining observations with
nonmissing treatment status, outcome, and all five covariates. After
excluding nine observations with missing outcome or covariate data, the
final analysis sample contains 368 students. The four grade-quartile strata
contain 92, 89, 93, and 94 students, respectively, with 22, 24, 25, and
25 treated students. Thus, the stratum-specific treatment proportions
range from 0.239 to 0.270. 

To evaluate type I error under a fixed finite population, we impute the
potential outcomes under the sharp null hypothesis.
The observed outcomes and covariates are then held fixed, while treatment
assignments are repeatedly generated using the original treated counts.

We consider two implementations. First, we analyze the four strata
jointly to estimate the overall finite-population ATE, using SRE as the
benchmark design. With Lin's regression, we compare
SRE with ReM-SS and ReP-SS, for which the acceptance probability is imposed
separately within each stratum. With the fixed-effects
regression, we compare SRE with ReP-FE.
Second, we analyze the four strata separately to estimate the
stratum-specific ATEs. Within each stratum, we use CRE as the benchmark
design, and ReM and ReP as alternative designs. Because
the potential outcomes are imputed under the sharp null hypothesis,
the overall finite-population ATE and all four stratum-specific ATEs are
exactly zero by construction.

 For each design, we generate
$B=10^5$ accepted treatment assignments and report the empirical
type I error rate of the hacked $p$-value at level $0.05$.

Figure~\ref{fig:real-data} presents the results. Relative to the SRE or
CRE baseline, both ReM and ReP generally reduce the empirical type I
error rate, with more stringent acceptance probabilities typically
producing greater reductions. The magnitude of the reduction varies
across analyses and grade-quartile strata, and rerandomization does not
always eliminate the inflation completely. Nevertheless, the results
provide evidence that improving covariate balance through
rerandomization can mitigate $p$-hacking.

\begin{figure*}[p]
    \centering

    \textbf{(a) Joint analysis under SRE}\par
    \vspace{0.15em}
    \includegraphics[width=0.68\textwidth]
    {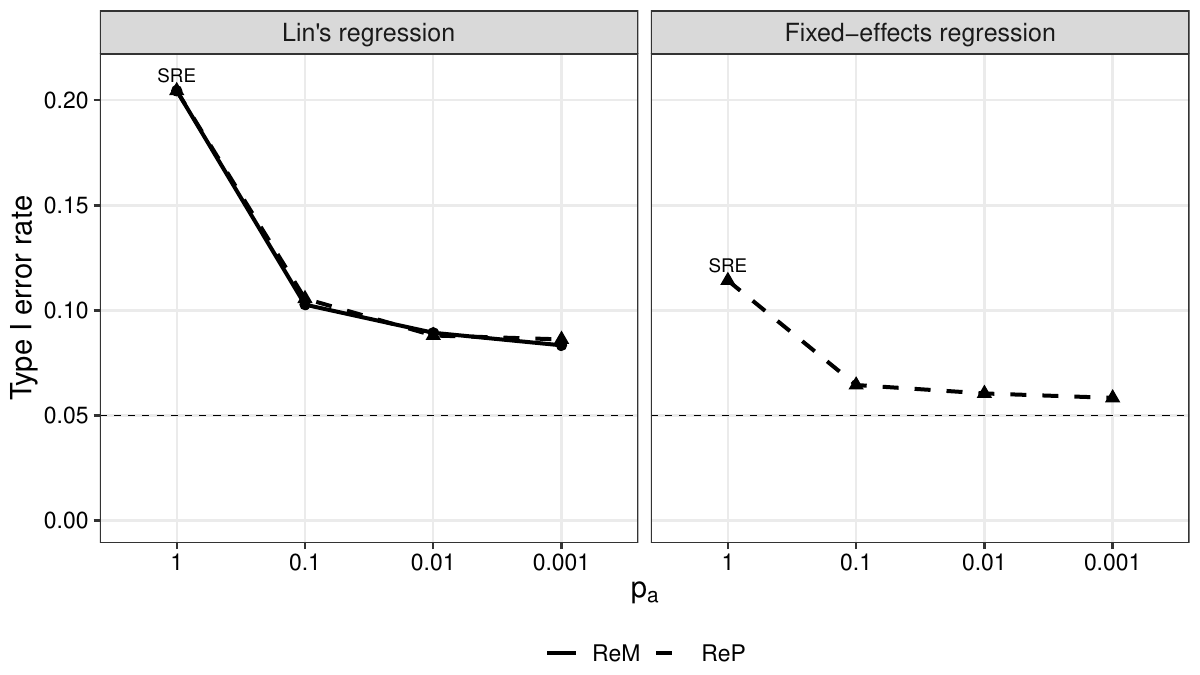}

    \vspace{0.45em}

    \textbf{(b) Separate analyses under CRE}\par
    \vspace{0.15em}
    \includegraphics[width=0.68\textwidth]
    {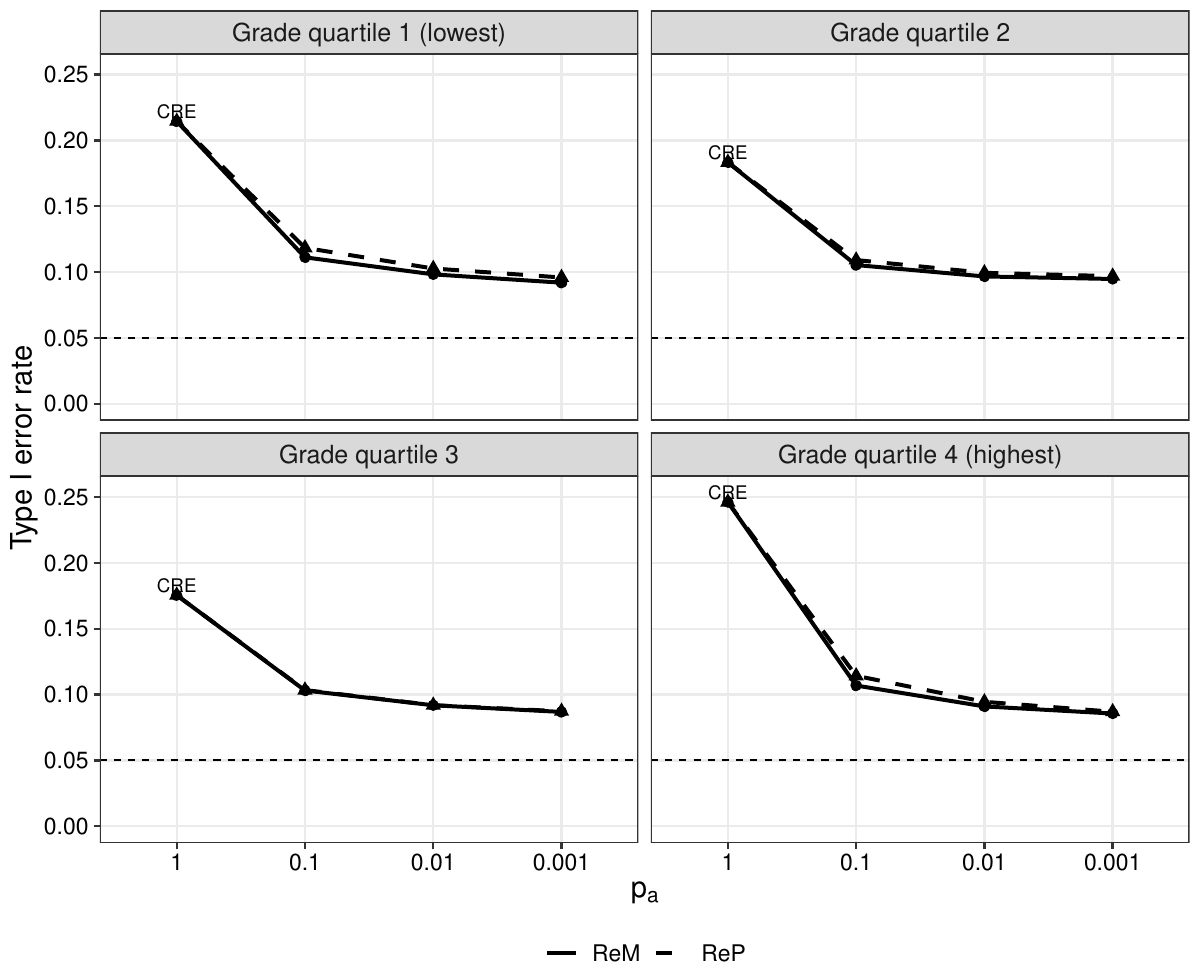}

    \caption{
    Empirical type I error rates of hacked $p$-values in the real-data
    simulation. The sample is partitioned into four strata according to
    baseline grade quartile, with quartile $1$ denoting the lowest-grade
    group. Panel (a) analyzes the four strata jointly under SRE, whereas
    Panel (b) treats each stratum as a separate CRE. ReM and ReP are
    evaluated at asymptotic acceptance probabilities
    $p_a\in\{0.1,0.01,0.001\}$. The horizontal dashed line denotes the
    significance level of $0.05$.
    }
    \label{fig:real-data}
\end{figure*}

\section{Discussion}
\label{sec:discussion}
Our theory shows that rerandomization, such as ReM and ReP, can mitigate $p$-hacking caused by strategically selecting covariates in regression adjustment for estimating the average treatment effect. Additionally, we present theorems that offer guidance on selecting rerandomization thresholds to resolve $p$-hacking depending on whether we know $R^2_{\bs{x}}$ and $\underline{\Delta}$. Compared to CRE, in SRE, we add more regressors such as stratum-treatment-covariate interactions in OLS regression. Consequently, we require a more stringent threshold to resolve $p$-hacking. We focus on ReM and ReP, and conjecture that similar results hold for other rerandomization designs \citep{morgankari2015,zhou2018sequential,johansson2020rerandomization,branson2021ridge,zhang2024pca,schindl2024unified}. We leave the technical details to future research.

Our theory demonstrates that without prior knowledge of $R^2_{\bs{x}}$ and $\underline{\Delta}$, or when the 
rerandomization threshold is not stringent enough given these quantities, rerandomization may not resolve $p$-hacking. To address this limitation, we can adjust the inference procedure. One approach is to specify a conservative test using the worst-case error bounds in Theorems~\ref{thm:ReM-bound-of-I-error-rate-inflation} and~\ref{thm:ReP-bound-of-I-error-rate-inflation}. \lx{At the analysis stage, the observed outcomes provide additional
information. Under the sharp null hypothesis, a Fisher randomization test
using $\hpL$ as the test statistic can therefore be more powerful than the
test based on the worst-case error bound. Extending this approach to ensure
validity under both the sharp and weak null hypotheses would require an
additional prepivoting procedure \citep{cohen2021prepivoting}, which in
turn requires characterizing the sampling distribution of the hacked
$p$-value. We leave this nontrivial task to future work.}

\lx{Our theory focuses on the difficulty of resolving $p$-hacking, but the
severity of $p$-hacking is also an important dimension of the problem. As
shown in our simulations, these two dimensions need not move in the same
direction. A smaller $R^2_{\bs{x}}$ or $\underline{\Delta}$ generally leads
to less severe $p$-hacking under CRE or under rerandomization with a
relatively large acceptance probability. On the other hand, resolving
$p$-hacking requires more stringent rerandomization in such settings,
because rerandomization is less effective. This trade-off can be useful in
practice when selecting a rerandomization threshold, when
complete resolution of $p$-hacking would require an impractically stringent
threshold.} 

\lx{Our theory focuses on $p$-hacking caused by strategically selecting covariates in regression adjustment.
It can accommodate $p$-hacking caused by covariate transformations if
the transformed variables are included in the covariate vector used for
rerandomization. However, it does
not cover other forms of $p$-hacking, such as model selection between linear and nonlinear specifications, data
trimming, or subgroup analyses. Studying the effects of rerandomization on
these practices, or developing design-stage approaches to mitigate the
resulting $p$-hacking, is an important direction for future research.}

% \paragraph{Acknowledgments}
% We are grateful for the technical assistance of A. Author.

% \paragraph{Data Availability Statement}
% A statement about how to access data, code and other materials allowing users to understand, verify and replicate findings --- e.g. Replication data and code can be found in Harvard Dataverse: \verb+\url{https://doi.org/link}+.

% \paragraph{Ethical Standards}
% The research meets all ethical guidelines, including adherence to the legal requirements of the study country.

%\endnote in some journals will behave like \footnote; and \printendnotes will not output anything. 
% \printendnotes

\clearpage
\bibliographystyle{agsm}
\bibliography{causal}

\clearpage
\appendix
\centerline{\Large\bf SUPPLEMENTARY MATERIAL}
\vspace{2mm}
\spacingset{1.5}

\renewcommand{\thesection}{\Alph{section}}
\setcounter{section}{0}
\renewcommand{\theequation}{\thesection.\arabic{equation}}
\makeatletter
\@addtoreset{equation}{section}
\makeatother
\setcounter{equation}{0}

\renewcommand{\theassumption}{S.\arabic{assumption}}
\setcounter{assumption}{0}
\renewcommand{\thetheorem}{S.\arabic{theorem}}
\setcounter{theorem}{0}
\renewcommand{\theproposition}{S.\arabic{proposition}}
\setcounter{proposition}{0}
\renewcommand{\thelemma}{S.\arabic{lemma}}
\setcounter{lemma}{0}
\renewcommand{\thedefinition}{S.\arabic{definition}}
\setcounter{definition}{0}
\renewcommand{\thecondition}{S.\arabic{condition}}
\setcounter{condition}{0}
\renewcommand{\thecorollary}{S.\arabic{corollary}}
\setcounter{corollary}{0}
\renewcommand{\theremark}{S.\arabic{remark}}
\setcounter{remark}{0}
\renewcommand{\thetable}{S.\arabic{table}}
\setcounter{table}{0}
\renewcommand{\thefigure}{S.\arabic{figure}}
\setcounter{figure}{0}

% Keep supplementary hyperlinks distinct from main-text theorem anchors.
\renewcommand{\theHassumption}{supp.assumption.\arabic{assumption}}
\renewcommand{\theHtheorem}{supp.theorem.\arabic{theorem}}
\renewcommand{\theHproposition}{supp.proposition.\arabic{proposition}}
\renewcommand{\theHlemma}{supp.lemma.\arabic{lemma}}
\renewcommand{\theHdefinition}{supp.definition.\arabic{definition}}
\renewcommand{\theHcondition}{supp.condition.\arabic{condition}}
\renewcommand{\theHcorollary}{supp.corollary.\arabic{corollary}}
\renewcommand{\theHremark}{supp.remark.\arabic{remark}}

\Cref{sec:hacked-p-value-limits} provides additional notation and lemmas that are useful for the proofs including joint asymptotic normality for the ATE estimators and difference in means of covariates, asymptotic limits of the regression coefficients and EHW standard errors, and asymptotic limits of hacked $p$-values under SRE with $H$ strata. The results under the CRE correspond to SRE with $H=1$. 

\Cref{sec:results-under-cre} provides the proofs of theoretical results under CRE  (Propositions \ref{prop:type I-error-rate-asymptotic}, Theorems \ref{thm:rem-mitigates-p-hacking}--\ref{thm:ReP-bound-of-I-error-rate-inflation}, and \Cref{cor:bound-of-rep-derived-by-direct-inequality}).

\Cref{sec:results-under-sre} provides the proofs of theoretical results under SRE (Theorems \ref{thm:re-mitigates-p-hacking-sre}--\ref {thm:inflation-bound-sre}).

%\spacingset{1.9} % DON'T change the spacing!
\section{Asymptotic limits of hacked $p$-values}
\label{sec:hacked-p-value-limits}

\subsection{Additional notation}

In the CRE, for observed quantities $\{\bs{a}_i\}_{i=1}^n$, $\{\bs{b}_i\}_{i=1}^n$ and treatment assignments $\{Z_i\}_{i=1}^n$, let $\bar{\bs{a}}_z =  n_z^{-1}\sum_{i:Z_i=z} \bs{a}_i$, $\bs{s}_{z,\bs{a}\bs{b}} = (n_z-1)^{-1}\sum_{i:Z_i=z}(\bs{a}_i-\bar{\bs{a}}_z)(\bs{b}_i-\bar{\bs{b}}_z)^\top$. 
In SRE, for observed quantities $\{\bs{a}_i\}_{i=1}^n$, $\{\bs{b}_i\}_{i=1}^n$ and treatment assignments $\{Z_i\}_{i=1}^n$, let $\bar{\bs{a}}_{hz} =  n_{hz}^{-1}\sum_{i:Z_i=z, i\in \ms_h} \bs{a}_i$, $\bs{s}_{hz,\bs{a}\bs{b}} = (n_{hz}-1)^{-1}\sum_{i:Z_i=z,i\in \ms_h}(\bs{a}_i-\bar{\bs{a}}_{hz})(\bs{b}_i-\bar{\bs{b}}_{hz})^\top$. To simplify the notation, when $x_{k}$ and $\bs{x}_{\mk}$ appear as the subscript, we denote them simply by $k$ and $\mk$. For example, we write $\bs{s}_{z,Y\bs{x}_{\mk}}$ and $s_{z,Y x_k}$ as $\bs{s}_{z,Y\mk}$ and $\bs{s}_{z,Y k}$, respectively. Let $\bs{V}_{h,\mk\mk}$ and $\bs{V}_{h,\mk\tau}$ be the submatrices of $\bs{V}_{h,\bs{x}\bs{x}}$ and $\bs{V}_{h,\bs{x}\tau}$ corresponding to $\mk$, respectively. Let $\bs{V}_{\omega,\mk\mk}$ and $\bs{V}_{\omega,\mk\tau}$ be the submatrices of $\bs{V}_{\omega,\bs{x}\bs{x}}$ and $\bs{V}_{\omega,\bs{x}\tau}$ corresponding to $\mk$, respectively.

For $k\in [K]$, $h\in [H]$, let $\hat{\sx}^2_{h,\Lin,k}$ be the EHW standard error of the coefficient estimator of $Z_i$ obtained from $\lm(x_{ik}\sim 1+Z_i)$, $i\in \ms_h$. Let $\hat{\sx}^2_{\fe ,k}$ be the EHW standard error of $Z_i$ obtained from $\lm(x_{ik}\sim 1+Z_i+\bs{S}_i)$. Here, we use $\hat{\sx}$ instead of $\hat{\se}$ to distinguish them from the EHW standard errors of $Z_i$ obtained from the regression $\lm(Y_i \sim 1+Z_i + x_{ik} + Z_i x_{ik})$, $i\in \ms_h$ and $\lm(Y_i \sim 1+Z_i+\bs{S}_i + x_{ik})$, respectively. \Cref{lem:ehw-of-covariate} gives the expression of $\hat{\sx}^2_{h,\Lin,k}$ and $\hat{\sx}^2_{\fe ,k}$.
 
 Let $\hat{\tau}_{\Lin}$ and $\hat{\se}_{\Lin}$ be the coefficient estimator and the EHW standard error of $Z_i$  obtained through Lin's regression with the full covariate vector, respectively. Let $\hat{\tau}_{h,\Lin}$ and $\hat{\se}_{h,\Lin}$ be the coefficient estimator and the EHW standard error of $Z_i$ obtained through $\lm(Y_i \sim 1 + Z_i + \bs{x}_i + Z_i\bs{x}_i)$, $i\in \ms_h$. Let $\hat{\tau}_{\fe }$ and $\hat{\se}_{\fe }$ be the coefficient estimator and the EHW standard error of $Z_i$  obtained through the fixed effects regression \eqref{eq:formula-stra-no-interaction} with the full covariate vector, respectively.

 Let $\tau_{e,i} = \tau_i - (\bs{\beta}_{h(i)}(1) -\bs{\beta}_{h(i)}(0))^\top \bs{x}_i$. Define the asymptotic limits of $\hat{\se}^2_{h,\Lin}$, $\hat{\se}^2_{\Lin}$ and $\hat{\se}^2_{\fe}$ as follows:
\begin{align*}
    & \tilde{\se}^2_{h,\Lin} = n_h^{-1}\sigma_{h, \adj}^2 + n_h^{-1}S_{h,\tau_e \tau_e},\quad \tilde{\se}^2_{\Lin} = \sum_{h=1}^H \pi_h^2 \tilde{\se}^2_{h,\Lin},\\
    &      \tilde{\se}^2_{\fe } = \cov(\hat{\tau}_{\omega} - \bs{\beta}^\top_{\fe, \bs{x}}\hat{\bs{\tau}}_{\omega,\bs{x}}) + \sum_{h=1}^H \frac{\omega_h^2(r_{h1}^{-1}r_{h0}^{-1}-3)}{n\pi_h}(\bar{\tau}_h-\bar{\tau}_{\omega})^2 + \sum_{h=1}^H \frac{\omega_h^2}{n\pi_h}S_{h,\tau\tau}.
\end{align*} 

Let $[n] = \{1,\ldots,n\}$. Let $\bs{0}_{S_1\times S_2} \in \mathbb{R}^{S_1\times S_2}$ and $\bs{1}_{S_1\times S_2} \in \mathbb{R}^{S_1\times S_2}$ be the matrices of all $0$'s and $1$'s, respectively. Let $\bs{0}_{S}\in \mathbb{R}^{S}$ and $\bs{1}_{S} \in \mathbb{R}^{S}$ be vectors of all $0$'s and $1$'s, respectively. When there is no ambiguity, we may omit their subscripts. For $\bs{A}_l \in \mathbb{R}^{S_l\times S_l}$, $l\in [L]$, let $\diag(\bs{A}_l)_{l\in [L]}$ and $\diag(\bs{A}_l)_{l=1}^L$ be the block diagonal matrix of $\bs{A}_l$. Let $\diag(\bs{A}_1,\bs{A}_2)$ be the block diagonal matrix of $\bs{A}_1$ and $\bs{A}_2$. For row vectors $\bs{V}_l$, $l\in [L]$, let $(\bs{V}_l)_{l\in [L]} = (\bs{V}_1,\ldots,\bs{V}_L)$ be the aggregated row vector of $\bs{V}_l$. For column vectors $\bs{U}_l$, $l\in [L]$, let $(\bs{U}_l)_{l\in [L]} = (\bs{U}_1^\top,\ldots,\bs{U}_L^\top)^\top$ be the aggregated column vector of $\bs{U}_l$. For scalars $a_l$, $l\in [L]$, let $(a_l)_{l\in [L]}$ be the column vector of $a_l$. For a matrix $\bs{A} \in \mathbb{R}^{n_r\times n_c}$ and $1\leq a\leq b \leq n_r, 1\leq c\leq d \leq n_c$, let $\bs{A}_{a:b,c:d}$ be the submatrix of $\bs{A}$ from $a$th to $b$th rows and from $c$th to $d$th columns. For a vector $\bs{v} \in \mathbb{R}^{s}$ and $1\leq a\leq b \leq s$, let $\bs{v}_{a:b}$ be the subvector of $\bs{v}$ from $a$th to $b$th entries.

\subsection{Joint asymptotic normality for 
the ATE estimators and difference in means of covariates}
 
 Let us consider (fixed) $d$-dimensional potential outcomes $\bs{R}_i(z)=(R_{i, 1}(z), $ $\ldots, R_{i, d}(z))^\top$, $i=1, \ldots, n$ and $z=0,1$. For example, $\bs{R}_i(z) \in \{Y_i(z), \bs{x}_i, (Y_i(z), \bs{x}_i^\top)^\top\}$. Define the vector-form average treatment effect: $\bar{\bs{\tau}}_{\bs{R}} = n^{-1}\sum_{i=1}^n \bs{\tau}_{\bs{R},i}$, where $\bs{\tau}_{\bs{R},i} = \bs{R}_i(1)-\bs{R}_i(0)$, and its stratified difference-in-means estimator: \[\hat{\bs{\tau}}_{\bs{R}} = \sum_{h=1}^h \pi_{h}\hat{\bs{\tau}}_{h,\bs{R}},\quad \hat{\bs{\tau}}_{h,\bs{R}} = n_{h1}^{-1}\sum_{i:i\in\ms_h, Z_i=1}\bs{R}_i(1) - n_{h0}^{-1}\sum_{i:i\in\ms_h, Z_i=0}\bs{R}_i(0). \]

\Cref{lem:covariance-of-hat-tau-vector-sre}, from the Proposition 1 of \cite{wang2021rerandomization}, gives the expression of $\cov\{$ $ n^{1 / 2}\left(\hat{\bs{\tau}}_{\bs{R}}-\bar{\bs{\tau}}_{\bs{R}}\right)\}$. 
\begin{lemma}
\label{lem:covariance-of-hat-tau-vector-sre}
We have, under SRE,
$$
\cov\{n^{1 / 2}\left(\hat{\bs{\tau}}_{\bs{R}}-\bar{\bs{\tau}}_{\bs{R}}\right)\}=\bs{\Sigma}_{\bs{R}},\quad \text{where}~\bs{\Sigma}_{\bs{R}} =\sum_{h=1}^H \pi_{h}\left\{\frac{\bs{S}_{h, \bs{R}(1)\bs{R}(1)}}{r_{h1}}+\frac{\bs{S}_{h, \bs{R}(0) \bs{R}(0)}}{r_{h0}}-\bs{S}_{h, \bs{\tau_R} \bs{\tau_R}}\right\} .
$$
\end{lemma}

Applying \Cref{lem:covariance-of-hat-tau-vector-sre} with $\bs{R}_i(z)=(Y_i(z),\bs{x}_i^\top)^\top$ for $i \in \ms_h$, we have
\begin{align*}
    n_h \cov\begin{pmatrix}
        \hat{\tau}_h -\bar{\tau}_h\\
        \hat{\bs{\tau}}_{h,\bs{x}}
    \end{pmatrix}  = \begin{pmatrix}
            V_{h,\tau\tau}& \bs{V}_{h, \tau \bs{x}}\\
            \bs{V}_{h, \bs{x} \tau} & \bs{V}_{h, \bs{x}\bs{x}}
        \end{pmatrix},
\end{align*}
where
\begin{gather*}
   V_{h,\tau\tau}  = \frac{S_{h,Y(1)Y(1)}}{r_{h1}} + \frac{S_{h,Y(0)Y(0)}}{r_{h0}}-S_{h,\tau\tau},\quad
 \bs{V}_{h,\tau \bs{x}}  =  \frac{\bs{S}_{h,Y(1)\bs{x}}}{r_{h1}} + \frac{\bs{S}_{h,Y(0)\bs{x}}}{r_{h0}},\quad \bs{V}_{h,\bs{x}\bs{x}} =  \frac{\bs{S}_{h,\bs{x}\bs{x}}}{r_{h1}r_{h0}}.
\end{gather*}

Let $h(i)$ be the stratum which $i$ belongs to. Applying \Cref{lem:covariance-of-hat-tau-vector-sre} with $\bs{R}_i(z)=$ $\omega_{h(i)}/\pi_{h(i)}(Y_i(z),\bs{x}_i^\top)^\top$, we have

\begin{align*}
    n \cov\begin{pmatrix}
        \hat{\tau}_{\omega} -\bar{\tau}_{\omega}\\
        \hat{\bs{\tau}}_{\omega,\bs{x}}
    \end{pmatrix}  = \begin{pmatrix}
            V_{\omega,\tau\tau}& \bs{V}_{\omega, \tau \bs{x}}\\
            \bs{V}_{\omega, \bs{x} \tau} & \bs{V}_{\omega, \bs{x}\bs{x}}
        \end{pmatrix},
\end{align*}
where
\begin{align*}
    V_{\omega,\tau\tau} = \sum_{h=1}^H 
    \frac{\omega_h^2}{\pi_{h}}{V}_{h,\tau\tau},\quad \bs{V}_{\omega, \tau \bs{x}} = \sum_{h=1}^H 
    \frac{\omega_h^2}{\pi_{h}}\bs{V}_{h,\tau \bs{x}},\quad
    \bs{V}_{\omega, \bs{x}\bs{x}} = \sum_{h=1}^H 
    \frac{\omega_h^2}{\pi_{h}}\bs{V}_{h,\bs{x}\bs{x}}.
\end{align*}

We require Condition \ref{a:assumption-for-CLT-sre} from \cite{wang2021rerandomization} to establish the CLT. 
\begin{condition}
\label{a:assumption-for-CLT-sre}
As $n \to \infty$, (i) $r_{h1}$, $h=1,\ldots,H$, have limits in $(0,1)$.
(ii) For $z=0,1$, $\max_{h\in [H]}$ $\max_{i\in \ms_h}\|\bs{R}_{i}(z)- \bar{\bs{R}}_{h}(z)\|_\infty^2 = o(n).$ (iii) The following three matrices have finite limits:
$$\sum_{h=1}^H \pi_{h}\frac{\bs{S}_{h, \bs{R}(1)\bs{R}(1)}}{r_{h1}},\quad \sum_{h=1}^H \pi_{h}\frac{\bs{S}_{h, \bs{R}(0)\bs{R}(0)}}{r_{h0}},\quad \sum_{h=1}^H \pi_{h}\bs{S}_{h, \bs{\tau_R} \bs{\tau_R}},$$
and the limit of $\bs{\Sigma}_{\bs{R}}$ is (strictly) positive definite.
\end{condition}
We use the same notation to denote the limits of the finite-population quantities when there is no ambiguity.

We require \Cref{lem:CLT-for-sre} from Theorem 1 of \cite{wang2021rerandomization}, to establish the CLT for $n^{1/2}(\hat{\bs{\tau}}_{\bs{R}}-\bar{\bs{\tau}}_{\bs{R}})$.
\begin{lemma}
\label{lem:CLT-for-sre}
    Under Condition \ref{a:assumption-for-CLT-sre} and SRE, $n^{1/2}(\hat{\bs{\tau}}_{\bs{R}}-\bar{\bs{\tau}}_{\bs{R}})\wconv\mathcal{N}(\bs{0},\bs{\Sigma}_{\bs{R}})$.
\end{lemma}

\begin{lemma}
\label{lem:V-omega-positive}
We have
    \[
V_{\omega,\tau\tau} - 
\bs{V}_{\omega,\tau\bs{x}}\bs{V}_{\omega,\bs{x}\bs{x}}^{-1}\bs{V}_{\omega,\bs{x}\tau} \geq  \sum_{h=1}^H \frac{\omega_h^2}{\pi_{h}} \sigma^2_{h,\adj}.
    \]
\end{lemma}
\begin{proof}[Proof of \Cref{lem:V-omega-positive}]
    Using the Cauchy--Schwarz inequality, $(\bs{b}^\top\bs{b})(\bs{d}^\top\bs{d}) \geq (\bs{b}^\top\bs{d})^2$ with 
    \begin{align*}
        \bs{b} &= (\frac{\omega_1}{\pi_{1}^{1/2}}\bs{V}_{1, \bs{x}\bs{x}}^{1/2}\bs{V}_{\omega,\bs{x}\bs{x}}^{-1}\bs{V}_{\omega,\bs{x}\tau}, \ldots, \frac{\omega_H}{\pi_{H}^{1/2}}\bs{V}_{H, \bs{x}\bs{x}}^{1/2}\bs{V}_{\omega,\bs{x}\bs{x}}^{-1}\bs{V}_{\omega,\bs{x}\tau})^\top;\\
        \bs{d} &= \Big(\frac{\omega_1}{\pi_{1}^{1/2}}\bs{V}_{1, \bs{x}\bs{x}}^{-1/2}\bs{V}_{1,\bs{x}\tau}, \ldots, \frac{\omega_H}{\pi_{H}^{1/2}}\bs{V}_{H, \bs{x}\bs{x}}^{-1/2}\bs{V}_{H,\bs{x}\tau}\Big)^\top,
    \end{align*}
we have
\begin{align*}
 \sum_{h=1}^H \frac{\omega_h^2}{\pi_{h}}\bs{V}_{h,\tau\bs{x}}\bs{V}_{h,\bs{x}\bs{x}}^{-1}\bs{V}_{h,\bs{x}\tau} \geq \bs{V}_{\omega,\tau\bs{x}}\bs{V}_{\omega,\bs{x}\bs{x}}^{-1}\bs{V}_{\omega,\bs{x}\tau}.
\end{align*}
Recall that
    $\sigma^2_{h, \adj} = V_{h,\tau\tau}- \bs{V}_{h,\tau\bs{x}}\bs{V}_{h,\bs{x}\bs{x}}^{-1}\bs{V}_{h,\bs{x}\tau}$.
Therefore, 
\begin{align*}
 V_{\omega,\tau\tau} - \bs{V}_{\omega,\tau\bs{x}}\bs{V}_{\omega,\bs{x}\bs{x}}^{-1}\bs{V}_{\omega,\bs{x}\tau} \geq  \sum_{h=1}^H \frac{\omega_h^2}{\pi_{h}} \Big(V_{h,\tau\tau} -\bs{V}_{h,\tau\bs{x}}\bs{V}_{h,\bs{x}\bs{x}}^{-1}\bs{V}_{h,\bs{x}\tau}\Big) = \sum_{h=1}^H \frac{\omega_h^2}{\pi_{h}} \sigma^2_{h,\adj}.
\end{align*}
\end{proof}

Let $(\varepsilon_{h,\tau},$ $\bs{\varepsilon}_{h,\bs{x}}^\top)^\top$, $h=1,\ldots,H$, be independent random vectors with  \begin{align*}
    \begin{pmatrix}
      \varepsilon_{h,\tau} \\
      \bs{\varepsilon}_{h,\bs{x}}
    \end{pmatrix} \sim \mathcal{N}\left(0,\begin{pmatrix}
      V_{h,\tau\tau} & \bs{V}_{h,\tau\bs{x}}\\
      \bs{V}_{h,\bs{x}\tau} & \bs{V}_{h,\bs{x}\bs{x}}
      \end{pmatrix}\right).
  \end{align*}
Let $(\varepsilon_{\omega,\tau},
      \bs{\varepsilon}_{\omega,\bs{x}})$ be random vector with
\begin{align*}
    \begin{pmatrix}
      \varepsilon_{\omega,\tau} \\
      \bs{\varepsilon}_{\omega,\bs{x}}
    \end{pmatrix} \sim \mathcal{N}\left(\bs{0},\begin{pmatrix}
      V_{\omega, \tau\tau} & \bs{V}_{\omega, \tau\bs{x}}\\
      \bs{V}_{\omega,\bs{x}\tau} & \bs{V}_{\omega,\bs{x}\bs{x}}
      \end{pmatrix}\right).
  \end{align*}

\begin{lemma}
\label{lem:CLT-for-rerandomization}
    Under Condition \ref{a:CLT-sre-for-rerandomization} (i)--(iii) and SRE, we have (i)
  $(n_1^{1/2}(\hat{\tau}_1-\bar{\tau}_1, \hat{\bs{\tau}}_{1,\bs{x}}^\top),\ldots$, $n_H^{1/2}(\hat{\tau}_H-\bar{\tau}_H, \hat{\bs{\tau}}_{H,\bs{x}}^\top))^\top \wconv ((\varepsilon_{1,\tau},\bs{\varepsilon}_{1,\bs{x}}^\top),\ldots, (\varepsilon_{H,\tau},\bs{\varepsilon}_{H,\bs{x}}^\top))^\top $ and
 (ii) $n^{1/2}(\hat{\tau}_{\omega} - \bar{\tau}_{\omega}, \hat{\bs{\tau}}_{\omega,\bs{x}}^\top)^\top \wconv (\varepsilon_{\omega,\tau},\bs{\varepsilon}_{\omega,\bs{x}}^\top) $.
\end{lemma}

\begin{proof}[Proof of \Cref{lem:CLT-for-rerandomization}]
    For Lemma $(i)$, applying \Cref{lem:CLT-for-sre} with $\bs{R}_i(z) = (Y_i(z),\bs{x}_i^\top)^\top$ for $i\in \ms_{h}$, we have, under Condition \ref{a:CLT-sre-for-rerandomization} $(i)$--$(iii)$,
 \begin{align*}
     n_h^{1/2}(\hat{\tau}_h-\bar{\tau}_h, \hat{\bs{\tau}}_{h,\bs{x}}^\top)^\top \wconv (\varepsilon_{h,\tau},\bs{\varepsilon}_{h,\bs{x}}^\top)^\top.
  \end{align*}
   Because for $h=1,\ldots,H$,  $n_h^{1/2}(\hat{\tau}_h-\bar{\tau}_h, \hat{\bs{\tau}}_{h,\bs{x}}^\top)^\top$ are independent, the conclusion follows.

   For Lemma $(ii)$, we see that under Condition \ref{a:CLT-sre-for-rerandomization} $(ii)$, \Cref{lem:V-omega-positive} implies that $V_{\omega,\tau\tau} - 
\bs{V}_{\omega,\tau\bs{x}}\bs{V}_{\omega,\bs{x}\bs{x}}^{-1}\bs{V}_{\omega,\bs{x}\tau}$ has a positive limit. 
Applying \Cref{lem:CLT-for-sre} with $\bs{R}_i(z)=\omega_{h(i)}/\pi_{h(i)}(Y_i(z),\bs{x}_i^\top)^\top$, the conclusion follows.
\end{proof}

% \Cref{lem:weak-convergence-independent-random-sequence} is from  Exercise 5  of \cite[Chapter 9]{resnick2013probability}
% \begin{lemma}
%     \label{lem:weak-convergence-independent-random-sequence}
% Suppose $X_n$ and $Y_n$ are independent for each $n\geq 0$ and $X_n \wconv X_0$, $Y_n \wconv Y_0$, we have $X_n + Y_n \wconv X_0 + Y_0$.
% \end{lemma}

\subsection{Asymptotic limits of the regression coefficients and EHW standard errors}
Let $\tilde{S}_i = (I(i\in \ms_1),\ldots, I(i\in \ms_H))$.
 
\begin{lemma}
    \label{lem:ehw-of-covariate}
    We have
    \begin{align*}
        &\hat{\sx}^2_{h,\Lin,k} =  \frac{n_1-1}{n_1^2}s_{h1,kk}  + \frac{n_0-1}{n_0^2}s_{h0,kk};  \quad \hat{\sx}^2_{\fe ,k} = \sum_{h=1}^H \omega_h^2 \Big\{\hat{\sx}^2_{h,\Lin,k} + \frac{r_{h1}^{-1}r_{h0}^{-1}-3}{n_h}(\hat{\tau}_{h,k}-\hat{\tau}_{k})^2\Big\}.
    \end{align*}
\end{lemma}
\begin{proof}[Proof of \Cref{lem:ehw-of-covariate}]
  Consider the CRE.  Let $\tilde{Z}_i = Z_i - r_1$, and $\hat{e}_i = Z_i(x_{ik}-\bar{x}_{k,1}) + (1-Z_i)(x_{ik}-\bar{x}_{k,0})$. Using FWL, we have 
    \begin{align*}
    \hat{\sx}^2_{\Lin,k} & = \frac{\sum_{i=1}^n \tilde{Z}_i^2 \hat{e}_i^2}{(\sum_{i=1}^n \tilde{Z}_i^2)^2} = \frac{r_0^2 \sum_{i:Z_i=1}(x_{ik}-\bar{x}_{k,1})^2 + r_1^2 \sum_{i:Z_i=0}(x_{ik}-\bar{x}_{k,0})^2}{(n_0 r_1^2 + n_1 r_0^2)^2}\\
    & = \frac{n_1-1}{n_1^2}s_{1,kk}  + \frac{n_0-1}{n_0^2}s_{0,kk}. 
    \end{align*}
    Thus, for stratum $h$, we have
    \[
    \hat{\sx}^2_{h,\Lin,k} =  \frac{n_{h1}-1}{n_{h1}^2}s_{h1,kk}  + \frac{n_{h0}-1}{n_{h0}^2}s_{h0,kk}.
    \]
    The equality for $\hat{\sx}^2_{\fe ,k}$ follows from Theorem $6$ of \cite{ding2021frisch} with $(Y_i(1),Y_i(0))\equiv (x_{ik}, x_{ik})$.
\end{proof}

\Cref{lem:limit-of-sample-covariance} is from Lemma A16 of \cite{Li9157}.
\begin{lemma}
\label{lem:limit-of-sample-covariance}
    Assume Condition \ref{a:CLT-sre-for-rerandomization} (i)--(iii) holds, under SRE, we have, for $z\in \{0,1\}$, $h\in [H]$,
    \begin{gather*}
         s_{hz,YY} - S_{h, Y(z)Y(z)} = \op(1),\quad \bs{s}_{hz,\bs{x}\bs{x}}- \bs{S}_{h,\bs{x}\bs{x}}  = \op(1),\quad\bs{s}_{hz,\bs{x}Y}- \bs{S}_{h,\bs{x}Y(z)}  = \op(1).
    \end{gather*}
\end{lemma}

 We will use Frisch–Waugh–
Lovell (FWL) theorems \citep{ding2021frisch} for both the regression coefficients and standard errors. We will use the invariance of OLS to compute ATE estimators and EHW standard errors.

\begin{lemma}[Invariance of OLS]
\label{lem:invariance-principle}
    For OLS regression of $\bs{Y} \in \mathbb{R}^{n}$ on $\bs{X}_1 \in \mathbb{R}^{n\times J}$, let $\hat{\bs{\beta}}_1$ be the coefficient estimator, $\hat{e}_{i,(1)}$ be the regression residual for unit $i$, $i \in [n]$, $\bs{V}_{(1)}$ be the EHW covariance matrix:
    \[
    \bs{V}_{(1)} = (\bs{X}_1^\top \bs{X}_1)^{-1}(\bs{X}_1^\top \diag(\hat{e}^2_{i,(1)})_{i=1}^n \bs{X}_1)(\bs{X}_1^\top \bs{X}_1)^{-1}.
    \]
    Let $\bs{X}_2 = \bs{X}_1\bs{P}$ where $\bs{P} \in \mathbb{R}^{J\times J}$ is an invertible matrix. Define, analogously,  $\hat{\bs{\beta}}_2$, $\hat{e}_{i,(2)}$ and $\bs{V}_{(2)}$ as the coefficient estimator, regression residual and EHW covariance matrix of OLS regression of $\bs{Y}$ on $\bs{X}_2$. We have: (i) $  \hat{\bs{\beta}}_1 =  \bs{P}\hat{\bs{\beta}}_2$; $(ii)$ for $i\in [n]$, $\hat{e}_{i,(1)} = \hat{e}_{i,(2)}$; (iii) $ \bs{V}_{(1)}  =  \bs{P}\bs{V}_{(2)} \bs{P}^\top .$
\end{lemma}
\begin{proof}[Proof of \Cref{lem:invariance-principle}]
The proof is by simple linear algebra; see \cite{ding2024linear}.
\end{proof}

\begin{lemma}
\label{lem:lin-estimator-expression}
 We have
    \[
     \hat{\se}^2_{\Lin} = \sum_{h=1}^H\pi_{h}^2\hat{\se}^2_{h,\Lin}, \quad   \hat{\tau}_{\Lin} = \sum_{h=1}^H \pi_h \hat{\tau}_{h,\Lin}.
    \]
    Moreover, we have $\hat{\tau}_{\Lin} = \hat{\tau} - \sum_{h=1}^H \pi_{h} \hat{\bs{\beta}}_{h,\Lin,\bs{x}}^\top\hat{\tau}_{h,\bs{x}}$ where $\hat{\bs{\beta}}_{h,\Lin,\bs{x}} =r_0\hat{\bs{\beta}}_{h}(1)+r_1\hat{\bs{\beta}}_{h}(0)$, $\hat{\bs{\beta}}_h(z) = \bs{s}_{hz,\bs{x}\bs{x}}^{-1} \bs{s}_{hz,\bs{x}{Y}}$.
\end{lemma}
\begin{proof}[Proof of \Cref{lem:lin-estimator-expression}]
The conclusion follows from linear algebra and \Cref{lem:invariance-principle}.
\end{proof}

\begin{lemma}
\label{lem:fe-estimator-expression}
$\hat{\tau}_{\fe } = \hat{\tau}_{\omega} -\hat{\bs{\beta}}_{\fe, \bs{x}}^\top \hat{\bs{\tau}}_{\omega,\bs{x}}$ where
    \begin{align*}
        &\hat{\bs{\beta}}_{\fe, \bs{x}} =  \Big[n^{-1}\sum_{h=1}^H \big\{ (n_{h1}-1) \bs{s}_{h1,\bs{x}\bs{x}}+(n_{h0}-1) \bs{s}_{h0,\bs{x}\bs{x}} + n_hr_{h1}r_{h0}(\hat{\bs{\tau}}_{\omega,\bs{x}} - \hat{\bs{\tau}}_{h,\bs{x}})(\hat{\bs{\tau}}_{\omega,\bs{x}} - \hat{\bs{\tau}}_{h,\bs{x}})^\top\big\}\Big]^{-1} \\
       & \cdot\Big[n^{-1}\sum_{h=1}^H \big\{ (n_{h1}-1) \bs{s}_{h1,\bs{x}Y}+(n_{h0}-1) \bs{s}_{h0,\bs{x}Y} + n_hr_{h1}r_{h0}(\hat{\bs{\tau}}_{\omega,\bs{x}} - \hat{\bs{\tau}}_{h,\bs{x}})(\hat{\tau}_{\omega} - \hat{\tau}_{h})\big\}\Big].
    \end{align*}
\end{lemma}
\begin{proof}[Proof of \Cref{lem:fe-estimator-expression}]
   By \Cref{lem:invariance-principle}, $\hat{\tau}_{\fe }$ is also equal to the coefficient estimator of $Z_i$ in 
   \begin{align}
   \label{eq:formula-main-effect-regression-equivalence-form}
    \lmm(Y_i \sim Z_i+ \tilde{\bs{S}}_i + \bs{x}_i).
\end{align}

Let $\hat{e}_i$ be the residual of unit $i$ obtained from $\lmm(Y_i \sim Z_i+ \tilde{\bs{S}}_i)$.  By the proof of \cite[][Theorem 6]{ding2021frisch}, we have the sum of squares of residual
\begin{align*}
    \sum_{i=1}^n \hat{e}_i^2 =& \sum_{h=1}^H  \Big\{\sum_{i:Z_i=1,i\in \ms_{ h }}(Y_i - \bar{Y}_{h1})^2 + \sum_{i:Z_i=0,i\in \ms_{ h }}(Y_i - \bar{Y}_{h0} )^2 \Big\} + \sum_{h=1}^H n_{h} r_{h0} r_{h1} (\hat{\tau}_{\omega}-\hat{\tau}_{h})^2.
\end{align*}

Let $\hat{\bs{\beta}}_{\fe, \bs{x}}$ be the coefficient of $\bs{x}_i$ in \eqref{eq:formula-main-effect-regression-equivalence-form}. Replacing $Y_i(z)$ by $Y_i(z)-\bs{x}_i^\top\bs{\beta}$, we view $\sum_{i=1}^n \hat{e}_i^2$ as a function of $\bs{\beta}$. $\hat{\bs{\beta}}_{\fe, \bs{x}}$ minimizes 
\begin{align*}
    \sum_{i=1}^n \hat{e}_i^2 =& \sum_{h=1}^H  \Big\{\sum_{i:Z_i=1,i\in \ms_{ h }}(Y_i-\bs{x}_i^\top \bs{\beta} - \bar{Y}_{h1} + \bar{\bs{x}}_{h1}^\top \bs{\beta})^2  +   \sum_{i:Z_i=0,i\in \ms_{ h }}(Y_i-\bs{x}_i^\top \bs{\beta} - \bar{Y}_{h0} + \bar{\bs{x}}_{h0}^\top \bs{\beta})^2\Big\} + \\
    &\qquad \qquad \qquad \qquad \sum_{h=1}^H n_{h} r_{h0} r_{h1} \big\{\hat{\tau}_{\omega}-\hat{\tau}_{h}-\bs{\beta}^\top(\hat{\bs{\tau}}_{\omega,\bs{x}}-\hat{\bs{\tau}}_{h,\bs{x}})\big\}^2.
\end{align*}

Therefore, the expression of $\hat{\bs{\beta}}_{\fe, \bs{x}}$ follows.

Substituting $Y_i(z)$ with $Y_i(z) - \hat{\bs{\beta}}_{\fe, \bs{x}}^\top \bs{x}_i$ in the expression of $\hat{\tau}_{\omega}$, we have $  \hat{\tau}_{\fe } = \hat{\tau}_{\omega} -\hat{\bs{\beta}}_{\fe, \bs{x}}^\top \hat{\bs{\tau}}_{\omega,\bs{x}}.$
\end{proof}

\begin{lemma}
\label{lem:limit-of-regression-coefficient}
    Under Condition \ref{a:CLT-sre-for-rerandomization} and under SRE, we have 
    \begin{align*}
    & \hat{\bs{\beta}}_{h}(z) - \bs{\beta}_{h}(z) = \op(1), \quad \bs{\beta}_{h}(z) = \bs{S}_{h,\bs{x}\bs{x}}^{-1}\bs{S}_{h,\bs{x}Y(z)} ,\\
  & \hat{\bs{\beta}}_{h,\Lin,\bs{x}} - \bs{\beta}_{h, \Lin , \bs{x}} = \op(1), \quad \bs{\beta}_{h, \Lin , \bs{x}} = \bs{V}_{h,\bs{x}\bs{x}}^{-1}\bs{V}_{h,\bs{x}\tau},\\
    & \hat{\bs{\beta}}_{\fe, \bs{x}} - \bs{\beta}_{\fe, \bs{x}}  = \op(1), \quad \bs{\beta}_{\fe, \bs{x}} = \Big(\sum_{h=1}^H  \pi_{ h }\bs{S}_{h,\bs{x}\bs{x}}\Big)^{-1}\Big\{\sum_{h=1}^H  \pi_{ h }\big(r_{h1}\bs{S}_{h,\bs{x}Y(1)}+r_{h0}\bs{S}_{h,\bs{x}Y(0)}\big)\Big\}.
  \end{align*}
\end{lemma}
\begin{proof}[Proof of \Cref{lem:limit-of-regression-coefficient}]
The first two lines follow  from the definitions of $\hat{\bs{\beta}}_{h}(z)$ and $\hat{\bs{\beta}}_{h,\Lin,\bs{x}}$, together with \Cref{lem:limit-of-sample-covariance}.

  It remains to prove the last line. By \Cref{lem:CLT-for-rerandomization}, we have
    \begin{align*}
\hat{\bs{\tau}}_{\omega,\bs{x}} - \hat{\bs{\tau}}_{h,\bs{x}} = \Op(n_h^{-1/2}) = \Op(n^{-1/2}),\quad \hat{\tau}_{\omega} - \hat{\tau}_{h} - (\bar{\tau}_{\omega} - \bar{\tau}_h) = \Op(n^{-1/2}).
    \end{align*}
  
Condition \ref{a:CLT-sre-for-rerandomization} $(iv)$ implies that $\bar{\tau}_{\omega} - \bar{\tau}_h = O(1)$. Therefore, we have
\begin{align*}
&\sum_{h}\pi_h r_{h1}r_{h0}(\hat{\bs{\tau}}_{\omega,\bs{x}} - \hat{\bs{\tau}}_{h,\bs{x}})(\hat{\bs{\tau}}_{\omega,\bs{x}} - \hat{\bs{\tau}}_{h,\bs{x}})^\top = \Op(n^{-1}),\\
    &\sum_{h}\pi_h r_{h1}r_{h0}(\hat{\bs{\tau}}_{\omega,\bs{x}} - \hat{\bs{\tau}}_{h,\bs{x}})(\hat{\tau}_{\omega} - \hat{\tau}_{h}) = \Op(n^{-1/2})\Op(1+n^{-1/2}) =  \Op(n^{-1/2}).
\end{align*}

    Combining this with \Cref{lem:limit-of-sample-covariance}, we complete the proof.
\end{proof}

\Cref{lem:limit-of-standard-errors} shows the asymptotic limits of the EHW standard errors.
\begin{lemma}
\label{lem:limit-of-standard-errors}
Assume Condition \ref{a:CLT-sre-for-rerandomization} holds. Under SRE, we have
\begin{align*}
    &n\hat{\se}^2_{\fe } - n\tilde{\se}^2_{\fe } = \op(1);\quad  n_h\hat{\se}^2_{h,\Lin} - n_h\tilde{\se}^2_{h,\Lin} = \op(1); \\
    & n \hat{\sx}_{\fe ,k}^2 - V_{\omega,kk} = \op(1), \quad n_h \hat{\sx}_{h,\Lin,k}^2- V_{h,kk} = \op(1), \quad k \in [K],\quad h \in [H].\\
\end{align*}
\end{lemma}

\begin{proof}[Proof of \Cref{lem:limit-of-standard-errors}]
We first prove $n\hat{\se}^2_{\fe } - n\tilde{\se}^2_{\fe } = \op(1)$. The proof proceeds in $3$ steps. In \textbf{Step 1}, we decompose  $n\hat{\se}^2_{\fe }$ into $6$ terms;  in \textbf{Step 2}, we derive the closed-form expression of related matrices; in \textbf{Step 3}, we derive the asymptotic limits for each term.

\noindent \textbf{Step 1}: We decompose  $n\hat{\se}^2_{\fe }$ into $6$ terms.

    Let $\bs{V}_i = (Z_i,\tilde{\bs{S}}_i)$, $\hat{e}_i$ be the residual of $\lm(Y_i\sim 1+ Z_i + \bs{S}_i)$ for  unit $i$, and 
\begin{align*}
    \bs{G} = \begin{pmatrix}
        \bs{G}_{11} & \bs{G}_{12}\\
        \bs{G}_{21} & \bs{G}_{22}
    \end{pmatrix}, \quad \bs{H} = \begin{pmatrix}
        \bs{H}_{11} & \bs{H}_{12}\\
        \bs{H}_{21} & \bs{H}_{22}
    \end{pmatrix}, \quad \bs{G}^{-1} = \begin{pmatrix}
        \bs{\Lambda}_{11} & \bs{\Lambda}_{12}\\
        \bs{\Lambda}_{21} & \bs{\Lambda}_{22}
    \end{pmatrix}
\end{align*}
where
\begin{align*}
    &\bs{G}_{11} =  n^{-1} \sum_{i=1}^n \bs{V}_i\bs{V}_i^\top, \quad \bs{G}_{12} = \bs{G}_{21}^\top = n^{-1} \sum_{i=1}^n \bs{V}_i \bs{x}_i^\top, \quad \bs{G}_{22} = n^{-1} \sum_{i=1}^n \bs{x}_i \bs{x}_i^\top;\\
    &\bs{H}_{11} =  n^{-1} \sum_{i=1}^n \hat{e}_i^2\bs{V}_i\bs{V}_i^\top, \quad \bs{H}_{12} = \bs{H}_{21}^\top = n^{-1} \sum_{i=1}^n \hat{e}_i^2 \bs{V}_i \bs{x}_i^\top, \quad \bs{H}_{22} = n^{-1} \sum_{i=1}^n \hat{e}_i^2 \bs{x}_i \bs{x}_i^\top.
\end{align*}

By definition of $n\hat{\se}^2_{\fe}$, we have
\[
n\hat{\se}^2_{\fe } = \bs{e}_1^\top\left(\begin{array}{cc}
      \bs{\Lambda}_{11}  &  \bs{\Lambda}_{12}
\end{array}\right)\bs{H}\left(\begin{array}{c}
      \bs{\Lambda}_{11} \\ 
      \bs{\Lambda}_{21}
    \end{array}\right)\bs{e}_1.
\]

By the formula of inverse $2\times 2$ block matrix, we have
\begin{align*}
    \bs{\Lambda}_{11} &= \bs{G}_{11}^{-1}+ \bs{G}_{11}^{-1} \bs{G}_{12} (\bs{G}_{22}-\bs{G}_{21}\bs{G}_{11}^{-1}\bs{G}_{12})^{-1}\bs{G}_{21}\bs{G}_{11}^{-1},\\
    \bs{\Lambda}_{21}^\top &= \bs{\Lambda}_{12} = -\bs{G}_{11}^{-1}\bs{G}_{12}(\bs{G}_{22}-\bs{G}_{21}\bs{G}_{11}^{-1}\bs{G}_{12})^{-1}.
\end{align*}

Let $\bs{\Sigma} = \bs{G}_{22}- \bs{G}_{21}\bs{G}_{11}^{-1}\bs{G}_{12}$ and $\bs{U}= \bs{G}_{11}^{-1}\bs{G}_{12}$. We have
\begin{align*}
    & n\hat{\se}^2_{\fe } = \bs{e}_1^\top\left(\begin{array}{cc}
      \bs{G}_{11}^{-1} + \bs{U}\bs{\Sigma}^{-1} \bs{U}^\top  &  -\bs{U}\bs{\Sigma}^{-1}
    \end{array}\right)\bs{H}\left(\begin{array}{c}
      \bs{G}_{11}^{-1} + \bs{U}\bs{\Sigma}^{-1}\bs{U}^\top  \\ 
      -\bs{\Sigma}^{-1}\bs{U}^\top
    \end{array}\right)\bs{e}_1\\
    =& \sum_{q=1}^6 T_q,
\end{align*}
where
\begin{align*}
    T_1 &= \bs{e}_1^\top  \bs{G}_{11}^{-1} \bs{H}_{11}  \bs{G}_{11}^{-1} \bs{e}_1,\quad T_2 = 2\bs{e}_1^\top  \bs{G}_{11}^{-1} \bs{H}_{11}  \bs{U}\bs{\Sigma}^{-1}\bs{U}^\top \bs{e}_1,\\
    T_3 &= \bs{e}_1^\top  \bs{U}\bs{\Sigma}^{-1}\bs{U}^\top \bs{H}_{11}  \bs{U}\bs{\Sigma}^{-1}\bs{U}^\top \bs{e}_1,\quad T_4=-2\bs{e}_1^\top  \bs{G}_{11}^{-1} \bs{H}_{12}  \bs{\Sigma}^{-1}\bs{U}^\top \bs{e}_1,\\
    T_5 &= -2\bs{e}_1^\top  \bs{U}\bs{\Sigma}^{-1}\bs{U}^\top \bs{H}_{12}  \bs{\Sigma}^{-1}\bs{U}^\top  \bs{e}_1,\quad  T_6 = \bs{e}_1^\top  \bs{U}\bs{\Sigma}^{-1} \bs{H}_{22}  \bs{\Sigma}^{-1}\bs{U}^\top  \bs{e}_1.
\end{align*}

\noindent \textbf{Step 2:} We now derive the closed-form expression of $\bs{e}_1^\top  \bs{G}_{11}^{-1}$, $\bs{\Sigma}$, $\bs{U}$ and $H$. Let $\bs{e}_{1} \in \mathbb{R}^{H+1}$ be the vector with $1$ at the first entry and $0$ at the other entries. Let $\bs{L}=  \bs{G}_{21}\bs{e}_{1}$. By simple calculation, we have $$\bs{L} = \sum_{h=1}^H \pi_h r_{h1}\bar{\bs{x}}_{h1},\quad \bs{G}_{21}  = \bs{L}\bs{e}_1^\top.$$ 
Therefore,
\begin{align*}
    &\bs{\Sigma} = n^{-1} \sum_{i=1}^n \bs{x}_i \bs{x}_i^\top - \bs{L}(\bs{G}_{11}^{-1})_{1,1}\bs{L}^\top,\quad \bs{U}= \bs{G}_{11}^{-1} \bs{e}_{1} \bs{L}^\top  .
\end{align*}

Let $c = \sum_{h=1}^H n^{-1} n_{h1} = \sum_{h=1}^H \pi_h r_{h1}$ and we have
\[
\bs{G}_{11} = \begin{pmatrix}
    c & \pi_{1} r_{11} & \cdots & \pi_{H} r_{H1} \\
    \pi_{1} r_{11} & \pi_1 & &\\
    \vdots & & \ddots & \\
    \pi_{H} r_{H1}  & & & \pi_H
\end{pmatrix}.
\]

Next, we prove
$$\bs{e}_1^\top  \bs{G}_{11}^{-1} = \Big(\sum_{h=1}^H \pi_{h} r_{h0} r_{h1}\Big)^{-1}(1,-r_{11},\ldots, -r_{H1}).$$

By linear algebra, we have
\begin{align*}
    (\bs{G}_{11}^{-1})_{1,1}  =& \det(\bs{G}_{11})^{-1}\Big(\prod_{h=1}^H \pi_{h}\Big) = 
    \Big(\sum_{h=1}^H \pi_{h} r_{h0} r_{h1} \Big)^{-1},
\end{align*}
where for the first equality, we apply the inverse of matrix formula $\bs{G}_{11}^{-1} = \operatorname{adj} (\bs{G}_{11}) /\operatorname{det}(\bs{G}_{11})$ and for the second equality, we apply 
\begin{align*}
\det\begin{pmatrix}\bs{A}&\bs{B}\\\bs{C}&\bs{D}\end{pmatrix} = \det(\bs{D})\det(\bs{A}-\bs{B}\bs{D}^{-1}\bs{C}).
\end{align*}

Using the formula of the inverse of $2\times 2$ block diagonal matrix, we have
\begin{align*}
    (\bs{G}_{11}^{-1})_{2:(H+1),1} = -\Big(c \diag(\pi_h)_{h\in [H]}-(\pi_{h}r_{h1})_{h\in [H]}(\pi_{h}r_{h1})_{h\in [H]}^\top\Big)^{-1} (\pi_{h}r_{h1})_{h\in [H]}, 
\end{align*}
where, by the Sherman-Morrison formula,
\begin{align*}
    &\Big(c \diag(\pi_h)_{h\in [H]}-(\pi_{h}r_{h1})_{h\in [H]}(\pi_{h}r_{h1})_{h\in [H]}^\top\Big)^{-1}\\
    =& c^{-1} \diag(\pi_h^{-1})_{h\in [H]} + \frac{(c^{-1}r_{h1})_{h\in [H]}(c^{-1}r_{h1})_{h\in [H]}^\top}{1-\sum_{h=1}^H (c\pi_h)^{-1}(\pi_{h}r_{h1})^2}\\
    =& c^{-1} \diag(\pi_h^{-1})_{h\in [H]} + \frac{\big(c^{-1}r_{h1}\big)_{h\in [H]}\big(c^{-1}r_{h1}\big)_{h\in [H]}^\top}{c^{-1}\sum_{h=1}^H \pi_{h} r_{h0} r_{h1}}.
\end{align*}
Thus, $$(\bs{G}_{11}^{-1})_{2:(H+1),1} = -\Big(\sum_{h=1}^H \pi_{h} r_{h0} r_{h1}\Big)^{-1}(r_{h1})_{h\in [H]}, $$
% \begin{align*}
%     & (\bs{G}_{11}^{-1})_{2:(H+1),1} 
    % = -(c^{-1} r_{h1})_{h\in [H]} -\frac{\big(r_{h1}\big)_{h\in [H]}\sum_{h=1}^H c^{-1}r_{h1}^2\pi_h}{\sum_{h=1}^H \pi_{h} r_{h0} r_{h1}}
    % \\
    % = & -(c^{-1} r_{h1})_{h\in [H]} \frac{\sum_{h=1}^H r_{h1}\pi_h}{\sum_{h=1}^H \pi_{h} r_{h0} r_{h1}}
%     = -\Big(\sum_{h=1}^H \pi_{h} r_{h0} r_{h1}\Big)^{-1}(r_{h1})_{h\in [H]}, 
% \end{align*}
and therefore $$\bs{e}_1^\top  \bs{G}_{11}^{-1} = \begin{pmatrix}
        (\bs{G}_{11}^{-1})_{1,1}\\
        (\bs{G}_{11}^{-1})_{2:(H+1),1}
    \end{pmatrix}^\top = \Big(\sum_{h=1}^H \pi_{h} r_{h0} r_{h1}\Big)^{-1}(1,-r_{11},\ldots, -r_{H1}).$$

By definition, we have
\begin{align*}
    \bs{H}_{11} &= \begin{pmatrix}
        n^{-1}\sum_{i:Z_i=1} \hat{e}_i^2& n^{-1}\sum_{i:i\in\ms_1,Z_i=1}\hat{e}_i^2 & \cdots & 
 n^{-1} \sum_{i:i\in\ms_H,Z_i=1}\hat{e}_i^2\\
        n^{-1} \sum_{i:i\in\ms_1,Z_i=1}\hat{e}_i^2& n^{-1} \sum_{i:i\in\ms_1}\hat{e}_i^2 & &\\
        \vdots & & \ddots & \\
        n^{-1} \sum_{i:i\in\ms_H,Z_i=1}\hat{e}_i^2 &&& n^{-1} \sum_{i:i\in\ms_H}\hat{e}_i^2
    \end{pmatrix} \\
    \bs{H}_{21} &= \Big(n^{-1} \sum_{i:Z_i=1} \hat{e}_i^2 \bs{x}_i,n^{-1} \sum_{i:i\in\ms_{1}} \hat{e}_i^2 \bs{x}_i, \ldots, n^{-1} \sum_{i:i\in\ms_{H}} \hat{e}_i^2 \bs{x}_i \Big), \quad \bs{H}_{22} = n^{-1} \sum_{i=1}^n \hat{e}_i^2 \bs{x}_i\bs{x}_i^\top.
\end{align*}

\noindent \textbf{Step 3}: We deal with $T_q$, $1 \leq q \leq 6$ one by one. We will prove $T_1 = n\tilde{\se}_{\fe }^2 + \op(1)$ and $T_q = \op(1)$ for $2 \leq q \leq 6$.

First, we see that
\begin{align*}
    T_1 =& \Big(\sum_{h=1}^H \pi_{h} r_{h0} r_{h1}\Big)^{-2}\Big\{ n^{-1}\sum_{i:Z_i=1} \hat{e}_i^2 -2n^{-1}\sum_{h=1}^H \sum_{i:i\in\ms_h,Z_i=1}r_{h1} \hat{e}_i^2 + n^{-1}\sum_{h=1}^H \sum_{i:i\in\ms_h}r_{h1}^2 \hat{e}_i^2\Big\} \\
    =& \Big(\sum_{h=1}^H \pi_{h} r_{h0} r_{h1}\Big)^{-2}n^{-1} \Big\{\sum_{h=1}^H \sum_{i:i\in\ms_h,Z_i=1}r_{h0}^2  \hat{e}_i^2 + \sum_{h=1}^H \sum_{i:i\in\ms_h, Z_i=0}r_{h1}^2 \hat{e}_i^2\Big\}. 
\end{align*} 

By the proof of Theorem 6 in \cite{ding2021frisch}, the residual for unit $i\in \ms_h$, obtained from $\lmm(Y_i \sim Z_i+ \tilde{\bs{S}}_i)$, equals to
\begin{align*}
\begin{cases}
        Y_i - \bar{Y}_{h1} - r_{h0}(\hat{\tau}_{\omega}-\hat{\tau}_{h}),\quad Z_i=1;\\
        Y_i - \bar{Y}_{h0} + r_{h1}(\hat{\tau}_{\omega}-\hat{\tau}_{h}), \quad Z_i=0.
    \end{cases}
\end{align*}

Replacing $Y_i$ with $Y_i - \bs{x}_i^\top \hat{\bs{\beta}}_{\fe, \bs{x}}$ in the above expression, we have, for $i \in \ms_h$,
\begin{align*}
    \hat{e}_i = \begin{cases}
        Y_i - \bar{Y}_{h1} - (\bs{x}_i - \bar{\bs{x}}_{h1})^\top\hat{\bs{\beta}}_{\fe, \bs{x}} - r_{h0}(\hat{\tau}_{\omega}-\hat{\tau}_{h}) + r_{h0}(\hat{\bs{\tau}}_{\omega, \bs{x}}-\hat{\bs{\tau}}_{h, \bs{x}})^\top \hat{\bs{\beta}}_{\fe, \bs{x}},\quad Z_i=1;\\
        Y_i - \bar{Y}_{h0} - (\bs{x}_i - \bar{\bs{x}}_{h0})^\top\hat{\bs{\beta}}_{\fe, \bs{x}} + r_{h1}(\hat{\tau}_{\omega}-\hat{\tau}_{h}) - r_{h1}(\hat{\bs{\tau}}_{\omega, \bs{x}}-\hat{\bs{\tau}}_{h, \bs{x}})^\top \hat{\bs{\beta}}_{\fe, \bs{x}}, \quad Z_i=0.
    \end{cases}
\end{align*}

Define $\hat{\epsilon}_i$, for $i\in \ms_h$
\begin{align*}
    \hat{\epsilon}_i = \begin{cases}
        Y_i - \bar{Y}_{h1} - (\bs{x}_i - \bar{\bs{x}}_{h1})^\top\hat{\bs{\beta}}_{\fe, \bs{x}} ,\quad Z_i=1;\\
        Y_i - \bar{Y}_{h0} - (\bs{x}_i - \bar{\bs{x}}_{h0})^\top\hat{\bs{\beta}}_{\fe, \bs{x}} , \quad Z_i=0.
    \end{cases}
\end{align*}

Using that $\sum_{i:i\in\ms_h,Z_i=z}  (\hat{\epsilon}_i+a_{hz})^2 = \sum_{i:i\in\ms_h,Z_i=z}  \hat{\epsilon}_i^2 + n_{hz} a_{hz}^2$ for $h\in [H]$, $z\in\{0,1\}$, with $a_{h1} = - r_{h0}(\hat{\tau}_{\omega}-\hat{\tau}_{h}) + r_{h0}(\hat{\bs{\tau}}_{\omega, \bs{x}}-\hat{\bs{\tau}}_{h, \bs{x}})^\top \hat{\bs{\beta}}_{\fe, \bs{x}}$ and $a_{h0} = r_{h1}(\hat{\tau}_{\omega}-\hat{\tau}_{h}) - r_{h1}(\hat{\bs{\tau}}_{\omega, \bs{x}}-\hat{\bs{\tau}}_{h, \bs{x}})^\top \hat{\bs{\beta}}_{\fe, \bs{x}}$, we have 
\begin{align*}
      T_1 =&  \Big(\sum_{h=1}^H \pi_{h} r_{h0} r_{h1}\Big)^{-2}n^{-1} \Big(\sum_{h=1}^H \sum_{i:i\in\ms_h,Z_i=1}r_{h0}^2  \hat{\epsilon}_i^2 + \sum_{h=1}^H \sum_{i:i\in\ms_h, Z_i=0}r_{h1}^2 \hat{\epsilon}_i^2\Big)+ \\
      & \Big(\sum_{h=1}^H \pi_{h} r_{h0} r_{h1}\Big)^{-2} \Big[\sum_{h=1}^H \pi_hr_{h1}r_{h0}(r_{h0}^3 + r_{h1}^3)  \big\{\hat{\tau}_{\omega}-\hat{\tau}_{h}-\hat{\bs{\beta}}_{\fe, \bs{x}}^\top(\hat{\bs{\tau}}_{\omega,\bs{x}}-\hat{\bs{\tau}}_{h,\bs{x}})\big\}^2\Big]\\
      & =: T_{11} + T_{12} .
\end{align*}

Using Lemmas \ref{lem:limit-of-sample-covariance} and \ref{lem:limit-of-regression-coefficient}, we have under Condition \ref{a:CLT-sre-for-rerandomization}
\begin{align*}
    &\frac{1}{n_{hz}-1}\sum_{i:i\in\ms_h,Z_i=z}  \hat{\epsilon}_i^2 = s_{hz,YY} - 2 \bs{s}_{hz,Y\bs{x}}\hat{\bs{\beta}}_{\fe, \bs{x}} + \hat{\bs{\beta}}_{\fe, \bs{x}}^\top\bs{s}_{hz,\bs{x}\bs{x}}\hat{\bs{\beta}}_{\fe, \bs{x}}  \\
     = & S_{h,Y(z)Y(z)} - 2 \bs{S}_{h,Y(z)\bs{x}}\bs{\beta}_{\fe, \bs{x}} + \bs{\beta}_{\fe, \bs{x}}^\top\bs{S}_{h,\bs{x}\bs{x}}\bs{\beta}_{\fe, \bs{x}} + \op(1) = S_{h,Y(z)-\bs{x}^\top \bs{\beta}_{\fe, \bs{x}},Y(z)-\bs{x}^\top \bs{\beta}_{\fe, \bs{x}}} + \op(1).
\end{align*}

Applying \Cref{lem:covariance-of-hat-tau-vector-sre} to stratum $h$ with $\bs{R}_i(z) = Y_i(z) - \bs{\beta}_{\fe, \bs{x}}^\top \bs{x}_i$, we have
 \[
 n_h \cov(\hat{\tau}_{h}-\bs{\beta}_{\fe, \bs{x}}^\top \hat{\bs{\tau}}_{h,\bs{x}}) = \sum_{z\in\{0,1\}}r_{hz}^{-1}S_{h,Y(z)-\bs{x}^\top \bs{\beta}_{\fe, \bs{x}},Y(z)-\bs{x}^\top \bs{\beta}_{\fe, \bs{x}}} - S_{h,\tau\tau}.
 \]

Therefore, we have
\begin{align*}
    T_{11} =& \Big(\sum_{h=1}^H \pi_{h} r_{h0} r_{h1}\Big)^{-2}  \Big\{\sum_{h=1}^H \pi_h r_{h1}^2 r_{h0}^2 \sum_{z\in\{0,1\}}r_{hz}^{-1}S_{h,Y(z)-\bs{x}^\top \bs{\beta}_{\fe, \bs{x}},Y(z)-\bs{x}^\top \bs{\beta}_{\fe, \bs{x}}}\Big\} + \op(1) \\
     =& \sum_{h=1}^H \frac{\omega_h^2}{\pi_h} \sum_{z\in\{0,1\}}r_{hz}^{-1}S_{h,Y(z)-\bs{x}^\top \bs{\beta}_{\fe, \bs{x}},Y(z)-\bs{x}^\top \bs{\beta}_{\fe, \bs{x}}} + \op(1) \\
     =& \sum_{h=1}^H \frac{\omega_h^2}{\pi_h}\Big\{n_h \cov(\hat{\tau}_{h}-\bs{\beta}_{\fe, \bs{x}}^\top \hat{\bs{\tau}}_{h,\bs{x}}) + S_{h,\tau\tau}\Big\} + \op(1)\\
     =& n\cov(\hat{\tau}_{\omega} - \bs{\beta}^\top_{\fe, \bs{x}}\hat{\bs{\tau}}_{\omega,\bs{x}}) + \sum_{h=1}^H \frac{\omega_h^2}{\pi_h} S_{h,\tau\tau} + \op(1).
\end{align*}

For $T_{12}$, we have
\begin{align*}
     (\hat{\tau}_{\omega}-\hat{\tau}_h-\hat{\bs{\beta}}_{\fe, \bs{x}}^\top(\hat{\bs{\tau}}_{\omega,\bs{x}}-\hat{\bs{\tau}}_{h,\bs{x}}))^2 
   = (\bar{\tau}_{\omega} - \bar{\tau}_{h} )^2 + \op(1).
\end{align*}

%  By Condition \ref{a:CLT-sre-for-rerandomization}, $|\bar{\tau}_{\omega} - \bar{\tau}_{h}| = O(1)$;
% \[
%  \hat{\tau}_{\omega}-\bar{\tau}_{\omega}-\hat{\tau}_h + \bar{\tau}_{h} = \Op(n_h^{-1/2}) = \Op(n^{-1/2});\quad \hat{\bs{\tau}}_{\omega,\bs{x}}-\hat{\bs{\tau}}_{h,\bs{x}} = \Op(n^{-1/2}).
% \]

% Combining with \Cref{lem:limit-of-regression-coefficient}, we have
% \[
% (\hat{\tau}_{\omega}-\hat{\tau}_h)^2 = (\bar{\tau}_{\omega} - \bar{\tau}_{h} )^2 + \Op(n^{-1}) + O(1)\cdot \Op(n^{-1/2}) = (\bar{\tau}_{\omega} - \bar{\tau}_{h} )^2  + \Op(n^{-1/2}),
% \]
% and

Therefore, we have
\begin{align*}
    T_{12} &= \Big(\sum_{h=1}^H \pi_{h} r_{h0} r_{h1}\Big)^{-2} \sum_{h=1}^H \pi_hr_{h1}r_{h0}(r_{h0}^3 + r_{h1}^3)  (\bar{\tau}_{\omega} - \bar{\tau}_{h} )^2 + \op(1)\\
    & =\sum_{h=1}^H \frac{\omega_h^2(r_{h1}^{-1}r_{h0}^{-1}-3)}{\pi_h}(\bar{\tau}_h-\bar{\tau}_{\omega})^2 + \op(1),
\end{align*}
where for the last equality, we use $(r_{h1}r_{h0})^{-1}(r_{h1}^3+r_{h0}^3) = r_{h1}^{-1}r_{h0}^{-1}-3$.

Combining the formulas of $T_{11}$ and $T_{12}$, we have
\begin{align*}
    T_1 = T_{11} + T_{12} = n\tilde{\se}_{\fe }^2 + \op(1).
\end{align*}

We prove that $T_q$, $q=2,\ldots,6$ are $\op(1)$. Using that
\begin{align*}
    \bs{L} = \Op(n^{-1/2});\quad (\bs{G}_{11}^{-1})_{1,1}= O(1),\quad \bs{U}= \bs{G}_{11}^{-1} \bs{e}_{1} \bs{L}^\top,
\end{align*}
we have
\begin{align}
    &\bs{e}_1^\top  \bs{G}_{11}^{-1} \bs{H}_{11}  \bs{U} = \bs{e}_1^\top  \bs{G}_{11}^{-1} \bs{H}_{11} \bs{G}_{11} \bs{e}_1 \bs{L}^\top = T_1 \bs{L}^\top = \Op(n^{-1/2});
    \label{eq:quantity-order-1}\\
    &\bs{U}^\top \bs{H}_{11}\bs{U} = T_1 \bs{L}^\top \bs{L} = \Op(n^{-1}), \quad \bs{U}^\top \bs{e}_1 = \bs{L} (\bs{G}_{11}^{-1})_{1,1} =  \Op(n^{-1/2}). \label{eq:quantity-order-2}
\end{align} 

On the other hand, we have
\begin{align*}
    &\bs{e}_1^\top  \bs{G}_{11}^{-1} \bs{H}_{12} = \Big(\sum_{h=1}^H \pi_{h} r_{h0} r_{h1}\Big)^{-1} n^{-1}\Big( \sum_{i:Z_i=1} \hat{e}_i^2 \bs{x}_i-\sum_{h=1}^H \sum_{i:i\in\ms_{h}} r_{h1} \hat{e}_i^2 \bs{x}_i\Big).
\end{align*}

  Using $\|\bs{x}_i\|_{\infty}^2 = o(n)$ 
from condition \ref{a:CLT-sre-for-rerandomization} $(iii)$, we have
\begin{align}
    \|\bs{e}_1^\top  \bs{G}_{11}^{-1} \bs{H}_{12}\|_{\infty} &= O\Big(n^{-1}\sum_{h=1}^H \sum_{i:i\in\ms_{h}} \hat{e}_i^2 \Big) O(\max_{i\in [n]}\|\bs{x}_i\|_{\infty}) = O(T_1)o(n^{1/2}) = \op(n^{1/2}) \label{eq:quantity-order-3}
\end{align}

and 
\begin{align}
\label{eq:quantity-order-4}
     \bs{U}^\top \bs{H}_{12} = \bs{L} \cdot \bs{e}_1^\top \bs{G}_{11}^{-1} \bs{H}_{12} = \Op(n^{-1/2})\cdot\op(n^{1/2}) = \op(1).
\end{align}

We see that 
\[
\frac{1}{n} \sum_{i=1}^n \bs{x}_i \bs{x}_i^\top = \sum_{h=1}^H \pi_h \frac{n_h-1}{n_h} \bs{S}_{h,\bs{x}\bs{x}}, 
\]
and, by Condition \ref{a:CLT-sre-for-rerandomization} $(ii)$, the limit of $n^{-1}\sum_{i=1}^n \bs{x}_i \bs{x}_i^\top$ is positive definite. Combining this with $\bs{L}(\bs{G}_{11}^{-1})_{1,1}\bs{L}^\top = \Op(n^{-1})$, we have
\begin{align}
    \label{eq:quantity-order-5}
    \bs{\Sigma}^{-1} = \Op(1).
\end{align}

Finally, we have
\begin{align}
\label{eq:quantity-order-6}
    \|\bs{H}_{22}\|_{\infty} &= O(\max_{i\in [n]}\|\bs{x}_i\|_{\infty}^2)O(n^{-1}\sum_{i:Z_i=1}\hat{e}_i^2) =  O(\max_{i\in [n]}\|\bs{x}_i\|_{\infty}^2)O(T_1)
    = o(n)\Op(1) = \op(n).
\end{align}

Putting together \eqref{eq:quantity-order-1}--\eqref{eq:quantity-order-6}, we have
\begin{align*}
    T_2& = \bs{e}_1^\top  \bs{G}_{11}^{-1} \bs{H}_{11}  \bs{U}\cdot\bs{\Sigma}^{-1} \cdot \bs{U}^\top \bs{e}_1 = \Op(n^{-1/2}\cdot 1 \cdot n^{-1/2}) = \op(1),\\
    T_3 &= \bs{e}_1^\top  \bs{U}\cdot\bs{\Sigma}^{-1}\cdot\bs{U}^\top \bs{H}_{11}  \bs{U}\cdot\bs{\Sigma}^{-1}\cdot\bs{U}^\top \bs{e}_1 = \Op(n^{-1/2}\cdot 1 \cdot n^{-1} \cdot 1 \cdot n^{-1/2}) = \op(1),\\
    T_4& =\bs{e}_1^\top  \bs{G}_{11}^{-1} \bs{H}_{12} \cdot \bs{\Sigma}^{-1}\cdot\bs{U}^\top \bs{e}_1 = \op(n^{1/2} \cdot 1 \cdot n^{-1/2}) = \op(1),\\
    T_5 &= \bs{e}_1^\top  \bs{U}\cdot\bs{\Sigma}^{-1}\cdot\bs{U}^\top \bs{H}_{12} \cdot \bs{\Sigma}^{-1} \cdot \bs{U}^\top  \bs{e}_1 = \op(n^{-1/2}\cdot 1 \cdot 1 \cdot 1 \cdot n^{-1/2}) = \op(1),\\
    T_6 &= \bs{e}_1^\top  \bs{U}\cdot\bs{\Sigma}^{-1}\cdot \bs{H}_{22} \cdot \bs{\Sigma}^{-1} \cdot \bs{U}^\top  \bs{e}_1 = \op(n^{-1/2} \cdot 1 \cdot n \cdot 1 \cdot n^{-1/2}) = \op(1).
\end{align*}

Thus 
\begin{align}
\label{eq:limit-of-fe}
    n\hat{\se}_{\fe }^2 = T_1 + \op(1) = n\tilde{\se}_{\fe }^2 + \op(1).
\end{align}

Now we prove the rest of the lemma. By \cite[][Theorem 8]{2020Rerandomization}, under Condition \ref{a:CLT-sre-for-rerandomization}, we have
\begin{align}
\label{eq:limit-of-lin}
     n_h\hat{\se}^2_{h,\Lin} - n_h\tilde{\se}^2_{h,\Lin} = \op(1).
\end{align}

Applying \eqref{eq:limit-of-fe} and \eqref{eq:limit-of-lin} with $Y_i(z) \equiv x_{ik}$ and $\bs{x}_i\equiv \emptyset$, we have
\[
n \hat{\sx}_{\fe ,k}^2 - V_{\omega,kk} = \op(1); \quad n_h \hat{\sx}_{h,\Lin,k}^2- V_{h,kk} = \op(1).
\]
\end{proof}

\subsection{Asymptotic limits of the hacked $p$-values}

We review the notation of rerandomization with the general covariate balance criterion
(ReG) from \cite{Li9157} and \cite{zdrep}.
\begin{definition}
    Let $\phi(\bs{B},\bs{C})$ be a binary covariate balance indicator function, where $(\bs{B}, \bs{C})$ are two statistics computed from the data. An ReG accepts a randomization
if $\phi(\bs{B},\bs{C})=1$.
\end{definition}

\begin{condition}
\label{a:conditon-for-indicator-function}
    The binary indicator function $\phi(\cdot,\cdot)$ satisfies: (i) $\phi(\cdot,\cdot)$ is almost surely continuous; (ii) for $\bs{u}\sim N(\bs{0}_J,\bs{V}_0)$, we have $\Prob\{\phi(\bs{u},\bs{V}_0)=1\}>0$ for all $\bs{V}_0>0$, and $\cov\{\bs{u}\mid\phi(\bs{u},\bs{V}_0)=1\}$ is a continuous function of $\bs{V}_0.$
\end{condition}

\Cref{lem:weak-convergence-under-ReG} is from \citet[][Lemma S7]{zdrep}.

\begin{lemma}[Weak convergence under ReG]
\label{lem:weak-convergence-under-ReG}
    Let $(\phi_n)_{n= 1}^\infty$ be a sequence of binary indicator functions under Condition \ref{a:conditon-for-indicator-function} that converges to $\phi.$ For a sequence of random elements $(\bs{A}_n,\bs{B}_n,\bs{C}_n)_{n=1}^\infty$ that satisfies $(\bs{A}_{n},\bs{B}_{n},\bs{C}_{n})\wconv(\bs{A},\bs{B},\bs{C})$ as $n\to\infty$, we have
$$(\bs{A}_n,\bs{B}_n)\mid\{\phi_n(\bs{B}_n,\bs{C}_n)=1\}\quad\wconv\quad(\bs{A},\bs{B})\mid\{\phi(\bs{B},\bs{C})=1\}.$$
\end{lemma}

Define $\hat{\tau}_{h, \Lin, \mk}$, $\hat{\bs{\beta}}_{h,\Lin,\mk}$, $\bs{\beta}_{h,\Lin,\mk}$, $\hat{\se}_{h,\Lin,\mk}^2$, $\tilde{\se}_{h,\Lin,\mk}^2$, $\hat{\tau}_{\Lin, \mk}$,  $\hat{\se}_{\Lin,\mk}^2$, $\tilde{\se}_{\Lin,\mk}^2$, $\hat{\tau}_{\fe , \mk}$, $\hat{\bs{\beta}}_{\fe ,\mk}$, $\bs{\beta}_{\fe ,\mk}$, $\hat{\se}_{\fe ,\mk}^2$, $\tilde{\se}_{\fe ,\mk}^2$ as the analogy of $\hat{\tau}_{h, \Lin}$, $\hat{\bs{\beta}}_{h,\Lin,\bs{x}}$, $\bs{\beta}_{h, \Lin , \bs{x}}$, $\hat{\se}_{h,\Lin}^2$, $\tilde{\se}_{h,\Lin}^2$, $\hat{\tau}_{\Lin}$, $\hat{\se}_{\Lin}^2$, $\tilde{\se}_{\Lin}^2$, $\hat{\tau}_{\fe }$, $\hat{\bs{\beta}}_{\fe, \bs{x}}$, $\bs{\beta}_{\fe, \bs{x}}$, $\hat{\se}_{\fe }^2$, $\tilde{\se}_{\fe }^2$ with $\bs{x}_{i\mk}$ replaced by $\bs{x}_i$.

 \begin{lemma}
 \label{lem:summary-weak-convergence}
    Under Condition \ref{a:CLT-sre-for-rerandomization}, we have under SRE
    \begin{align*}
        &\Big(n_h^{1/2}(\hat{\tau}_{h}-\bar{\tau}_h), n_h^{1/2}\hat{\bs{\tau}}_{h,\bs{x}},(\hat{\bs{\beta}}_{h,\Lin,\mk}, n_h \hat{\se}_{h,\Lin,\mk}^2)_{\mk\subseteq [K]} , (n_h \hat{\sx}^2_{h,\Lin,k})_{k\in [K]} , \bs{V}_{h,\bs{x}\bs{x}}\Big)_{h\in [H]} \wconv \\
        &
\qquad\qquad\qquad\qquad\qquad\qquad\qquad\Big(\varepsilon_{h,\tau}, \bs{\varepsilon}_{h,\bs{x}}, (\bs{\beta}_{h,\Lin,\mk}, n_h \tilde{\se}_{h,\Lin,\mk}^2)_{\mk\subseteq [K]}, (V_{h,kk})_{k \in [K]}, \bs{V}_{h,\bs{x}\bs{x}}\Big)_{h \in [H]};\\
      &\Big(n^{1/2}(\hat{\tau}_{\omega}-\bar{\tau}_{\omega}), n^{1/2}\hat{\bs{\tau}}_{\omega,\bs{x}},(\hat{\bs{\beta}}_{\fe ,\mk}, n\hat{\se}_{\fe ,\mk}^2)_{\mk\subseteq [K]} , (n\hat{\sx}^2_{\fe ,k})_{k\in [K]}\Big) \wconv \\
      &
      \qquad\qquad\qquad\qquad\qquad\qquad\qquad\qquad\qquad\Big(\varepsilon_{\omega,\tau}, \bs{\varepsilon}_{\omega,\bs{x}}, (\bs{\beta}_{\fe ,\mk}, n \tilde{\se}_{\fe ,\mk}^2)_{\mk\subseteq [K]}, (V_{\omega,kk})_{k\in [K]}\Big).
    \end{align*}
    \end{lemma}   
\begin{proof}[Proof of \Cref{lem:summary-weak-convergence}]
   Applying \Cref{lem:limit-of-regression-coefficient} and \ref{lem:limit-of-standard-errors} with $\bs{x}_i \equiv \bs{x}_{i\mk}$ for $\mk \subseteq [K]$, we have
   \begin{align*}
       & (\hat{\bs{\beta}}_{h,\Lin,\mk}, n_h \hat{\se}_{h,\Lin,\mk}^2)_{\mk\subseteq [K]} \wconv (\bs{\beta}_{h,\Lin,\mk}, n_h \tilde{\se}_{h,\Lin,\mk}^2)_{\mk\subseteq [K]};\\
       & (\hat{\bs{\beta}}_{\fe ,\mk}, n_h \hat{\se}_{\fe ,\mk}^2)_{\mk\subseteq [K]} \wconv (\bs{\beta}_{\fe ,\mk}, n_h \tilde{\se}_{\fe ,\mk}^2)_{\mk\subseteq [K]}.
   \end{align*}
   
  Combining this with \Cref{lem:CLT-for-rerandomization}, we complete the proof.
\end{proof}

Let 
\begin{align}
\label{eq:def-of-Nlk-Nfek}
    \mathcal{N}_{\Lin,\mk} = \frac{\sum_{h=1}^H \pi_h^{1/2}(\varepsilon_{h,\tau}-\bs{\beta}_{h,\Lin, \mk}^\top \bs{\varepsilon}_{h,\mk}) }{n^{1/2}\tilde{\se}_{\Lin, \mk}}, \quad \mathcal{N}_{\fe ,\mk} = \frac{\varepsilon_{\omega,\tau}-\bs{\beta}_{\fe ,\mk}^\top \bs{\varepsilon}_{\omega,\mk}}{n^{1/2}\tilde{\se}_{\fe , \mk}},
\end{align}
where $\bs{\varepsilon}_{h,\mk}$ and $\bs{\varepsilon}_{\omega,\mk}$ are subvectors of $\bs{\varepsilon}_{h,\bs{x}}$ and $\bs{\varepsilon}_{\omega,\bs{x}}$ corresponding to $\mk$, respectively. Let $\hat{\bs{\tau}}_{h,\mk}$ and $\hat{\bs{\tau}}_{\omega,\mk}$ be the subvectors of $\hat{\bs{\tau}}_{h,\bs{x}}$ and $\hat{\bs{\tau}}_{\omega,\bs{x}}$ corresponding to $\mk$, respectively.

\begin{lemma}
    \label{lem:limit-of-t-value}
    Assume Condition \ref{a:CLT-sre-for-rerandomization} holds. Under SRE, we have (i) under $H_{0\textsc{n}}$, $(\hat{\tau}_{\Lin,\mk}/\hat{\se}_{\Lin,\mk})_{\mk \subseteq [K]} \wconv (\mathcal{N}_{\Lin,\mk})_{\mk \subseteq [K]}$; (ii) under $H_{0\omega}$, $(\hat{\tau}_{\fe ,\mk}/\hat{\se}_{\fe ,\mk})_{\mk \subseteq [K]} \wconv (\mathcal{N}_{\fe ,\mk})_{\mk \subseteq [K]}$.
\end{lemma}

\begin{proof}[Proof of \Cref{lem:limit-of-t-value}]
    By Lemma \ref{lem:lin-estimator-expression} and
\ref{lem:fe-estimator-expression}, we have
\begin{align*}
    &\frac{\hat{\tau}_{\Lin,\mk}}{\hat{\se}_{\Lin,\mk}} = \frac{n^{1/2}\sum_{h=1}^H \pi_h(\hat{\tau}_{h} - \hat{\bs{\beta}}_{h,\Lin,\mk}^\top\hat{\bs{\tau}}_{h,\mk})}{n^{1/2}\hat{\se}_{\Lin,\mk}} = \frac{\sum_{h=1}^H \pi_h^{1/2} n_{h}^{1/2}(\hat{\tau}_{h} - \hat{\bs{\beta}}_{h,\Lin,\mk}^\top\hat{\bs{\tau}}_{h,\mk})}{n^{1/2}\hat{\se}_{\Lin,\mk}},\\
    &\frac{\hat{\tau}_{\fe ,\mk}}{\hat{\se}_{\fe ,\mk}} = \frac{n^{1/2}(\hat{\tau}_{\omega}-\hat{\bs{\beta}}_{\fe , \mk}^\top \hat{\bs{\tau}}_{\omega,\mk})}{n^{1/2}\hat{\se}_{\fe ,\mk}}.
\end{align*}

The conclusion follows from \Cref{lem:CLT-for-rerandomization} and \ref{lem:limit-of-standard-errors}.
\end{proof}  

\begin{lemma}
\label{lem:standard-normal-under-fisher-null}
   Recall $\mathcal{N}_{\Lin,\mk}$ in \eqref{eq:def-of-Nlk-Nfek}. For all $\mk\subseteq [K]$, $\var(\mathcal{N}_{\Lin,\mk}) \leq 1$; when $H_{0\textsc{f}}$ holds, we have $\var(\mathcal{N}_{\Lin,\mk}) = 1$.
\end{lemma}
\begin{proof}[Proof of \Cref{lem:standard-normal-under-fisher-null}]
   We only prove the case $\mk = [K]$, and the other cases are similar. By definition
  \begin{align*}
      \var(\mathcal{N}_{\Lin,[K]}) &= \frac{\sum_{h=1}^H \pi_h \var(\varepsilon_{h,\tau}-\bs{\beta}_{h,\Lin, \bs{x}}^\top \bs{\varepsilon}_{h,\bs{x}}) }{n\tilde{\se}_{\Lin}^2} = \frac{\sum_{h=1}^H \pi_h (
V_{h,\tau\tau}- \bs{V}_{h,\tau \bs{x}} \bs{V}_{h,\bs{x}\bs{x}}^{-1} \bs{V}_{h,\bs{x}\tau}) }{n\tilde{\se}_{\Lin}^2} \\
&= \frac{\sum_{h=1}^H \pi_h \sigma_{h, \adj}^2 }{\sum_{h=1}^H \pi_h (\sigma_{h, \adj}^2 + S_{h,\tau_e \tau_e})} \leq 1.
  \end{align*}

 When $H_{0\textsc{f}}$ holds, i.e., $Y_i(1)=Y_i(0)$, for $i\in [n]$, $S_{h,\tau_e \tau_e} = 0$. Therefore, $\var(\mathcal{N}_{\Lin,[K]})=1$.
\end{proof}

\begin{lemma}
\label{lem:asymptotic-limit-of-type-I-error-rate}
Under Condition \ref{a:CLT-sre-for-rerandomization} and SRE, we have: under $H_{0\textsc{n}}$,
    \begin{align*}
        &(i)~\Prob_{\infty}(\hpL \leq \alpha\mid \ma_{\repss}(\alpha_t))  =  \Prob\Big( \max_{\mk \subseteq [K]}|\mathcal{N}_{\Lin,\mk}| \geq z_{1-\alpha/2} ~\Big | ~ \max_{h \in [H]}\|\bs{\sigma}(\bs{V}_{h,\bs{x}\bs{x}})^{-1}\bs{\varepsilon}_{h,\bs{x}}\|_{\infty} \leq z_{1-\alpha_t/2} \Big); \\
         &(ii)~\Prob_{\infty}(\hpL \leq \alpha\mid \ma_{\remss}(a))  = \Prob\Big( \max_{\mk \subseteq [K]}|\mathcal{N}_{\Lin,\mk}| \geq z_{1-\alpha/2} ~\Big | ~ \max_{h\in [H]}\bs{\varepsilon}_{h,\bs{x}}^\top \bs{V}_{h,\bs{x}\bs{x}}^{-1} \bs{\varepsilon}_{h,\bs{x}} \leq a \Big);\\
         &(iii)~\Prob_{\infty}(\hpL \leq \alpha) = \Prob\Big( \max_{\mk \subseteq [K]}|\mathcal{N}_{\Lin,\mk}| \geq z_{1-\alpha/2} \Big);\quad \\
         & (iv) ~\Prob_{\infty}(\ma_{\repss}(\alpha_t)) = \Prob\Big(\max_{h \in [H]}\|\bs{\sigma}(\bs{V}_{h,\bs{x}\bs{x}})^{-1}\bs{\varepsilon}_{h,\bs{x}}\|_{\infty} \leq z_{1-\alpha_t/2}\Big);\\
         &(v)~ \Prob_{\infty}(\ma_{\remss}(\alpha_t)) = \Prob\Big(\max_{h\in [H]}\bs{\varepsilon}_{h,\bs{x}}^\top \bs{V}_{h,\bs{x}\bs{x}}^{-1} \bs{\varepsilon}_{h,\bs{x}} \leq a \Big);
    \end{align*}
    under $H_{0\omega}$,
    \begin{align*}
     &(vi)~\Prob_{\infty}(\hpfe \leq \alpha\mid \ma_{\repfe}(\alpha_t))  = \Prob\Big( \max_{\mk \subseteq [K]}|\mathcal{N}_{\fe ,\mk}| \geq z_{1-\alpha/2} ~ \Big | ~ \|\bs{\sigma}(\bs{V}_{\omega,\bs{x}\bs{x}})^{-1}\bs{\varepsilon}_{\omega,\bs{x}}\|_{\infty} \leq z_{1-\alpha_t/2} \Big);\\
      &(vii)~\Prob_{\infty}(\hpfe \leq \alpha)  = \Prob\Big( \max_{\mk \subseteq [K]}|\mathcal{N}_{\fe ,\mk}| \geq z_{1-\alpha/2}  \Big);\\
     &(viii)~ \Prob_{\infty}(\ma_{\repfe}(\alpha_t)) = \Prob( \|\bs{\sigma}(\bs{V}_{\omega,\bs{x}\bs{x}})^{-1}\bs{\varepsilon}_{\omega,\bs{x}}\|_{\infty} \leq z_{1-\alpha_t/2}).
    \end{align*}
\end{lemma}
\begin{proof}[Proof of \Cref{lem:asymptotic-limit-of-type-I-error-rate}]
% Denote $\ma_{h,\rem}(a) = \{n_h\hat{\bs{\tau}}_{h,\bs{x}}^\top \bs{V}_{h,\bs{x}\bs{x}}^{-1}\hat{\bs{\tau}}_{h,\bs{x}} \leq a\}$. We have $\ma_{\remss}(a) = \ma_{1,\rem}(a)\cap \ldots \cap \ma_{H,\rem}(a)$. For each stratum,
 By definition,
   \begin{align*}
    & p_{\Lin, \mk} = 2\Big\{1-\Phi\Big(\big|\frac{\hat{\tau}_{\Lin,\mk}}{\hat{\se}_{\Lin,\mk}}\big|\Big)\Big\},\quad  p_{\fe , \mk} = 2\Big\{1-\Phi\Big(\big|\frac{\hat{\tau}_{\fe ,\mk}}{\hat{\se}_{\fe ,\mk}}\big|\Big)\Big\},\\
     & p_{t,hk} = 2\Big\{1-\Phi\Big(\big|\frac{\hat{\tau}_{h,k}}{\hat{\sx}_{h,\Lin,k}}\big|\Big) \Big\},\quad p_{\fe ,k}^{x} =  2\Big\{1-\Phi\Big(\big|\frac{\hat{\tau}_{\omega, k}}{\hat{\sx}_{\fe ,k}}\big|\Big)\Big\}.
   \end{align*}
   We have 
   \begin{align*}
   & \ma_{\repss}(\alpha_t) = \{\max_{h\in [H]}\max_{k\in [K]} \big|\hat{\tau}_{h,k}/\hat{\sx}_{h,\Lin,k}\big|  \leq z_{1-\alpha_t/2}\},\quad  \ma_{\repfe}(\alpha_t) = \{\max_{k\in [K]} \big|\hat{\tau}_{\omega, k}/\hat{\sx}_{\fe ,k}\big|  \leq z_{1-\alpha_t/2}\},\\
   &\{ \hpL \leq \alpha \} = \Big\{\max_{\mk \subseteq [K]}\frac{\big|\sum_{h=1}^H \pi_h(\hat{\tau}_{h} - \hat{\bs{\beta}}_{h,\Lin,\mk}^\top\hat{\bs{\tau}}_{h,\mk})\big|}{(\sum_{h=1}^H \pi_h^2 \hat{\se}_{h,\Lin,\mk}^2)^{1/2}} \geq z_{1-\alpha/2}\Big\},\\
   & \{\hpfe \leq \alpha\} = \Big\{\max_{\mk \subseteq [K]}\frac{\big|\hat{\tau}_{\fe ,\mk}\big|}{\hat{\se}_{\fe ,\mk}} \geq z_{1-\alpha/2}\Big\}.
\end{align*}

$(iii)$--$(v)$, $(vii)$-$(viii)$ follows  from \Cref{lem:summary-weak-convergence}.
    % \begin{align*}
    %     & (n_h^{1/2}(\hat{\tau}_{h}-\bar{\tau}_h), n_h^{1/2}\hat{\bs{\tau}}_{h,\bs{x}},(\hat{\bs{\beta}}_{h,\Lin,\mk}, n_h \hat{\se}_{h,\Lin,\mk}^2)_{\mk\subseteq [K]}) \mid \ma_{h,\rem}(a) \wconv \\
    %     &
    %   (\varepsilon_{h,\tau}, \bs{\varepsilon}_{h,\bs{x}}, (\bs{\beta}_{h,\Lin,\mk}, n_h \tilde{\se}_{h,\Lin,\mk}^2)_{\mk\subseteq [K]}) \mid \bs{\varepsilon}_{h,\bs{x}}^\top \bs{V}_{h,\bs{x}\bs{x}}^{-1} \bs{\varepsilon}_{h,\bs{x}} \leq a.
    % \end{align*}
    % Using independence for $h=1,\ldots,H$, we have

 Applying \Cref{lem:weak-convergence-under-ReG} with $(\bs{B}_n,\bs{C}_n) = ((\bs{B}_{n,h})_{h\in [H]},\diag(\bs{C}_{n, h})_{h\in [H]} )$ and $\phi(\bs{B}_n,\bs{C}_n) = I(\max_{h\in [H]} \|\bs{\sigma}(\bs{C}_{n,h})^{-1}\bs{B}_{n,h}\|_{\infty}\leq z_{1-\alpha_t/2})$ where $\bs{B}_{n,h} =  n_h^{1/2}\hat{\bs{\tau}}_{h,\bs{x}}$, $ \bs{C}_{n, h} = \diag(n_h\hat{\sx}^2_{h,\Lin,k})_{k\in [K]} $, and by Lemma \ref{lem:summary-weak-convergence},  we have
     \begin{align*}
        & \Big(n_h^{1/2}(\hat{\tau}_{h}-\bar{\tau}_h), n_h^{1/2}\hat{\bs{\tau}}_{h,\bs{x}},(\hat{\bs{\beta}}_{h,\Lin,\mk}, n_h \hat{\se}_{h,\Lin,\mk}^2)_{\mk\subseteq [K]}\Big)_{h\in [H]} ~\Big|~ \ma_{\repss}(\alpha_t) \wconv \\
        &
      \Big(\varepsilon_{h,\tau}, \bs{\varepsilon}_{h,\bs{x}}, (\bs{\beta}_{h,\Lin,\mk}, n_h \tilde{\se}_{h,\Lin,\mk}^2)_{\mk\subseteq [K]}\Big)_{h\in [H]} ~\Big | ~ \max_{h\in [H]}\|\bs{\sigma}(\bs{V}_{h,\bs{x}\bs{x}})^{-1}\bs{\varepsilon}_{h,\bs{x}}\|_{\infty} \leq z_{1-\alpha_t/2}.
    \end{align*}
Lemma $(i)$ follows.

By \Cref{lem:weak-convergence-under-ReG} with $(\bs{B}_n,\bs{C}_n) = ((\bs{B}_{n,h})_{h\in [H]}, \diag(\bs{C}_{n,h})_{h\in [H]})$ and $\phi(\bs{B}_n, \bs{C}_n) = I($ $\max_{h\in [H]} \bs{B}_{n,h}^\top \bs{C}_{n,h}^{-1} \bs{B}_{n,h} \leq a)$, where $\bs{B}_{n,h} = n_h^{1/2}\hat{\bs{\tau}}_{h,\bs{x}}$ and $\bs{C}_{n,h} = \bs{V}_{h,\bs{x}\bs{x}}$ , and \Cref{lem:summary-weak-convergence}, we have 
    \begin{align*}
        & \Big(n_h^{1/2}(\hat{\tau}_{h}-\bar{\tau}_h), n_h^{1/2}\hat{\bs{\tau}}_{h,\bs{x}},(\hat{\bs{\beta}}_{h,\Lin,\mk}, n_h \hat{\se}_{h,\Lin,\mk}^2)_{\mk\subseteq [K]}\Big)_{h\in [H]} ~\Big|~ \ma_{\remss}(a) \wconv \\
        &
      \Big(\varepsilon_{h,\tau}, \bs{\varepsilon}_{h,\bs{x}}, (\bs{\beta}_{h,\Lin,\mk}, n_h \tilde{\se}_{h,\Lin,\mk}^2)_{\mk\subseteq [K]}\Big)_{h\in [H]} ~\Big|~ \max_{h\in [H]}\bs{\varepsilon}_{h,\bs{x}}^\top \bs{V}_{h,\bs{x}\bs{x}}^{-1} \bs{\varepsilon}_{h,\bs{x}} \leq a.
    \end{align*}
Lemma $(ii)$ follows.

By \Cref{lem:weak-convergence-under-ReG} with $(\bs{B}_n,\bs{C}_n) = (n^{1/2}\hat{\bs{\tau}}_{\omega,\bs{x}},\diag(n\hat{\sx}^2_{\fe ,k})_{k\in [K]} )$ and $\phi(\bs{B}_n,\bs{C}_n) = I($ $\|\bs{\sigma}(\bs{C}_n)^{-1}\bs{B}_n\|_{\infty}\leq z_{1-\alpha_t/2})$, and Lemma \ref{lem:summary-weak-convergence}, we have
\begin{align*}
   & \Big(n^{1/2}(\hat{\tau}_{\omega}-\bar{\tau}_{\omega}), n^{1/2}\hat{\bs{\tau}}_{\omega,\bs{x}},(\hat{\bs{\beta}}_{\fe ,\mk}, n \hat{\se}_{\fe ,\mk}^2)_{\mk\subseteq [K]} \Big) ~\Big|~ \ma_{\repfe}(\alpha_t) \wconv \\
      & \qquad
      \Big(\varepsilon_{\omega,\tau}, \bs{\varepsilon}_{\omega,\bs{x}}, (\bs{\beta}_{\fe ,\mk}, n \tilde{\se}_{\fe ,\mk}^2)_{\mk\subseteq [K]}\Big) ~\Big|~ \|\bs{\sigma}(\bs{V}_{\omega,\bs{x}\bs{x}})^{-1}\bs{\varepsilon}_{\omega,\bs{x}}\|_{\infty} \leq z_{1-\alpha_t/2}.
\end{align*}
Lemma $(vi)$ follows.

\end{proof}

\section{Proofs of theoretical results under CRE}
\label{sec:results-under-cre}

Let $\bs{V}_{\bs{x}\tau}^\top = \bs{V}_{\tau\bs{x}} = n\cov(\hat{\tau},\hat{\bs{\tau}}_{\bs{x}})$, $V_{\tau\tau} = n\var(\hat{\tau})$ and $\bs{V}_{\bs{x}\bs{x}} = n\cov(\hat{\bs{\tau}}_{\bs{x}})$. Let $\bs{V}_{\mk\mk}$ and $\bs{V}_{\mk\tau}$ be the submatrices of $\bs{V}_{\bs{x}\bs{x}}$ and $\bs{V}_{\bs{x}\tau}$ corresponding to $\mk$, respectively. Let $\bs{\beta}_{\Lin, \bs{x}} = \bs{V}_{\bs{x}\bs{x}}^{-1}\bs{V}_{\bs{x}\tau}$ and $\bs{\beta}_{\Lin,\mk} = \bs{V}_{\mk\mk}^{-1}\bs{V}_{\mk\tau}$. Define the random vector $(\varepsilon_{\tau},\bs{\varepsilon}_{\bs{x}}^\top)$:
\begin{align*}
    \begin{pmatrix}
      \varepsilon_{\tau} \\
      \bs{\varepsilon}_{\bs{x}}
    \end{pmatrix} \sim \mathcal{N}\left(0,\begin{pmatrix}
      V_{\tau\tau} & \bs{V}_{\tau\bs{x}}\\
      \bs{V}_{\bs{x}\tau} & \bs{V}_{\bs{x}\bs{x}}
      \end{pmatrix}\right).
  \end{align*}
  Let $\bs{\varepsilon}_{\mk}$ be the subector of $\bs{\varepsilon}_{\bs{x}}$ corresponding to $\mk$. 

  A CRE is an SRE with $H=1$. $\mathcal{N}_{\Lin,\mk}$ defined in \eqref{eq:def-of-Nlk-Nfek} reduces to
  \[
  \mathcal{N}_{\Lin,\mk} = \frac{\varepsilon_{\tau}-\bs{\beta}_{\Lin, \mk}^\top \bs{\varepsilon}_{\mk}}{n^{1/2}\tilde{\se}_{\Lin, \mk}}.
  \]

\subsection{Proof of \Cref{prop:type I-error-rate-asymptotic}}
\begin{proof}[Proof of \Cref{prop:type I-error-rate-asymptotic}]

By \Cref{lem:standard-normal-under-fisher-null}, under $H_{0\textsc{f}}$, $\mathcal{N}_{\Lin,\mk}$, $\mk\subseteq [K]$, satisfies $\cov(\mathcal{N}_{\Lin,\mk}) = 1$. Therefore, $\mathcal{N}_{\Lin,\mk}$ are standard Gaussian random variables. Therefore, 
\[
\Prob_{\infty}(\min_{\mk\subseteq [K]} p_{\Lin,\mk} \leq \alpha) = \Prob(\max_{\mk\subseteq [K]} |\mathcal{N}_{\Lin,\mk}|\geq z_{1-\alpha/2}) \geq \Prob(|\mathcal{N}_{\Lin,\emptyset}|\geq z_{1-\alpha/2}) = \alpha.
\]

The above inequality holds if and only if $\cov(\mathcal{N}_{\Lin,\mk}, \mathcal{N}_{\Lin,\emptyset}) = 1$, for all $\mk \subseteq [K]$, i.e., $R^2_{\bs{x}}=0$.    
\end{proof}

\subsection{Proof of Theorem \ref{thm:rem-mitigates-p-hacking} and Theorem \ref{thm:rep-mitigates-p-hacking}}
It follows from Theorem \ref{thm:re-mitigates-p-hacking-sre} with $H=1$.

\subsection{Proof of Theorem \ref{thm:iff-for-rem_control-TIE}}
\begin{proof}[Proof of \Cref{thm:iff-for-rem_control-TIE}]
\noindent \textbf{The proof of the``if" part}:

Applying \Cref{lem:limit-of-t-value} with $H=1$, we have $(\hat{\tau}_{\Lin,\mk}/\hat{\se}_{\Lin,\mk})_{\mk\subseteq [K]} \wconv (\mathcal{N}_{\Lin,\mk})_{\mk \subseteq [K]}.$

 The ``if" part  follows from \Cref{lem:if-statement-for-rem-sre} with $H=1$. 

\noindent \textbf{The proof of the ``only if" part}:

To prove the ``only if" part, we prove that if $a > \bar{a}_{\rem}(\alpha,R^2_{\bs{x}},\underline{\Delta})$, for $(R^2_{\bs{x}}, \underline{\Delta}) \in \mathcal{B}$, there exists $(Y_i(1),Y_i(0))_{i=1}^n \in \mathbb{M}(R^2_{\bs{x}},\underline{\Delta})$ such that $\Prob_{\infty}(\hpL\mid \ma_{\rem}(a)) > 0$. 

Since $(R^2_{\bs{x}}, \underline{\Delta}) \in \mathcal{B}$, there exists $(\breve{Y}_i(1), \breve{Y}_i(0))_{i=1}^n \in \mathbb{M}(R^2_{\bs{x}},\underline{\Delta})$.  We construct $(Y_i(1),Y_i(0))_{i=1}^n$ $ \in \mathbb{M}(R^2_{\bs{x}},\underline{\Delta})$ as follows: $Y_i(1)=Y_i(0) = r_1 \tilde{Y}_i(0)+r_0 \tilde{Y}_i(1)$ for $i\in [n]$. Since $\tau_i = 0$ under this construction of $(Y_i(1),Y_i(0))_{i=1}^n$, we can see that by the definition of $\tilde{\se}_{\Lin, \mk}$, $n\tilde{\se}_{\Lin, \mk}^2 = \var(\varepsilon_{\tau}-\bs{\beta}_{\Lin, \mk}^\top \bs{\varepsilon}_{\mk})$ and therefore $ \mathcal{N}_{\Lin, \mk} \sim \mathcal{N}(0,1)$. 

 Let $\bar{R}_{\bs{x}}^2 = 1- R_{\bs{x}}^2$, $\bs{\xi} = \bs{V}_{\bs{x}\bs{x}}^{-1/2} \bs{\varepsilon}_{\bs{x}}$, $\varepsilon_0 =  \mathcal{N}_{\Lin,[K]}$, and $\tilde{R}^2_{\mk} = R^2_{\bs{x}}-R^2_{\mk}$. Let $\mathcal{E} = \{ \max_{\mk \subseteq [K]} |\mathcal{N}_{\Lin,\mk}| \geq z_{1-\alpha/2}\}$, $\ma_{\rem}^\infty (a) = \{\bs{\varepsilon}_{\bs{x}}^\top \bs{V}_{\bs{x}\bs{x}}^{-1}\bs{\varepsilon}_{\bs{x}}\leq a\}$. 

We see that 
\begin{align*}
\Prob_{\infty}(\hpL \leq \alpha \mid \mathcal{A}_{\rem}(a)) = \Prob(\mathcal{E}\mid\mathcal{A}_{\rem}^{\infty}(a)) = \Prob(\mathcal{E}_1\mid\mathcal{A}_{\rem}^{\infty}(a))+\Prob(\mathcal{E}_2\mid\mathcal{A}_{\rem}^{\infty}(a)),
\end{align*}
where
\begin{gather*}
    \mathcal{E}_1 = \{ |\mathcal{N}_{\Lin,[K]}| \geq z_{1-\alpha/2}\} = \{ |\varepsilon_0| \geq z_{1-\alpha/2}\},\quad \mathcal{E}_2 = \{ \max_{\mk \subsetneq [K]} |\mathcal{N}_{\Lin,\mk}| \geq z_{1-\alpha/2}, |\mathcal{N}_{\Lin,[K]}| < z_{1-\alpha/2}\}.
\end{gather*}

For $\mathcal{E}_1$, because $\bs{\varepsilon}_{\bs{x}}$ and $\varepsilon_0$ are independent, we have
\[
\Prob(\mathcal{E}_1\mid\mathcal{A}_{\rem}^{\infty}(a)) = \Prob(\mathcal{E}_1) = \Prob(|\varepsilon_0|\geq z_{1-\alpha/2}) = \alpha.
\]

It remains to show that $\Prob(\mathcal{E}_2\cap\mathcal{A}_{\remss}^{\infty}(a))>0$. $\mathcal{E}_2\cap\mathcal{A}_{\remss}^{\infty}(a)$ can be seen as a measurable set of $(\varepsilon_0,\bs{\xi})$.  Since nonempty open set has measure greater than $0$, it suffices to show that the following open subset $\mathcal{E}_3 \subseteq \mathcal{E}_2\cap\mathcal{A}_{\remss}^{\infty}(a)$ is nonempty: 
\[
\mathcal{E}_3 =\{|\mathcal{N}_{\Lin, [K]}|< z_{1-\alpha/2}, \max_{\mk \subsetneq [K]} |\mathcal{N}_{\Lin, \mk}| > z_{1-\alpha/2} \} \cap \{ \|\bs{\xi}\|_2^2 <a\}. 
\]

 Define $\tilde{\bs{\beta}}_{\Lin,\mk}\in \mathbb{R}^K$ with $(\tilde{\bs{\beta}}_{\Lin,\mk})_{\mk} = \bs{\beta}_{\Lin,\mk}$ and $0$ at the other entries. We have $\var(\varepsilon_{\tau}-\bs{\beta}_{\Lin, \mk}^\top \bs{\varepsilon}_{\mk}) = (1-R^2_{\mk}) V_{\tau\tau} = (\bar{R}^2_{\bs{x}}+\tilde{R}_{\mk}^2) V_{\tau\tau}$ and
\[
|\mathcal{N}_{\Lin,\mk}| = \frac{|V_{\tau\tau}^{1/2}\bar{R}_{\bs{x}}\varepsilon_0 + \bs{\varepsilon}_{\bs{x}}^\top (\bs{\beta}_{\Lin, \bs{x}}-\tilde{\bs{\beta}}_{\Lin,\mk})|}{(\bar{R}^2_{\bs{x}}+\tilde{R}_{\mk}^2)^{1/2}V_{\tau\tau}^{1/2}} = \frac{\Big|\bar{R}_{\bs{x}}|\varepsilon_0|+s_{\mk}\|\bs{\xi}\|_2 \tilde{R}_{\mk}\Big|}{(\bar{R}_{\bs{x}}^2 + \tilde{R}^2_{\mk})^{1/2}},
\]
 where $s_{\mk} = \operatorname{sign}(\varepsilon_{0})\bs{\varepsilon}_{\bs{x}}^\top (\bs{\beta}_{\Lin, \bs{x}}-\tilde{\bs{\beta}}_{\Lin,\mk})/(V_{\tau\tau}^{1/2}\|\bs{\xi}\|_2\tilde{R}_{\mk})$.

Let 
\begin{align*}
    k_0 = \argmin_{k \in [K]} \frac{(\bar{R}_{\bs{x}}^2 + \Delta_{k})^{1/2}z_{1-\alpha/2}-\bar{R}_{\bs{x}}|\varepsilon_0|}{\Delta_{k}^{1/2}},\quad \psi(|\varepsilon_0|)  = \min_{k\in [K]} \frac{(\bar{R}_{\bs{x}}^2 + \Delta_{k})^{1/2}z_{1-\alpha/2}-\bar{R}_{\bs{x}}|\varepsilon_0|}{\Delta_{k}^{1/2}} .
\end{align*}

Noticing that $\Delta_{k_0} = \tilde{R}^2_{[-k_0]}$, we have, under $\{s_{[-k_0]}=1\}$,
\begin{align*}
    \max_{\mk \subsetneq [K]} |\mathcal{N}_{\Lin, \mk}| &\geq |\mathcal{N}_{\Lin, [-k_0]}| = \frac{\bar{R}_{\bs{x}}|\varepsilon_0|+\|\bs{\xi}\|_2 \Delta_{k_0}^{1/2} }{(\bar{R}_{\bs{x}}^2 + \Delta_{k_0})^{1/2}}.
\end{align*}

Therefore,
\begin{align*}
    & \mathcal{E}_3 \cap \{s_{[-k_0]}=1\} \\
   \supseteq &  \Big\{|\varepsilon_0| < z_{1-\alpha/2},s_{ [-k_0]}=1,\frac{\bar{R}_{\bs{x}}|\varepsilon_0|+\|\bs{\xi}\|_2 \Delta_{k_0}^{1/2} }{(\bar{R}_{\bs{x}}^2 + \Delta_{k_0})^{1/2}} > z_{1-\alpha/2}, \|\bs{\xi}\|_2^2 < a \Big\}\\
   \supseteq & \Big\{|\varepsilon_0| < z_{1-\alpha/2},s_{[-k_0]}=1, \|\bs{\xi}\|_2 >  \psi(|\varepsilon_0|),  \|\bs{\xi}\|_2^2 < a \Big\} .
\end{align*}

We see that
\begin{align*} 
    &\psi(z_{1-\alpha/2}) = \min_{k\in [K]} \frac{(\bar{R}_{\bs{x}}^2 + \Delta_{k})^{1/2}z_{1-\alpha/2}-\bar{R}_{\bs{x}}z_{1-\alpha/2}}{\Delta_{k}^{1/2}} = \{\bar{a}_{\rem}(\alpha, R^2_{ \bs{x}}, \underline{\Delta})\}^{1/2}  < a^{1/2}.
\end{align*}

Since $\psi(|\varepsilon_0|)$ is a continuous function of $|\varepsilon_0|$, there exists $0 < t < z_{1-\alpha/2}$ such that $\psi(t) < a^{1/2}$. Therefore, 
\begin{align*}
     \mathcal{E}_3 \cap \{s_{[-k_0]}=1\} \supseteq & \Big\{|\varepsilon_0| < z_{1-\alpha/2},~ s_{[-k_0]}=1, \|\bs{\xi}\|_2 >  \psi(|\varepsilon_0|),  ~ \|\bs{\xi}\|_2^2 < a \Big\} \\
    \supseteq & \Big\{|\varepsilon_0| = t ,~s_{[-k_0]}=1, ~ a^{1/2} > \|\bs{\xi}\|_2 >  \psi(t) \Big\}.
\end{align*}
We see that
\[
s_{\mk} = \frac{\operatorname{sign}(\varepsilon_{0})\bs{\varepsilon}_{\bs{x}}^\top (\bs{\beta}_{\Lin, \bs{x}}-\tilde{\bs{\beta}}_{\Lin,\mk})}{\|\bs{\xi}\|_2\tilde{R}_{\mk}V_{\tau\tau}^{1/2}} = \frac{\operatorname{sign}(\varepsilon_{0})\bs{\xi}^\top \bs{V}_{\bs{x}\bs{x}}^{1/2}(\bs{\beta}_{\Lin, \bs{x}}-\tilde{\bs{\beta}}_{\Lin,\mk})}{\|\bs{\xi}\|_2 \cdot\|\bs{V}_{\bs{x}\bs{x}}^{1/2}(\bs{\beta}_{\Lin, \bs{x}}-\tilde{\bs{\beta}}_{\Lin,\mk})\|_2},
\]
which implies that
\[
\{s_{[-k_0]} = 1\} = \Big\{ \frac{\bs{\xi}}{\|\bs{\xi}\|_2} = \bs{\nu}(\varepsilon_0) \Big\}, \quad \text{where}\quad \bs{\nu}(\varepsilon_0) = \frac{\operatorname{sign}(\varepsilon_{0}) \bs{V}_{\bs{x}\bs{x}}^{1/2}(\bs{\beta}_{\Lin, \bs{x}}-\tilde{\bs{\beta}}_{\Lin,[-k_0]})}{\|\bs{V}_{\bs{x}\bs{x}}^{1/2}(\bs{\beta}_{\Lin, \bs{x}}-\tilde{\bs{\beta}}_{\Lin,[-k_0]})\|_2}.
\]

It follows that
\begin{align*}
     \mathcal{E}_3 \cap \{s_{[-k_0]}=1\} \supseteq & \Big\{\varepsilon_0 =t,~ \frac{\bs{\xi}}{\|\bs{\xi}\|_2} = \bs{\nu}(t), ~ a^{1/2} > \|\bs{\xi}\|_2 >  \psi(t) \Big\}. 
\end{align*}
Therefore, $\mathcal{E}_3$ is  nonempty. We complete the proof. 
\end{proof}

\subsection{Proof of \Cref{cor:bound-of-rep-derived-by-direct-inequality}}
Let $\ma_{\rep}^{\infty}(\alpha_t) = \{\|\bs{\sigma}(\bs{V}_{\bs{x}\bs{x}})^{-1}\bs{\varepsilon}_{\bs{x}}\|_{\infty} \leq z_{1-\alpha_t/2}\}$ denote the asymptotic limit of $\ma_{\rep}(\alpha_t)$.

 Let $\bar{R}_{\bs{x}}^2 = 1- R_{\bs{x}}^2$. By \eqref{eq:me-sufficient-condition} in \Cref{lem:if-statement-for-rem-sre}, we have 
\begin{align*}
    \Prob_{\infty}(\hpL \leq \alpha \mid \ma_{\rep}(\alpha_t)) -\alpha \leq \Prob(\mathcal{E}_2\mid \ma_{\rep}^{\infty}(\alpha_t))
\end{align*}
where
\begin{align*}
   \mathcal{E}_2 =  \Big\{|\varepsilon_0| < z_{1-\alpha/2}, \|\bs{V}_{\bs{x}\bs{x}}^{-1/2}\bs{\varepsilon}_{\bs{x}}\|_2 \geq   \{\bar{a}_{\rem}(\alpha, R^2_{\bs{x}}, \underline{\Delta})\}^{1/2}\Big\},
\end{align*}
and $\varepsilon_0$ is a standard Gaussian random variable independent of $\bs{\varepsilon}_{\bs{x}}$.

Let $ \bs{\xi}= \bs{D}(\bs{V}_{\bs{x}\bs{x}})^{-1/2}\bs{\sigma}(\bs{V}_{\bs{x}\bs{x}})^{-1}\bs{\varepsilon}_{\bs{x}}$. We have $\|\bs{\xi}\|_2 = \|\bs{V}_{\bs{x}\bs{x}}^{-1/2}\bs{\varepsilon}_{\bs{x}}\|_2$ and
\begin{align*}
    \ma_{\rep}^{\infty}(\alpha_t) &= \{\|\bs{\sigma}(\bs{V}_{\bs{x}\bs{x}})^{-1}\bs{\varepsilon}_{\bs{x}}\|_{\infty} \leq z_{1-\alpha_t/2}\} = \{\|\bs{D}(\bs{V}_{\bs{x}\bs{x}})^{1/2}\bs{\xi}\|_{\infty} \leq z_{1-\alpha_t/2}\}.
\end{align*}

Since $\|\bs{\xi}\|_2 \leq \|\bs{D}(\bs{V}_{\bs{x}\bs{x}})^{-1/2}\|_{\infty,2} \|\bs{D}(\bs{V}_{\bs{x}\bs{x}})^{1/2}\bs{\xi}\|_{\infty},$
we have, if
\[
z_{1-\alpha_t/2} \leq  \|\bs{D}(\bs{V}_{\bs{x}\bs{x}})^{-1/2}\|_{\infty,2}^{-1} \{\bar{a}_{\rem}(\alpha, R^2_{\bs{x}}, \underline{\Delta})\}^{1/2},
\]
 $\|\bs{D}(\bs{V}_{\bs{x}\bs{x}})^{1/2}\bs{\xi}\|_{\infty} \leq z_{1-\alpha_t/2}$ implies that $\|\bs{\xi}\|_2  \leq  \{\bar{a}_{\rem}(\alpha, R^2_{\bs{x}}, \underline{\Delta})\}^{1/2}.$

Therefore $\ma_{\rep}^{\infty}(\alpha_t) \subseteq \ma_{\rep}^\infty(\bar{a}_{\rem}(\alpha, R^2_{\bs{x}},\underline{\Delta}))$. Since $\Prob\big\{\ma_{\rem}^\infty(\bar{a}_{\rem}(\alpha, R^2_{\bs{x}},\underline{\Delta})) \cap \mathcal{E}_2\big\} = 0$,
\[
 \Prob\big\{\ma_{\rep}^{\infty}(\alpha_t) \cap \mathcal{E}_2\big\} \leq \Prob\big\{\ma_{\rem}^\infty(\bar{a}_{\rem}(\alpha, R^2_{\bs{x}},\underline{\Delta})) \cap \mathcal{E}_2\big\} = 0
\]

Thus, we prove the desired result.

\subsection{Proofs of Theorems \ref{thm:ReM-bound-of-I-error-rate-inflation}, \ref{thm:rep-orthogonal-covariates-equally-importance}, and \ref{thm:ReP-bound-of-I-error-rate-inflation}} 
Proofs of Theorems \ref{thm:ReM-bound-of-I-error-rate-inflation}, \ref{thm:rep-orthogonal-covariates-equally-importance}, and \ref{thm:ReP-bound-of-I-error-rate-inflation} follows from Theorem \ref{thm:inflation-bound-sre} $(i)$, Theorem \ref{thm:sufficient-complete-condition-for-orthogonal-equal-importance} $(ii)$, and Theorem \ref{thm:inflation-bound-sre} $(ii)$  with $H=1$, respectively.

\section{Proofs of theoretical results under SRE}

\label{sec:results-under-sre}

\subsection{Proof of Theorem \ref{thm:re-mitigates-p-hacking-sre}}

\begin{definition}
    For two symmetric random vectors $\bs{A}$ and $\bs{B}$ in $\mathbb{R}^m$, we say $\bs{A}$ \emph{is more peaked than} $\bs{B}$, denoted by $\bs{A}\succeq \bs{B}$, if
$\Prob(\bs{A}\in\mathcal{C})\geq\Prob(\bs{B}\in\mathcal{C})$ for all symmetric convex sets $\mathcal{C}$ in $\mathbb{R}^m$.
\end{definition}

\Cref{lem:gaussian-correlation-inequality} is the Gaussian correlation inequality \citep{royen}.
\begin{lemma}
\label{lem:gaussian-correlation-inequality}
    Let $\Prob_{\mathcal{N}}$ be a probability measure on $\mathbb{R}^m,m>1$, given by a Gaussian random variable with  $\bs{0}$ mean and a non-singular covariance matrix. We have
$$\Prob_{\mathcal{N}}(\mathcal{C}_1\cap \mathcal{C}_2)\geq \Prob_{\mathcal{N}}(\mathcal{C}_1) \Prob_{\mathcal{N}}(\mathcal{C}_2)$$
for all symmetric convex sets $\mathcal{C}_1,\mathcal{C}_2\in\mathbb{R}^m$.
\end{lemma}

Define $\tilde{\bs{\beta}}_{h,\Lin,\mk}\in \mathbb{R}^K$ such that $(\tilde{\bs{\beta}}_{h,\Lin,\mk})_{\mk} = \bs{\beta}_{h,\Lin,\mk}$ and all other entries are $0$.

\begin{lemma}
\label{lem:re-mitigate-p-hacking-sre-(i)}
Consider SRE and assume Condition \ref{a:CLT-sre-for-rerandomization} holds. 
 (i) Under $H_{0\textsc{n}}$, we have
    \begin{align*}
        &\Prob_{\infty}(\hpL \leq \alpha\mid \ma_{\repss}(\alpha_t)) \leq \Prob_{\infty}(\hpL \leq \alpha), \quad \Prob_{\infty}(\hpL \leq \alpha\mid \ma_{\remss}(a)) \leq \Prob_{\infty}(\hpL \leq \alpha);
    \end{align*}
and (ii) under $H_{0\omega}$, we have
 \begin{align*}
            &\Prob_{\infty}(\hpfe \leq \alpha\mid \ma_{\repfe}(\alpha_t)) \leq \Prob_{\infty}(\hpfe \leq \alpha).
 \end{align*}
\end{lemma}

\begin{proof}[Proof of \Cref{lem:re-mitigate-p-hacking-sre-(i)}]
   We see that, by \Cref{lem:asymptotic-limit-of-type-I-error-rate},
    \begin{align*}
       \Prob_{\infty} (\hpL > \alpha) =&  \Prob \bigg(\cap_{\mk \subseteq [K]}\Big\{ -z_{1-\alpha/2} < \frac{\sum_{h=1}^H \pi_h^{1/2}(\varepsilon_{h,\tau}-\bs{\beta}_{h,\Lin, \mk}^\top \bs{\varepsilon}_{h,\mk}) }{n^{1/2}\tilde{\se}_{\Lin, \mk}} < z_{1-\alpha/2} \Big\}\bigg)
        \\
        =&  \Prob \Big((\varepsilon_{h,\tau},\bs{\varepsilon}_{h,\bs{x}}^\top)_{h\in [H]}  \in \cap_{\mk \subseteq [K]}\mathcal{C}_{\mk} \Big),
    \end{align*}
    where
    \begin{align*}
        \mathcal{C}_{\mk} = \{\bs{x}\in \mathbb{R}^{H (K+1)} \mid -z_{1-\alpha/2} < (n^{1/2}\tilde{\se}_{\Lin, \mk})^{-1}(\pi_h^{1/2},-\pi_h^{1/2} \tilde{\bs{\beta}}_{h,\Lin,\mk}^\top)_{h\in [H]} \bs{x}< z_{1-\alpha/2}\}.
    \end{align*}

    Since $\mathcal{C}_{\mk}$, $\mk \subseteq [K]$ are symmetric convex sets, so is $\cap_{\mk \subseteq [K]} \mathcal{C}_{\mk}$.

    On the other hand, we have, for $\star \in\{\repss,\remss\}$,
 \begin{align*}
        &\Prob_{\infty}(\hpL \leq \alpha\mid \ma_{\star}(\alpha_t)) = \Prob \Big((\varepsilon_{h,\tau},\bs{\varepsilon}_{h,\bs{x}}^\top)_{h\in [H]}  \in \cap_{\mk \subseteq [K]}\mathcal{C}_{\mk} ~\Big | ~ (\varepsilon_{h,\tau},\bs{\varepsilon}_{h,\bs{x}}^\top)_{h\in [H]}  \in \mathcal{C}_{\star} \Big),
    \end{align*}
    where 
    \begin{align*}
        &\mathcal{C}_{\repss} = \{  \bs{x}= (\bs{x}_{h})_{h\in [H]} \in \mathbb{R}^{H (K+1)} \mid \bs{x}_{h} \in \mathbb{R}^{K+1}, \|\bs{\sigma}(\bs{V}_{h,\bs{x}\bs{x}})^{-1}\bs{x}_{h,2:(K+1)}\|_{\infty} \leq z_{1-\alpha_t/2}, h\in [H]\};\\
        &\mathcal{C}_{\remss} = \{\bs{x}= (\bs{x}_{h})_{h\in [H]} \in \mathbb{R}^{H (K+1)} \mid \bs{x}_{h} \in \mathbb{R}^{K+1}, \bs{x}_{h,2:(K+1)}^\top \bs{V}_{h,\bs{x}\bs{x}}^{-1} \bs{x}_{h,2:(K+1)} \leq a, h\in [H] \}.
    \end{align*}

    $\mathcal{C}_{\remss}$ and $\mathcal{C}_{\repss}$ are both symmetric convex sets. 

    If follows that, by \Cref{lem:gaussian-correlation-inequality},
    \begin{align*}
        &\Prob_{\infty}(\hpL > \alpha\mid \ma_{\repss}(\alpha_t)) = \frac{\Prob \Big((\varepsilon_{h,\tau},\bs{\varepsilon}_{h,\bs{x}}^\top)_{h\in [H]}  \in \big(\cap_{\mk \subseteq [K]}\mathcal{C}_{\mk}\big) \cap \mathcal{C}_{\repss} \Big)}{\Prob \Big((\varepsilon_{h,\tau},\bs{\varepsilon}_{h,\bs{x}}^\top)_{h\in [H]}  \in \mathcal{C}_{\repss} \Big)}\\
        &\geq \Prob \Big((\varepsilon_{h,\tau},\bs{\varepsilon}_{h,\bs{x}}^\top)_{h\in [H]}  \in  \cap_{\mk \subseteq [K]}\mathcal{C}_{\mk} \Big) =  \Prob_{\infty}(\hpL > \alpha).
    \end{align*}
    As a consequence, we have
    \[
    \Prob_{\infty}(\hpL \leq \alpha\mid \ma_{\repss}(\alpha_t)) \leq \Prob_{\infty}(\hpL \leq \alpha).
    \]

    Similarly, we can prove that
    \begin{align*}
     &\Prob_{\infty}(\hpL \leq \alpha\mid \ma_{\remss}(a)) \leq \Prob_{\infty}(\hpL \leq \alpha);\quad \Prob_{\infty}(\hpfe \leq \alpha\mid \ma_{\repfe}(\alpha_t)) \leq \Prob_{\infty}(\hpfe \leq \alpha).
    \end{align*}
\end{proof}

\begin{proof}[Proof of Theorem \ref{thm:re-mitigates-p-hacking-sre}]
    By \Cref{lem:re-mitigate-p-hacking-sre-(i)}, 
we have, under $H_{0\textsc{n}}$, 
   \begin{align*}
        &\Prob_{\infty}(\hpL \leq \alpha\mid \ma_{\repss}(\alpha_t)) \leq \Prob_{\infty}(\hpL \leq \alpha); \quad \Prob_{\infty}(\hpL \leq \alpha\mid \ma_{\remss}(a)) \leq \Prob_{\infty}(\hpL \leq \alpha),
    \end{align*}
and, under $H_{0\omega}$,
 \begin{align*}
            &\Prob_{\infty}(\hpfe \leq \alpha\mid \ma_{\repfe}(\alpha_t)) \leq \Prob_{\infty}(\hpfe \leq \alpha).
 \end{align*}

The remaining results follow from Theorem \ref{thm:inflation-bound-sre}, by letting $a\rightarrow 0$ and $\alpha_t\rightarrow 1$. 
\end{proof}

\subsection{Proof of Theorem \ref{thm:sufficient-complete-condition-for-orthogonal-equal-importance} $(i)$}
Let $V_{\tau\tau} = \sum_{h=1}^H \pi_{ h }  V_{ h,\tau\tau}$, $R^2_{h,\mk} = V_{h,\tau \mk} V_{h,\mk \mk}^{-1} V_{h, \mk \tau}/V_{h, \tau\tau}$, $R^2_{\mk} = \sum_{h=1}^H \pi_{ h } {R}^2_{ h, \mk} V_{ h,\tau\tau} /V_{\tau\tau}$, $\tilde{R}_{\mk}^2 = 
R^2_{\bs{x}} - R^2_{\mk}$ and $\bar{R}_{\bs{x}}^2 = 1-R_{\bs{x}}^2$. To prove Theorem \ref{thm:sufficient-complete-condition-for-orthogonal-equal-importance} $(i)$, we first prove a lemma.
 
\begin{lemma}
\label{lem:if-statement-for-rem-sre}
Consider SRE and assume Condition \ref{a:CLT-sre-for-rerandomization} holds. Under $H_{0\textsc{n}}$, if
    \begin{align*}
       a \leq  H^{-1}\bar{a}_{\rem}(\alpha, R^2_{\bs{x}}, \underline{\Delta}),
    \end{align*}
then we have $\Prob_{\infty}( \hpL \leq \alpha\mid \ma_{\remss}(a)) \leq \alpha.$
\end{lemma}

\begin{proof}[Proof of \Cref{lem:if-statement-for-rem-sre}]
Let $\bs{\beta}_{h,\mk}(z) = \bs{S}_{h,\mk\mk}^{-1}\bs{S}_{h,\mk Y(z)}$, $z=0,1$. Let $\tau_{\mk, i} = Y_i(1)-Y_i(0)- \bs{x}_i^\top(\bs{\beta}_{h(i),\mk}(1) - \bs{\beta}_{h(i),\mk}(0))$. Note that $\mathcal{N}_{\Lin,\mk} = (n^{1/2}\tilde{\se}_{\Lin,\mk})^{-1}\sum_{h=1}^H \pi_{ h }^{1/2}(\varepsilon_{h,\tau} -\bs{\varepsilon}_{ h, \mk}^\top\bs{\beta}_{ h, \Lin, \mk})$ and, by definition,
\begin{align*}
    n\tilde{\se}_{\Lin,\mk}^2 =& \sum_{h=1}^H \pi_h \big\{\var(\varepsilon_{h,\tau}-\bs{\varepsilon}_{ h, \mk}^\top\bs{\beta}_{ h, \Lin, \mk}) + S_{h,\tau_\mk \tau_\mk}\big\}  \geq  \sum_{h=1}^H \pi_h \var(\varepsilon_{h,\tau}-\bs{\varepsilon}_{ h, \mk}^\top\bs{\beta}_{ h, \Lin, \mk}) \\
    =& \sum_{h=1}^H \pi_{ h }(1-R^2_{ h, \mk}) V_{ h,\tau\tau} = (\bar{R}^2_{\bs{x}} + \tilde{R}^2_{\mk}) V_{ \tau\tau},
\end{align*}
where the inequality holds when $Y_i(1)=Y_i(0)$ for $i=1,\ldots,n$.

Define 
\[
\tilde{\mathcal{N}}_{\Lin,\mk} = \mathcal{N}_{\Lin,\mk} \frac{n^{1/2}\tilde{\se}_{\Lin,\mk}}{(\bar{R}^2_{\bs{x}} + \tilde{R}^2_{\mk})^{1/2} V_{ \tau\tau}^{1/2}},\quad \mk \subseteq [K].
\]
 $\tilde{\mathcal{N}}_{\Lin,\mk}$ are standard Gaussian random variables and $|\tilde{\mathcal{N}}_{\Lin,\mk}| \geq |\mathcal{N}_{\Lin,\mk}|$.

The remaining proof contains $2$ steps. 

\noindent\textbf{Step 1:} We prove that
\[
\Prob_{\infty}(\hpL \leq \alpha \mid \mathcal{A}_{\remss}(a))-\alpha \leq \Prob\Big(\max_{\mk \subsetneq [K]} |\tilde{\mathcal{N}}_{\Lin,\mk}| \geq z_{1-\alpha/2}, |\tilde{\mathcal{N}}_{\Lin,[K]}| < z_{1-\alpha/2}~\Big|~\mathcal{A}_{\remss}^{\infty}(a)\Big).
\]

Since $\varepsilon_{h,\tau} -  \bs{\varepsilon}_{ h, \bs{x}}^\top \bs{\beta}_{h, \Lin , \bs{x}}$ is independent of $\bs{\varepsilon}_{ h, \bs{x}}$, 
we have
\begin{align*}
\var(\varepsilon_{h,\tau}-\bs{\varepsilon}_{ h, \mk}^\top\bs{\beta}_{ h, \Lin, \mk}) - \var(\varepsilon_{h,\tau} -  \bs{\varepsilon}_{ h, \bs{x}}^\top \bs{\beta}_{h, \Lin , \bs{x}}) = \var(\bs{\varepsilon}_{ h, \bs{x}}^\top \bs{\beta}_{h, \Lin , \bs{x}}-\bs{\varepsilon}_{ h, \mk}^\top\bs{\beta}_{ h, \Lin, \mk}).
\end{align*}

By the definition of $R^2_{h,\mk}$ and $\bs{\beta}_{ h, \Lin, \mk}$, we have
\begin{align*}
    \var(\varepsilon_{h,\tau}-\bs{\varepsilon}_{ h, \mk}^\top\bs{\beta}_{ h, \Lin, \mk}) = V_{ h,\tau\tau}(1-R^2_{ h, \mk});\quad \var(\varepsilon_{h,\tau} -  \bs{\varepsilon}_{ h, \bs{x}}^\top \bs{\beta}_{h, \Lin , \bs{x}}) = V_{ h,\tau\tau}(1 - R^2_{ h, \bs{x}}),
\end{align*}
which yields that
\begin{align*}
 \var(\bs{\varepsilon}_{ h, \bs{x}}^\top \bs{\beta}_{h, \Lin , \bs{x}}-\bs{\varepsilon}_{ h, \mk}^\top\bs{\beta}_{ h, \Lin, \mk}) =  V_{ h,\tau\tau} (R^2_{ h, \bs{x}} - R^2_{ h, \mk}).
\end{align*}

Let $\tilde{R}^2_{ h, \mk} = R^2_{ h, \bs{x}} - R^2_{ h, \mk}$, $\bar{R}_{h,\bs{x}}^2 = 1 - R_{h,\bs{x}}^2$. Let $\mathcal{A}_{\remss}^{\infty}(a) = \cap_{h=1}^H\{\bs{\varepsilon}_{h,\bs{x}}^\top\bs{V}_{h,\bs{x}\bs{x}}^{-1}\bs{\varepsilon}_{h,\bs{x}}\leq a\}$. By \Cref{lem:asymptotic-limit-of-type-I-error-rate} and $|\mathcal{N}_{\Lin,\mk}|\leq |\tilde{\mathcal{N}}_{\Lin,\mk}|$, we have
\[
\Prob_{\infty}(\hpL \leq \alpha \mid \mathcal{A}_{\remss}(a)) = \Prob(\max_{\mk \subseteq [K]}|\mathcal{N}_{\Lin,\mk}|\geq z_{1-\alpha/2}\mid\mathcal{A}^\infty_{\remss}(a)) \leq  \Prob(\mathcal{E}\mid\mathcal{A}^\infty_{\remss}(a)),
\] 
where $\mathcal{E} = \{ \max_{\mk \subseteq [K]} |\tilde{\mathcal{N}}_{\Lin,\mk}| \geq z_{1-\alpha/2}\}$. 

We see that
\[
\mathcal{E} = \{ |\tilde{\mathcal{N}}_{\Lin,[K]}| \geq z_{1-\alpha/2}\} \cup \{ \max_{\mk \subsetneq [K]} |\tilde{\mathcal{N}}_{\Lin,\mk}| \geq z_{1-\alpha/2}, |\tilde{\mathcal{N}}_{\Lin,[K]}| < z_{1-\alpha/2}\} =: \mathcal{E}_1\cup\mathcal{E}_2.
\]

Therefore, 
\begin{align*}
\Prob_{\infty}(\hpL \leq \alpha \mid \mathcal{A}_{\remss}(a)) \leq \Prob(\mathcal{E}_1\mid\mathcal{A}_{\remss}^{\infty}(a))+\Prob(\mathcal{E}_2\mid\mathcal{A}_{\remss}^{\infty}(a)).
\end{align*}

Since $\tilde{\mathcal{N}}_{\Lin,[K]}$ is independent of $\bs{\varepsilon}_{h,\bs{x}}$, $h\in [H]$, we have
\[
\Prob(\mathcal{E}_1\mid\mathcal{A}_{\remss}^{\infty}(a)) = \Prob(\mathcal{E}_1) = \alpha,
\]
which yields that
\[
\Prob_{\infty}(\hpL \leq \alpha \mid \mathcal{A}_{\remss}(a))-\alpha \leq \Prob(\mathcal{E}_2\mid\mathcal{A}_{\remss}^{\infty}(a)).
\]

\noindent\textbf{Step 2:} We prove that: If $a \leq  H^{-1}\bar{a}_{\rem}(\alpha, R^2_{\bs{x}}, \underline{\Delta})$, we have
\[
\Prob\Big(\max_{\mk \subsetneq [K]} |\tilde{\mathcal{N}}_{\Lin,\mk}| \geq z_{1-\alpha/2}, |\tilde{\mathcal{N}}_{\Lin,[K]}| < z_{1-\alpha/2}~\Big|~\mathcal{A}_{\remss}^{\infty}(a)\Big) = 0.
\]

Let $\bs{\xi}_{ h } = \bs{V}_{ h, \bs{x}\bs{x}}^{-1/2}\bs{\varepsilon}_{ h, \bs{x}} $ and  $\varepsilon_0 = \tilde{\mathcal{N}}_{\Lin,[K]} $. They both are standard Gaussian 
random variables. We express $$\mathcal{A}_{\remss}^{\infty}(a) = \Big\{\max_{h\in [H]} \|\bs{\xi}_{ h }\|_2^2 \leq a\Big\}.$$ 

In the remainder of the proof, we will show that
\begin{align}
\label{eq:me-sufficient-condition}
    \mathcal{E}_2  \subset \Big\{|\varepsilon_0| < z_{1-\alpha/2}, \max_{h\in [H]} \|\bs{\xi}_{ h }\|_2 \geq  \frac{\bar{a}_{\rem}(\alpha, R^2_{\bs{x}}, \underline{\Delta})}{H}\Big\}.
\end{align}
By \eqref{eq:me-sufficient-condition}, it follows that if $a \leq \bar{a}_{\rem}(\alpha, R^2_{\bs{x}}, \underline{\Delta})/H$, 
$\Prob(\mathcal{E}_2\mid\mathcal{A}_{\remss}^{\infty}(a)) = 0.$

 Define $\tilde{\bs{\beta}}_{h,\Lin,\mk}\in \mathbb{R}^K$ such that $(\tilde{\bs{\beta}}_{h,\Lin,\mk})_{\mk} = \bs{\beta}_{h,\Lin,\mk}$ and all other entries are $0$.
Therefore, 
\[
\bs{\varepsilon}_{ h, \bs{x}}^\top \bs{\beta}_{h, \Lin , \bs{x}}-\bs{\varepsilon}_{ h, \mk}^\top\bs{\beta}_{ h, \Lin, \mk} = \bs{\varepsilon}_{ h, \bs{x}}^\top (\bs{\beta}_{h, \Lin , \bs{x}}-\tilde{\bs{\beta}}_{h,\Lin,\mk}) = \bs{\xi}_{ h }^\top \bs{V}_{h,\bs{x}\bs{x}}^{1/2} (\bs{\beta}_{h, \Lin , \bs{x}}-\tilde{\bs{\beta}}_{h,\Lin,\mk}),
\]
with
\begin{align*}
  \| \bs{V}_{h,\bs{x}\bs{x}}^{1/2} (\bs{\beta}_{h, \Lin , \bs{x}}-\tilde{\bs{\beta}}_{h,\Lin,\mk})\|_2^2 = \var(\bs{\varepsilon}_{ h, \bs{x}}^\top \bs{\beta}_{h, \Lin , \bs{x}}-\bs{\varepsilon}_{ h, \mk}^\top\bs{\beta}_{ h, \Lin, \mk} ) = V_{ h,\tau\tau} \tilde{R}^2_{h,\mk}.
\end{align*}

Therefore, we have $|\bs{\varepsilon}_{ h, \bs{x}}^\top \bs{\beta}_{h, \Lin , \bs{x}}-\bs{\varepsilon}_{ h, \mk}^\top\bs{\beta}_{ h, \Lin, \mk}| \leq \|\bs{\xi}_{ h }\|_2\tilde{R}_{ h, \mk}V_{ h,\tau\tau}^{1/2}$ and
\begin{align*}
    |\tilde{\mathcal{N}}_{\Lin,\mk}| &= \Big| \frac{\bar{R}_{\bs{x}} V_{\tau\tau}^{1/2}\varepsilon_0+\sum_{h=1}^H \pi_{ h }^{1/2}(\bs{\varepsilon}_{ h, \bs{x}}^\top \bs{\beta}_{h, \Lin , \bs{x}}-\bs{\varepsilon}_{ h, \mk}^\top\bs{\beta}_{ h, \Lin, \mk})}{(\bar{R}_{\bs{x}}^2 +\tilde{R}^2_{\mk})^{1/2}V_{\tau\tau}^{1/2}} \Big|\\
    &\leq \frac{\bar{R}_{\bs{x}} V_{\tau\tau}^{1/2}|\varepsilon_0|+\sum_{h=1}^H \pi_{ h }^{1/2}\|\bs{\xi}_{ h }\|_2\tilde{R}_{ h, \mk}V_{ h,\tau\tau}^{1/2}}{(\bar{R}_{\bs{x}}^2 +\tilde{R}^2_{\mk})^{1/2}V_{\tau\tau}^{1/2}}.
\end{align*}
By the Cauchy-Schwarz inequality,
\begin{align*}
    \sum_{h=1}^H \pi_{ h }^{1/2} \tilde{R}_{ h, \mk}V_{ h,\tau\tau}^{1/2} \leq \Big(\sum_{h=1}^H \pi_{ h } \tilde{R}^2_{h,\mk} V_{ h,\tau\tau}\Big)^{1/2} \sqrt{H} = V_{\tau\tau}^{1/2}\tilde{R}_{\mk}\sqrt{H}.
\end{align*}
Therefore, we have
\begin{align*}
    |\tilde{\mathcal{N}}_{\Lin,\mk}|\leq & \frac{\bar{R}_{\bs{x}} V_{\tau\tau}^{1/2}|\varepsilon_0|+\max_{h\in [H]}\|\bs{\xi}_{ h }\|_2 \sum_{h=1}^H \pi_{ h }^{1/2}\tilde{R}_{ h, \mk}V_{ h,\tau\tau}^{1/2}}{(V_{\tau\tau}\bar{R}_{\bs{x}}^2 +V_{\tau\tau}\tilde{R}^2_{\mk})^{1/2}}\\
    \leq & \frac{\bar{R}_{\bs{x}} V_{\tau\tau}^{1/2}|\varepsilon_0|+\max_{h\in [H]}\|\bs{\xi}_{ h }\|_2 V_{\tau\tau}^{1/2}\tilde{R}_{\mk}\sqrt{H}}{(V_{\tau\tau}\bar{R}_{\bs{x}}^2 +V_{\tau\tau}\tilde{R}^2_{\mk})^{1/2}}\\
    =&  \frac{\bar{R}_{\bs{x}} |\varepsilon_0|+\max_{h\in [H]}\|\bs{\xi}_{ h }\|_2 \tilde{R}_{\mk}\sqrt{H}}{(\bar{R}_{\bs{x}}^2 +\tilde{R}^2_{\mk})^{1/2}}.
\end{align*}
This implies that
\begin{align*}
\mathcal{E}_2 & \subset \Big\{|\varepsilon_0| < z_{1-\alpha/2},\max_{\mk \subsetneq [K]}\frac{\bar{R}_{\bs{x}} |\varepsilon_0|+\max_{h\in [H]} \|\bs{\xi}_{ h }\|_2\tilde{R}_{\mk}\sqrt{H}}{(\bar{R}_{\bs{x}}^2 +\tilde{R}^2_{\mk})^{1/2}} \geq z_{1-\alpha/2}\Big\},\\
   & \subset \Big\{|\varepsilon_0| < z_{1-\alpha/2},\max_{\mk \subsetneq [K]}\frac{\bar{R}_{\bs{x}} z_{1-\alpha/2}+\max_{h\in [H]} \|\bs{\xi}_{ h }\|_2\tilde{R}_{\mk}\sqrt{H}}{(\bar{R}_{\bs{x}}^2 +\tilde{R}^2_{\mk})^{1/2}} \geq z_{1-\alpha/2}\Big\},\\
    & = \Big\{|\varepsilon_0| < z_{1-\alpha/2}, \max_{h\in [H]} \|\bs{\xi}_{ h }\|_2 \geq  z_{1-\alpha/2} \min_{\mk \subsetneq [K]}\frac{(\bar{R}_{\bs{x}}^2 +\tilde{R}^2_{\mk})^{1/2}-\bar{R}_{\bs{x}} }{\sqrt{H}\tilde{R}_{\mk}}\Big\}\\
    & = \Big\{|\varepsilon_0| < z_{1-\alpha/2}, \max_{h\in [H]} \|\bs{\xi}_{ h }\|_2 \geq  z_{1-\alpha/2} \frac{(\bar{R}_{\bs{x}}^2 +\min_{k \in [K]} \tilde{R}^2_{[-k]})^{1/2}-\bar{R}_{\bs{x}} }{\sqrt{H} \min_{k \in [K]}\tilde{R}_{ [-k]}}\Big\},
\end{align*}
where the last equality holds because $\big\{(\bar{R}_{\bs{x}}^2 +\tilde{R}^2_{\mk})^{1/2}-\bar{R}_{\bs{x}} \big\}/(\sqrt{H}\tilde{R}_{\mk})$ is an increasing function of $\tilde{R}^2_{\mk}$ and $\min_{\mk \subsetneq [K]} \tilde{R}^2_{\mk} = \min_{k \in [K]}\tilde{R}_{ [-k]}^2$.

Since 
\begin{align*}
    & z_{1-\alpha/2}\frac{(\bar{R}_{\bs{x}}^2 +\min_{k \in [K]} \tilde{R}^2_{[-k]})^{1/2}-\bar{R}_{\bs{x}} }{\sqrt{H} \min_{k \in [K]}\tilde{R}_{ [-k]}} = z_{1-\alpha/2} \frac{\min_{k \in [K]} \tilde{R}_{ [-k]}/\sqrt{H}}{(\bar{R}_{\bs{x}}^2 + \min_{k \in [K]} \tilde{R}^2_{[-k]})^{1/2}+\bar{R}_{\bs{x}}} \\
    =& \Big\{\frac{\bar{a}_{\rem}(\alpha, R^2_{\bs{x}}, \underline{\Delta})}{H}\Big\}^{1/2},
\end{align*}
we have
\begin{align*}
    \mathcal{E}_2  \subset \Big\{|\varepsilon_0| < z_{1-\alpha/2}, \max_{h\in [H]} \|\bs{\xi}_{ h }\|_2 \geq  \frac{\bar{a}_{\rem}(\alpha, R^2_{\bs{x}}, \underline{\Delta})}{H}\Big\}.
\end{align*}

This completes the proof.
\end{proof}

\begin{proof}[Proof of Theorem \ref{thm:sufficient-complete-condition-for-orthogonal-equal-importance} $(i)$]
\noindent\textbf{The proof of the ``if" part:}

We use the same notation as in the proof of \Cref{lem:if-statement-for-rem-sre}. If Condition \ref{a:orthogonal-covariates-equally-importance} holds for $h\in [H]$, $k\in[K]$, we have
\begin{align*}
    \tilde{R}^2_{ h, [-k]} = \frac{1}{K} R^2_{h,\bs{x}}
\end{align*}
which yields that
\[
 \tilde{R}^2_{[-k]} =  \sum_{h=1}^H \pi_{ h } \tilde{R}^2_{ h, [-k]} V_{ h,\tau\tau}/V_{\tau\tau} =  \sum_{h=1}^H \pi_{ h } \frac{1}{K} R^2_{h,\bs{x}} V_{ h,\tau\tau}/V_{\tau\tau} = \frac{1}{K} R_{\bs{x}}^2.
\]

Therefore $\underline{\Delta} = \min_{k\in [K]}\tilde{R}^2_{[-k]} = K^{-1}R_{\bs{x}}^2$. The ``if" part follows from \Cref{lem:if-statement-for-rem-sre}.

\noindent\textbf{The proof of the ``only if" part:}

We then prove the ``only if" part, i.e., for any $R^2_{\bs{x}} \in (0,1)$, there exists $(Y_i(1),Y_i(0))_{i=1}^n \in \mathbb{M}_{\sss}(R^2_{\bs{x}})$ satisfying Condition \ref{a:orthogonal-covariates-equally-importance} such that for any $a > \bar{a}_{\rem}(\alpha, R^2_{\bs{x}}, R^2_{\bs{x}}/K)/H$, we have
\[
\Prob_{\infty}(\hpL \leq \alpha \mid \ma_{\remss}(a)) > 0.
\]

We construct $(Y_i(1),Y_i(0))_{i=1}^n \in \mathbb{M}_{\sss}(R^2_{\bs{x}})$ as follows, which  satisfies Condition \ref{a:orthogonal-covariates-equally-importance}:  $(i)$ $Y_i(1)=Y_i(0) $ for $i\in [n]$ and $(ii)$  for $h \in [H]$, $\bs{V}_{h, \bs{x}\bs{x}} = \bs{I}_{K}$, $\pi_{h}^{1/2}\bs{V}_{h, \bs{x}\tau}  = \bs{1}_{K}$, $V_{h,\tau\tau} = (\pi_{h}R^2_{\bs{x}})^{-1}K$. 

Then for any $\mk \subseteq [K]$, $h\in [H]$, $\pi_h \bs{V}_{h,\tau\mk} \bs{V}_{h,\mk\mk}^{-1}\bs{V}_{h,\mk \tau} = |\mk|$, where $|\mk|$ is the size of the set $\mk$.  We can verify that for any $\mk \subseteq [K]$, $h\in [H]$,
\begin{align}
\label{eq:eq1}
R^2_{ h, \mk} = R^2_{\bs{x}}|\mk|/K,~ \pi_h\tilde{R}^2_{ h, \mk}V_{ h,\tau\tau} = K - |\mk|,~  V_{\tau\tau} = \sum_{h\in [H]}\pi_h V_{h,\tau\tau} = HK/R^2_{\bs{x}},~ \tilde{R}^2_{\mk} = \frac{K-|\mk|}{K}R^2_{\bs{x}}.
\end{align}

Recall the notation $\mathcal{E}$, $\mathcal{E}_1$ and $\mathcal{E}_2$ in the proof of \Cref{lem:if-statement-for-rem-sre}. Since $Y_i(1)=Y_i(0) $, for $i\in [n]$, we have $\mathcal{N}_{\Lin,\mk} = \tilde{\mathcal{N}}_{\Lin,\mk}$, and we have
\begin{align*}
\Prob_{\infty}(\hpL \leq \alpha \mid \mathcal{A}_{\remss}(a)) =& \Prob(\mathcal{E}\mid\mathcal{A}_{\remss}^{\infty}(a)) = \Prob(\mathcal{E}_1\mid\mathcal{A}_{\remss}^{\infty}(a))+\Prob(\mathcal{E}_2\mid\mathcal{A}_{\remss}^{\infty}(a))\\
=& \alpha + \Prob(\mathcal{E}_2\mid\mathcal{A}_{\remss}^{\infty}(a)).
\end{align*}

It remains to show that $\Prob(\mathcal{E}_2\cap\mathcal{A}_{\remss}^{\infty}(a))>0$. We treat $\mathcal{E}_2\cap\mathcal{A}_{\remss}^{\infty}(a)$ as a measurable set of $(\varepsilon_0,(\bs{\xi}_{h})_{h\in [H]} )$.  Since nonempty open set has measure greater than $0$, it suffices to show that the following open subset $\mathcal{E}_3\subseteq \mathcal{E}_2\cap\mathcal{A}_{\remss}^{\infty}(a)$ is nonempty: 
\[
\mathcal{E}_3 =\{|\mathcal{N}_{\Lin, [K]}|< z_{1-\alpha/2}, \max_{\mk \subsetneq [K]} |\mathcal{N}_{\Lin, \mk}| > z_{1-\alpha/2} \} \cap \Big\{\max_{h\in [H]} \|\bs{\xi}_h\|_2^2 <a\Big\}. 
\]

We see that 
 \begin{align*}
    |\mathcal{N}_{\Lin,\mk}| =& \frac{V_{\tau\tau}^{1/2}|\bar{R}_{\bs{x}}\varepsilon_0 + \sum_{h=1}^H \pi_{h}^{1/2}\bs{\varepsilon}_{h,\bs{x}}^\top (\bs{\beta}_{h, \Lin , \bs{x}}-\tilde{\bs{\beta}}_{h,\Lin,\mk})|}{(\bar{R}^2_{\bs{x}}+\tilde{R}_{ \mk}^2)^{1/2}V_{\tau\tau}^{1/2}} \\
    =& \frac{\Big|\bar{R}_{\bs{x}}V_{\tau\tau}^{1/2}|\varepsilon_0|+\sum_{h=1}^H \pi_{h}^{1/2} s_{h,\mk}\|\bs{\xi}_h\|_2 \tilde{R}_{h,\mk}V_{h,\tau\tau}^{1/2}\Big|}{V_{\tau\tau}^{1/2}(\bar{R}_{\bs{x}}^2 + \tilde{R}^2_{\mk})^{1/2}}, 
 \end{align*}
where $s_{h,\mk} = \operatorname{sign}(\varepsilon_{0})\bs{\varepsilon}_{h,\bs{x}}^\top (\bs{\beta}_{h, \Lin , \bs{x}}-\tilde{\bs{\beta}}_{h,\Lin,\mk})/\{\|\bs{\xi}_{h}\|_2\tilde{R}_{h,\mk}V_{h,\tau\tau}^{1/2}\}$.

Define 
\begin{align*}
    \psi(|\varepsilon_0|) = \frac{(\bar{R}_{\bs{x}}^2 + K^{-1} R^2_{\bs{x}})^{1/2}z_{1-\alpha/2}-\bar{R}_{\bs{x}}|\varepsilon_0|}{\sqrt{H K^{-1} R^2_{\bs{x}}}}.
\end{align*}

Under the event $\{s_{h,[-1]}=1, h \in [H]\}$, we have
\begin{align*}
    \max_{\mk \subsetneq [K]} |\mathcal{N}_{\Lin, \mk}| &\geq |\mathcal{N}_{\Lin, [-1]}| = \frac{\Big|\bar{R}_{\bs{x}}V_{\tau\tau}^{1/2}|\varepsilon_0|+\sum_{h=1}^H \pi_{h}^{1/2}\|\bs{\xi}_h\|_2 \tilde{R}_{h,[-1]}V_{h,\tau\tau}^{1/2}\Big|}{V_{\tau\tau}^{1/2}(\bar{R}_{\bs{x}}^2 + \tilde{R}^2_{[-1]})^{1/2}}\\
    &\geq \frac{\bar{R}_{\bs{x}}|\varepsilon_0|+\min_{h\in [H]} \|\bs{\xi}_h\|_2(H/K)^{1/2}R_{\bs{x}}}{(\bar{R}_{\bs{x}}^2 + K^{-1}R^2_{\bs{x}})^{1/2}},
\end{align*}
where the last inequality is due to the second and third equalities in \eqref{eq:eq1}.

Therefore,
\begin{align*}
    & \mathcal{E}_3 \cap \{s_{h,[-1]}=1, h \in [H]\} \\
    & \qquad \supseteq   \Big\{|\varepsilon_0| < z_{1-\alpha/2},s_{h, [-1]}=1, h\in [H], \frac{\bar{R}_{\bs{x}}|\varepsilon_0|+\min_{h\in [H]} \|\bs{\xi}_h\|_2(H/K)^{1/2}R_{\bs{x}}}{(\bar{R}_{\bs{x}}^2 + K^{-1}R^2_{\bs{x}})^{1/2}} > z_{1-\alpha/2},\\
    & \qquad \qquad \qquad\qquad\qquad\qquad\qquad\qquad\qquad\qquad\qquad\qquad\qquad \qquad\qquad\qquad \max_{h\in [H]} \|\bs{\xi}_h\|_2^2 < a \Big\}\\
    & \qquad = \Big\{|\varepsilon_0| < z_{1-\alpha/2},s_{h, [-1]}=1, h\in [H], \min_{h\in [H]} \|\bs{\xi}_h\|_2 >  \psi(|\varepsilon_0|), \max_{h\in [H]} \|\bs{\xi}_h\|_2^2 < a \Big\} .
\end{align*}

Since
\begin{align*} 
    \psi(z_{1-\alpha/2})= \bar{a}_{\rem}(\alpha,R^2_{\bs{x}},R^2_{\bs{x}}/K)^{1/2}/H^{1/2}  < a^{1/2},
\end{align*}
and $\psi$ is a continuous function, there exists $0 < t < z_{1-\alpha/2}$ such that $\psi(t) < a^{1/2}$. Therefore, 
\begin{align*}
    & \mathcal{E}_3 \cap \{s_{h,[-1]}=1, h \in [H]\} \\
    & \qquad \supseteq \Big\{|\varepsilon_0| = t ,s_{h, [-1]}=1, h\in [H], \min_{h\in [H]} \|\bs{\xi}_h\|_2 >  \psi(t), \max_{h\in [H]} \|\bs{\xi}_h\|_2^2 < a \Big\}.
\end{align*}

Obviously, $\mathcal{E}_3$ is  nonempty. The conclusion follows.
\end{proof}

\subsection{Proof of Theorem \ref{thm:sufficient-complete-condition-for-orthogonal-equal-importance} $(ii)$}
\begin{proof}[Proof of \Cref{thm:sufficient-complete-condition-for-orthogonal-equal-importance} $(ii)$]
\noindent \textbf{The proof of the ``if" part:}
    Let $\mathcal{A}_{\repss}^{\infty}(\alpha_t) = \{\max_{h\in [H]}$ $ \|\bs{\sigma}(\bs{V}_{h,\bs{x}\bs{x}})^{-1}\bs{\varepsilon}_{h,\bs{x}}\|_\infty\leq z_{1-\alpha_t/2}\}$. Recall the definition of $\mathcal{E}_2$ in the proof of \Cref{lem:if-statement-for-rem-sre}.

We have
\begin{align*}
\Prob_{\infty}(\hpL \leq \alpha \mid \mathcal{A}_{\repss}(\alpha_t))-\alpha \leq \Prob(\mathcal{E}_2\mid\mathcal{A}_{\repss}^{\infty}(\alpha_t)).
\end{align*}

We next prove, 
\begin{align}
\label{eq:me-sufficient-condition-rep}
\mathcal{E}_2  \subset \Big\{|\varepsilon_0| < z_{1-\alpha/2},\max_h \|\bs{\sigma}(\bs{V}_{h,\bs{x}\bs{x}})^{-1/2}\bs{\varepsilon}_{h,\bs{x}}\|_\infty \geq \frac{c_{\rep}(R_{\bs{x}}^2)}{H^{1/2}}g(\alpha)\Big\}.  
\end{align}

By \eqref{eq:me-sufficient-condition-rep}, if
\[
z_{1-\alpha_t/2} \leq \frac{c_{\rep}(R_{\bs{x}}^2)}{H^{1/2}}g(\alpha), \quad \text{i.e.},\quad \alpha_t \geq g^{-1}\Big(\frac{c_{\rep}(R_{\bs{x}}^2)}{H^{1/2}}g(\alpha)\Big), 
\]
we have
\[
\Prob(\mathcal{E}_2\mid\mathcal{A}_{\repss}^{\infty}(\alpha_t)) = 0.
\]
As a consequence, we will have $\Prob_{\infty}(\hpL \leq \alpha \mid \mathcal{A}_{\repss}(\alpha_t)) \leq \alpha$.

Using that 
\begin{align*}
    |\tilde{\mathcal{N}}_{\Lin, \mk}| &= \frac{\big|\bar{R}_{\bs{x}} V_{\tau\tau}^{1/2}\varepsilon_0+\sum_{h=1}^H \pi_{ h }^{1/2}\bs{\varepsilon}_{ h, \bs{x}}^\top (\bs{\beta}_{h, \Lin , \bs{x}}-\tilde{\bs{\beta}}_{ h, \Lin, \mk})\big|}{(V_{\tau\tau}\bar{R}_{\bs{x}}^2 +V_{\tau\tau}\tilde{R}^2_{\mk})^{1/2}}\\
    & \leq \frac{\bar{R}_{\bs{x}}\vtt^{1/2}|\varepsilon_0|+\sum_{h}\pi_{h}^{1/2}\|\bs{\sigma}(\bs{V}_{h,\bs{x}\bs{x}})(\bs{\beta}_{h, \Lin , \bs{x}}-\tilde{\bs{\beta}}_{h,\Lin,\mk})\|_1\max_{h}\|\bs{\sigma}(\bs{V}_{h,\bs{x}\bs{x}})^{-1}\bs{\varepsilon}_{h,\bs{x}}\|_\infty}{\vtt^{1/2}(\bar{R}_{\bs{x}}^2 + \tilde{R}^2_{\mk})^{1/2}}
\end{align*}
we have
\begin{align*}
    \mathcal{E}_2 &\subseteq  \Big\{|\varepsilon_0|< z_{1-\alpha/2},\\
    &\max_{\mk \subsetneq [K]}\frac{\bar{R}_{\bs{x}}\vtt^{1/2}|\varepsilon_0|+\max_{h}\|\bs{\sigma}(\bs{V}_{h,\bs{x}\bs{x}})^{-1}\bs{\varepsilon}_{h,\bs{x}}\|_\infty\sum_{h}\pi_{h}^{1/2}\|\bs{\sigma}(\bs{V}_{h,\bs{x}\bs{x}})(\bs{\beta}_{h, \Lin , \bs{x}}-\tilde{\bs{\beta}}_{h,\Lin,\mk})\|_1}{\vtt^{1/2}(\bar{R}_{\bs{x}}^2 + \tilde{R}^2_{\mk})^{1/2}} \geq z_{1-\alpha/2}\Big\}
\\
&\subseteq  \Big\{|\varepsilon_0|< z_{1-\alpha/2},\\
&\max_{\mk \subsetneq [K]}\frac{\bar{R}_{\bs{x}}\vtt^{1/2}z_{1-\alpha/2}+\max_{h}\|\bs{\sigma}(\bs{V}_{h,\bs{x}\bs{x}})^{-1}\bs{\varepsilon}_{h,\bs{x}}\|_\infty\sum_{h}\pi_{h}^{1/2}\|\bs{\sigma}(\bs{V}_{h,\bs{x}\bs{x}})(\bs{\beta}_{h, \Lin , \bs{x}}-\tilde{\bs{\beta}}_{h,\Lin,\mk})\|_1}{\vtt^{1/2}(\bar{R}_{\bs{x}}^2 + \tilde{R}^2_{\mk})^{1/2}} \geq z_{1-\alpha/2}\Big\}\\
    & =  \Big\{|\varepsilon_0|< z_{1-\alpha/2}, \max_{h}\|\bs{\sigma}(\bs{V}_{h,\bs{x}\bs{x}})^{-1}\bs{\varepsilon}_{h,\bs{x}}\|_\infty \geq z_{1-\alpha/2}\min_{\mk \subsetneq [K]}\frac{\vtt^{1/2}\{(\bar{R}_{\bs{x}}^2 + \tilde{R}^2_{\mk})^{1/2}-\bar{R}_{\bs{x}}\}}{\sum_{h}\pi_{h}^{1/2}\|\bs{\sigma}(\bs{V}_{h,\bs{x}\bs{x}})(\bs{\beta}_{h, \Lin , \bs{x}}-\tilde{\bs{\beta}}_{h,\Lin,\mk})\|_1}\Big\}.
\end{align*}

Next, we calculate the term $\sum_{h}\pi_{h}^{1/2}\|\bs{\sigma}(\bs{V}_{h,\bs{x}\bs{x}})(\bs{\beta}_{h, \Lin , \bs{x}}-\tilde{\bs{\beta}}_{h,\Lin,\mk})\|_1$.

If Condition \ref{a:orthogonal-covariates-equally-importance} holds for every stratum, we have
\begin{align*}
    \bs{\beta}_{h, \Lin , \bs{x}} = \bs{V}_{h,\bs{x}\bs{x}}^{-1} \bs{V}_{h,\bs{x}\tau} = \Big(\frac{V_{h,k\tau}}{{V}_{h,kk}}\Big)_{k\in [K]} . 
\end{align*}
By the definition of $R_{h,k}^2$, we have $V_{h,\tau\tau} R_{h,k}^2 = V_{h,kk}^{-1}V_{h,k\tau}^2.$
Therefore, we have
\[
\bs{\sigma}(\bs{V}_{h,\bs{x}\bs{x}})\bs{\beta}_{h, \Lin , \bs{x}} = \Big(\frac{V_{h,k\tau}V_{h,kk}^{1/2}}{{V}_{h,kk}}\Big)_{k\in [K]}   = V_{h,\tau\tau}^{1/2}R_{h,1} (\operatorname{sign}(V_{h,k\tau}))_{k\in [K]} .
\]
Similarly, we have $\bs{\sigma}(\bs{V}_{h,\bs{x}\bs{x}})\tilde{\bs{\beta}}_{h,\Lin,\mk} = V_{h,\tau\tau}^{1/2} R_{h,1} (\operatorname{sign}(V_{h,k\tau})I(k\in \mk))_{k\in [K]} $. Thus
\begin{align}
\label{eq:sigma-Vhxx-beta-h-lin-beta-h-lin-mk}
    \bs{\sigma}(\bs{V}_{h,\bs{x}\bs{x}})(\bs{\beta}_{h, \Lin , \bs{x}}-\tilde{\bs{\beta}}_{h,\Lin, \mk}) = V_{h,\tau\tau}^{1/2} R_{h,1} (\operatorname{sign}(V_{h,k\tau})I(k\not\in \mk))_{k\in [K]}.
\end{align}

On the other hand, since Condition \ref{a:orthogonal-covariates-equally-importance} holds for every stratum, we have $R^2_{h,\bs{x}} = K R_{h,1}^2$, $R^2_{h,\mk} = |\mk| R_{h,1}^2$, $\tilde{R}^2_{h,\mk} = (K-|\mk|) R_{h,1}^2$. Therefore, we have
\begin{align*}
        &\sum_{h} \pi_{h}^{1/2}\|\bs{\sigma}(\bs{V}_{h,\bs{x}\bs{x}})(\bs{\beta}_{h, \Lin , \bs{x}}-\tilde{\bs{\beta}}_{h,\Lin, \mk})\|_1 = \sum_{h}\pi_h^{1/2} (K-|\mk|)V_{h,\tau\tau}^{1/2}R_{h,1}\\
       = & (K-|\mk|)^{1/2} \sum_{h}\pi_h^{1/2} V_{h,\tau\tau}^{1/2}\tilde{R}_{h,\mk} \leq  (K-|\mk|)^{1/2} H^{1/2} \big(\sum_{h}\pi_h V_{h,\tau\tau}\tilde{R}_{h,\mk}^2\big)^{1/2} \\
       = & (K-|\mk|)^{1/2} H^{1/2} V_{\tau\tau}^{1/2}\tilde{R}_{\mk}.
    \end{align*}

   It follows that
    \begin{align*}
       &\min_{\mk \subsetneq [K]}\frac{\vtt^{1/2}\{(\bar{R}_{\bs{x}}^2 + \tilde{R}^2_{\mk})^{1/2}-\bar{R}_{\bs{x}}\}}{\sum_{h}\pi_{h}^{1/2}\|\bs{\sigma}(\bs{V}_{h,\bs{x}\bs{x}})(\bs{\beta}_{h, \Lin , \bs{x}}-\tilde{\bs{\beta}}_{h,\Lin,\mk})\|_1} \geq \min_{\mk\subsetneq [K]}\frac{(\bar{R}_{\bs{x}}^2 + \tilde{R}^2_{\mk})^{1/2}-\bar{R}_{\bs{x}}}{\sqrt{(K-|\mk|)H}\tilde{R}_{\mk}} 
      \\
      =&  \min_{\mk \subsetneq [K]} \frac{(KH)^{-1/2}R_{\bs{x}}}{(\bar{R}_{\bs{x}}^2 + \tilde{R}^2_{\mk})^{1/2}+\bar{R}_{\bs{x}}} = \frac{(KH)^{-1/2}R_{\bs{x}}}{1+\bar{R}_{\bs{x}}},
    \end{align*}
where the second to last equality holds because $(K-|\mk|)^{-1} \tilde{R}_{\mk}^2 = K^{-1}  R_{\bs{x}}^2$ and the last equality holds because $\min_{\mk \subsetneq [K]} \tilde{R}^2_{\mk} = \tilde{R}^2_{\emptyset}=R_{\bs{x}}^2$.

   Thus, 
\begin{align*}
\mathcal{E}_2  &\subset \Big\{|\varepsilon_0| < z_{1-\alpha/2},\max_h \|\bs{\sigma}(\bs{V}_{h,\bs{x}\bs{x}})^{-1/2}\bs{\varepsilon}_{h,\bs{x}}\|_\infty \geq z_{1-\alpha/2}\frac{(HK)^{-1/2}R_{\bs{x}}}{1+\bar{R}_{\bs{x}}}\Big\}\\
&=\Big\{|\varepsilon_0| < z_{1-\alpha/2},\max_h \|\bs{\sigma}(\bs{V}_{h,\bs{x}\bs{x}})^{-1/2}\bs{\varepsilon}_{h,\bs{x}}\|_\infty \geq \frac{c_{\rep}(R_{\bs{x}}^2)}{H^{1/2}}g(\alpha)\Big\}.  
\end{align*}

We complete the proof of ``if" part.

\noindent \textbf{The proof of the ``only if" part:}

To prove the ``only if" part, we show that if $\alpha_t < g^{-1}\big(H^{-1/2}c_{\rep}(R^2_{\bs{x}} )g(\alpha)\big)$, there exists $(Y_i(1),Y_i(0))_{i=1}^n \in \mathbb{M}_{\sss}(R^2_{\bs{x}})$ such that $\Prob_{\infty}(\hpL\mid \ma_{\repss}(\alpha_t)) > 0$.

Consider the same construction as we prove the ``only if" part in Theorem \ref{thm:sufficient-complete-condition-for-orthogonal-equal-importance} (i), i.e.,
$(i)$ $Y_i(1) = Y_i(0)$ for $i \in [n]$ such that $\tilde{\mathcal{N}}_{\Lin,\mk} = \mathcal{N}_{\Lin, \mk}$; $(ii)$  for $h \in [H]$, $\bs{V}_{h, \bs{x}\bs{x}} = \bs{I}_{K}$, $\pi_{h}^{1/2}\bs{V}_{h, \bs{x}\tau} = \bs{1}_{K}$ and $V_{h,\tau\tau} = (\pi_{h}R^2_{\bs{x}})^{-1}K$. 

Since
\begin{align*}
\Prob_{\infty}(\hpL \leq \alpha \mid \mathcal{A}_{\repss}(\alpha_t))-\alpha = \Prob(\mathcal{E}_2\mid\mathcal{A}_{\repss}^{\infty}(\alpha_t)),
\end{align*}
it remains to show that if $\alpha_t < g^{-1}\big(H^{-1/2}c_{\rep}(R^2_{\bs{x}} )g(\alpha)\big)$, $\Prob(\mathcal{E}_2\mid\mathcal{A}_{\repss}^{\infty}(\alpha_t)) > 0.$

Since $Y_i(1)=Y_i(0)$, for $i \in [n]$, we have $\mathcal{N}_{\Lin, \mk} = \tilde{\mathcal{N}}_{\Lin, \mk}$.  By \eqref{eq:sigma-Vhxx-beta-h-lin-beta-h-lin-mk}, we have
\begin{align*}
    \mathcal{N}_{\Lin, \mk} =& \frac{\bar{R}_{\bs{x}}\vtt^{1/2}\varepsilon_0+\sum_{h=1}^H \pi_{h}^{1/2} (\bs{\beta}_{h, \Lin , \bs{x}}-\tilde{\bs{\beta}}_{h,\Lin,\mk})\bs{\varepsilon}_{h,\bs{x}}}{\vtt^{1/2}(\bar{R}_{\bs{x}}^2 + \tilde{R}^2_{\mk})^{1/2}}\\
    =& \frac{\bar{R}_{\bs{x}}\vtt^{1/2}\varepsilon_0+\sum_{h=1}^H \pi_{h}^{1/2} \sum_{k\not\in \mk}V_{h,\tau\tau}^{1/2}R_{h,1}\sign(V_{h,k \tau})\varepsilon_{h,k}/V_{h,kk}}{\vtt^{1/2}(\bar{R}_{\bs{x}}^2 + \tilde{R}^2_{\mk})^{1/2}}.
\end{align*}

 Let $s_{h,k} = \sign(\varepsilon_0)\sign(\varepsilon_{h,k})\sign(V_{h,k \tau})$, we have
\begin{align*}
    |\mathcal{N}_{\Lin, \mk}| =\frac{\big|\bar{R}_{\bs{x}}\vtt^{1/2}|\varepsilon_0|+\sum_{h=1}^H \pi_{h}^{1/2} \sum_{k\not\in \mk}V_{h,\tau\tau}^{1/2}R_{h,1} s_{h,k}|\varepsilon_{h,k}/V_{h,kk}|\big|}{\vtt^{1/2}(\bar{R}_{\bs{x}}^2 + \tilde{R}^2_{\mk})^{1/2}}.
\end{align*}
Since the strata are independent  and $\bs{V}_{h,\bs{x}\bs{x}}$ are diagonal,  $\varepsilon_0$ and $\varepsilon_{h,k}$, $h\in [H], k\in [K]$ are independent. Therefore, $s_{h,k}$ are independent and take values in $\{-1,1\}$ with equal probability. 

Under $\{s_{h,k}=1, h\in [H], k \in [K]\}$, we have
\begin{align*}
    \max_{\mk \subseteq [K]} |\mathcal{N}_{\Lin, \mk}| \geq & |\mathcal{N}_{\Lin, \emptyset}| = \frac{\bar{R}_{\bs{x}}\vtt^{1/2}|\varepsilon_0|+\sum_{h=1}^H \pi_{h}^{1/2} \sum_{k=1}^K V_{h,\tau\tau}^{1/2}R_{h,1} |\varepsilon_{h,k}/V_{h,kk}|}{\vtt^{1/2}}\\
    \geq &  \frac{\bar{R}_{\bs{x}}\vtt^{1/2}|\varepsilon_0|+ \min_{h,k}|\varepsilon_{h,k}/V_{h,kk}| \sum_{h=1}^H \pi_{h}^{1/2} \sum_{k=1}^K V_{h,\tau\tau}^{1/2}R_{h,1} }{\vtt^{1/2}}\\
    = & \bar{R}_{\bs{x}}|\varepsilon_0|+ \min_{h,k}|\varepsilon_{h,k}/V_{h,kk}| R_{\bs{x}} (HK)^{1/2},
\end{align*}
where the last equality holds because by our construction, $\sum_{h=1}^H \pi_{h}^{1/2} \sum_{k=1}^K V_{h,\tau\tau}^{1/2}R_{h,1} = \sum_{h=1}^H \pi_{h}^{1/2} K^{1/2} V_{h,\tau\tau}^{1/2}R_{h,\bs{x}} = V_{\tau\tau}^{1/2}R_{\bs{x}} (HK)^{1/2}.$ 

Therefore, we have
\begin{align*}
 \mathcal{E}_2 &\cap \{s_{h,k}=1, h\in [H], k \in [K]\} \\
\supseteq &  \Big\{|\varepsilon_0| < z_{1-\alpha/2},s_{h,k}=1, h\in [H], k \in [K], \bar{R}_{\bs{x}}|\varepsilon_0|+ \min_{h,k}|\varepsilon_{h,k}/V_{h,kk}| R_{\bs{x}} (HK)^{1/2} \geq z_{1-\alpha/2}\Big\} \\
=  & \Big\{|\varepsilon_0| < z_{1-\alpha/2},s_{h,k}=1, h\in [H], k \in [K],  
\min_{h,k}|\varepsilon_{h,k}/V_{h,kk}|  \geq \frac{z_{1-\alpha/2} - \bar{R}_{\bs{x}}|\varepsilon_0| }{R_{\bs{x}} (HK)^{1/2}} \Big\} =: \mathcal{E}_3.
\end{align*}

Let
\[
\psi(|\varepsilon_0|) = \frac{z_{1-\alpha/2} - \bar{R}_{\bs{x}}|\varepsilon_0| }{R_{\bs{x}} (HK)^{1/2}},\quad  \mathcal{E}_4 = \Big\{|\varepsilon_0| < z_{1-\alpha/2},
\min_{h,k}|\varepsilon_{h,k}/V_{h,kk}|  \geq \psi(|\varepsilon_0|) \Big\},
\]
with 
\[
\mathcal{E}_4 \cap \ma^{\infty}_{\repss}(\alpha_t) = \Big\{|\varepsilon_0| < z_{1-\alpha/2},
\min_{h,k}|\varepsilon_{h,k}/V_{h,kk}|  \geq \psi(|\varepsilon_0|), \max_{h,k}|\varepsilon_{h,k}/V_{h,kk}| \leq z_{1-\alpha_t/2} \Big\}.
\]

Since
$\alpha_t < g^{-1}\big(H^{-1/2}c_{\rep}(R_{\bs{x}}^2)g(\alpha)\big)$, we have
\[
z_{1-\alpha_t/2} > \frac{z_{1-\alpha/2} (1- \bar{R}_{\bs{x}}) }{R_{\bs{x}} (HK)^{1/2}}  = \psi(z_{1-\alpha/2}).
\]

Since $\psi(x)$ is a continuous function of $x$, there exists $0 < x_0 < z_{1-\alpha/2}$ such that $\psi(x_0) < z_{1-\alpha_t/2}$.

Therefore, we have
\[
\mathcal{E}_4 \cap \ma^{\infty}_{\repss}(\alpha_t) \supseteq \Big\{|\varepsilon_0| = x_0,
 \min_{h,k}|\varepsilon_{h,k}/V_{h,kk}|  \geq \psi(x_0), \max_{h,k}|\varepsilon_{h,k}/V_{h,kk}| \leq z_{1-\alpha_t/2} \Big\}
\]
which is nonempty. It follows that
\[
\Prob(\mathcal{E}_4 \cap \ma^{\infty}_{\repss}(\alpha_t)) >0.
\]

Noticing that $s_{h,k}$ and $|\varepsilon_0|$ and $|\varepsilon_{h,k}|$ are independent and $\mathcal{E}_3 = \mathcal{E}_4 \cap \big\{s_{h,k}=1, h\in [H], k \in [K]\big\}$, it follows that
\begin{align*}
    &\Prob(\mathcal{E}_2 \cap \ma^{\infty}_{\repss}(\alpha_t) ) \geq 
 \Prob( \mathcal{E}_3 \cap \ma^{\infty}_{\repss}(\alpha_t))
 =  \Prob(\mathcal{E}_4 \cap \ma^{\infty}_{\repss}(\alpha_t)) \Prob(s_{h,k}=1, h\in [H], k \in [K]) > 0.
\end{align*}
Therefore $\Prob(\mathcal{E}_2 \mid \ma^{\infty}_{\repss}(\alpha_t) ) >0 $.

The conclusion follows.
\end{proof}

\subsection{Proof of Theorem \ref{thm:inflation-bound-sre}}
\begin{proof}[Proof of Theorem \ref{thm:inflation-bound-sre}]
\noindent \textbf{Proof of \Cref{thm:inflation-bound-sre} (i) and (ii):}
   By \Cref{lem:asymptotic-limit-of-type-I-error-rate}, we have
  \begin{align*}
        &\Prob_{\infty}( \hpL \leq \alpha\mid \ma_{\remss}(a)) = \Prob ( \mathcal{E} \mid \ma_{\remss}^\infty(a) ),\quad \Prob_{\infty}( \hpL \leq \alpha\mid \ma_{\repss}(\alpha_t)) = \Prob(\mathcal{E}\mid \ma_{\repss}^\infty(\alpha_t)),
    \end{align*}
    where
    \[
    \mathcal{E} = \Big\{\max_{\mk \subseteq [K]} |\mathcal{N}_{\Lin,\mk}| \geq z_{1-\alpha/2}\Big\}.
    \]
We have shown in the proof of  \Cref{lem:if-statement-for-rem-sre} that 
\begin{align*}
  |\mathcal{N}_{\Lin,\mk}| \leq  |\tilde{\mathcal{N}}_{\Lin,\mk}| \leq \frac{\bar{R}_{\bs{x}} V_{\tau\tau}^{1/2}|\varepsilon_0|+\sum_{h=1}^H \pi_{ h }^{1/2}\|\bs{\xi}_{ h }\|_2\tilde{R}_{ h, \mk}V_{ h,\tau\tau}^{1/2}}{(\bar{R}_{\bs{x}}^2 +\tilde{R}^2_{\mk})^{1/2}V_{\tau\tau}^{1/2}}.
\end{align*}

Noticing that $V_{\tau\tau}\bar{R}_{\bs{x}}^2 +V_{\tau\tau}\tilde{R}^2_{\mk} = V_{\tau\tau}\bar{R}^2_{\bs{x}} +\sum_{h=1}^H \pi_{ h }\tilde{R}_{ h, \mk}^2V_{ h,\tau\tau}$ and by Cauchy-Schwarz inequality, we have
\begin{align*}
    &\frac{\bar{R}_{\bs{x}} V_{\tau\tau}^{1/2}|\varepsilon_0|+\sum_{h=1}^H \pi_{ h }^{1/2}\|\bs{\xi}_{ h }\|_2\tilde{R}_{ h, \mk}V_{ h,\tau\tau}^{1/2}}{(V_{\tau\tau}\bar{R}_{\bs{x}}^2 +V_{\tau\tau}\tilde{R}^2_{\mk})^{1/2}} \leq (\varepsilon_0^2+\sum_{h}\|\bs{\xi}_{ h }\|_2^2)^{1/2}.
\end{align*}

It follows that
\begin{align}
\label{eq:thm-S.3-1}
    \mathcal{E} \subseteq \Big\{(\varepsilon_0^2+\sum_{h}\|\bs{\xi}_{ h }\|_2^2)^{1/2} \geq z_{1-\alpha/2}\Big\}.
\end{align}
Combining this with $\ma_{\remss}^\infty(a) = \{\max_{h}\|\bs{\xi}_h\|_2^2\leq a\}$,
we prove the first bound.

For the second bound, we see that
\[
\ma_{\repss}^\infty(\alpha_t) = \{\max_{h\in  [H] } \|\bs{\sigma}(\bs{V}_{h,\bs{x}\bs{x}})^{-1}\bs{V}_{h, \bs{x}\bs{x}}^{1/2}\bs{\xi}_h\|_{\infty} \leq z_{1-\alpha_t/2}\}.
\]
Let $\bs{\xi}_h^\prime =\bs{D}(\bs{V}_{h,\bs{x}\bs{x}})^{-1/2} \bs{\sigma}(\bs{V}_{h,\bs{x}\bs{x}})^{-1}\bs{V}_{h, \bs{x}\bs{x}}^{1/2}\bs{\xi}_h$. Then,
\begin{align}
\label{eq:thm-S.3-2}
    \ma_{\repss}^\infty(\alpha_t) = \Big\{\max_{h\in  [H] } \|\bs{D}(\bs{V}_{h,\bs{x}\bs{x}})^{1/2}\bs{\xi}_h^\prime\|_{\infty} \leq z_{1-\alpha_t/2}\Big\}.
\end{align}

Noticing that 
\[
\bs{V}_{h, \bs{x}\bs{x}}^{1/2} \bs{\sigma}(\bs{V}_{h,\bs{x}\bs{x}})^{-1} \bs{D}(\bs{V}_{h,\bs{x}\bs{x}})^{-1} \bs{\sigma}(\bs{V}_{h,\bs{x}\bs{x}})^{-1}\bs{V}_{h, \bs{x}\bs{x}}^{1/2} = \bs{I}_K,
\]
 $\bs{\xi}_h^\prime$, $h\in [H]$, are also standard Gaussian random vectors and $\|\bs{\xi}_h^\prime\|_2^2 = \|\bs{\xi}_h\|_2^2$. Therefore, by \eqref{eq:thm-S.3-1} and \eqref{eq:thm-S.3-2}, we have
 \begin{align*}
     &\Prob_{\infty}( \hpL \leq \alpha\mid \ma_{\repss}(\alpha_t)) \leq \Prob \Big(\big(\varepsilon^{2}+\sum_{h=1}^H \|\bs{\xi}_h^\prime\|_2^2 \big)^{1/2} \geq z_{1-\alpha/2}~\Big | ~ \max_{h\in  [H] } \|\bs{D}(\bs{V}_{h,\bs{x}\bs{x}})^{1/2}\bs{\xi}_h^\prime\|_{\infty} \leq z_{1-\alpha_t/2}\Big).
 \end{align*}

Thus, we prove the second bound.

\noindent \textbf{Proof of \Cref{thm:inflation-bound-sre} (iii):}

Let $\ma_{\rep-\fe}^\infty(\alpha_t) = \{\|\bs{\sigma}(\bs{V}_{\omega,\bs{x}\bs{x}})^{-1}\bs{\varepsilon}_{\omega,\bs{x}}\|_{\infty} \leq z_{1-\alpha_t/2}\}$. Note that by \Cref{lem:asymptotic-limit-of-type-I-error-rate}, we have
\begin{align*}
  \Prob_\infty(\hpfe \leq \alpha \mid \ma_{\rep-\fe}(\alpha_t)) = \Prob(\max_{\mk \subseteq [K]}|\mathcal{N}_{\fe , \mk}| \geq z_{1-\alpha/2}\mid \ma_{\rep-\fe}^\infty(\alpha_t)) , 
\end{align*}
where 
\[
\mathcal{N}_{\fe , \mk} = \frac{\varepsilon_{\omega,\tau}-\bs{\beta}_{\fe ,\mk}^\top \bs{\varepsilon}_{\omega, \mk}}{n^{1/2}\tilde{\se}_{\fe , \mk}}.
\]

We define
\[
\tilde{\mathcal{N}}_{\fe ,\mk} = \frac{\varepsilon_{\omega,\tau}-\bs{\beta}_{\fe , \mk}^\top \bs{\varepsilon}_{\omega, \mk}}{\var(\varepsilon_{\omega,\tau}-\bs{\beta}_{\fe , \mk}^\top \bs{\varepsilon}_{\omega, \mk})^{1/2}}.
\]
By \Cref{lem:limit-of-standard-errors}, $n\tilde{\se}_{\fe }^2 \geq \var(\varepsilon_{\omega,\tau}-\bs{\beta}_{\fe ,\mk}^\top \bs{\varepsilon}_{\omega, \mk})$, and therefore $|\tilde{\mathcal{N}}_{\fe ,\mk}| \geq |\mathcal{N}_{\fe , \mk}|$.

Therefore, we have
\begin{align*}
     \Prob_{\infty}( \hpfe \leq \alpha\mid \ma_{\repfe}(\alpha_t)) \leq \Prob(\max_{\mk \subseteq [K]}|\tilde{\mathcal{N}}_{\fe , \mk}| \geq z_{1-\alpha/2}\mid \ma_{\rep-\fe}^\infty(\alpha_t)).
\end{align*}

Let $\varepsilon_0 = (\varepsilon_{\omega,\tau}- \bs{\beta}_{\omega, \bs{x}}^\top \bs{\varepsilon}_{\omega, \bs{x}})/\var(\varepsilon_{\omega,\tau}- \bs{\beta}_{\omega, \bs{x}}^\top \bs{\varepsilon}_{\omega, \bs{x}})^{1/2}$ and $\bs{\xi}_{0} = \bs{V}_{\omega, \bs{x}\bs{x}}^{-1/2}\bs{\varepsilon}_{\omega,\bs{x}}$. Define 
$\tilde{\bs{\beta}}_{\fe ,\mk}$ such that $(\tilde{\bs{\beta}}_{\fe ,\mk})_{\mk}=\bs{\beta}_{\fe ,\mk}$  and all other entries are $0$. We have
\begin{align*}
    & |\tilde{\mathcal{N}}_{\fe ,\mk}| = \Big|\frac{\varepsilon_{\omega,\tau}-\bs{\beta}_{\omega, \bs{x}}^\top\bs{\varepsilon}_{\omega,\bs{x}} + \bs{\beta}_{\omega, \bs{x}}^\top\bs{\varepsilon}_{\omega,\bs{x}} -\bs{\beta}_{\fe , \mk}^\top \bs{\varepsilon}_{\omega, \mk}}{\var(\varepsilon_{\omega,\tau}-\bs{\beta}_{\fe ,\mk}^\top \bs{\varepsilon}_{\omega, \mk})^{1/2}}\Big|\\
\leq & \frac{\var(\varepsilon_{\omega,\tau}-\bs{\beta}_{\omega, \bs{x}}^\top\bs{\varepsilon}_{\omega,\bs{x}})^{1/2}|\varepsilon_0| + \|(\bs{\beta}_{\omega, \bs{x}} -\tilde{\bs{\beta}}_{\fe , \mk})^\top \bs{V}_{\omega,\bs{x}\bs{x}}^{1/2}\|_2\|\bs{\xi}_0\|_2}{\var(\varepsilon_{\omega,\tau}-\bs{\beta}_{\fe, \mk}^\top \bs{\varepsilon}_{\omega, \mk})^{1/2}},
\end{align*}
Noticing that
\begin{align*}
    \var(\varepsilon_{\omega,\tau}-\bs{\beta}_{\fe, \mk}^\top \bs{\varepsilon}_{\omega, \mk}) &= \var(\varepsilon_{\omega,\tau}-\bs{\beta}_{\omega, \bs{x}}^\top\bs{\varepsilon}_{\omega,\bs{x}}) + \var(\bs{\beta}_{\omega, \bs{x}}^\top\bs{\varepsilon}_{\omega,\bs{x}} -\bs{\beta}_{\fe , \mk}^\top \bs{\varepsilon}_{\omega, \mk})\\
    &= \var(\varepsilon_{\omega,\tau}-\bs{\beta}_{\omega, \bs{x}}^\top\bs{\varepsilon}_{\omega,\bs{x}}) + (\tilde{\bs{\beta}}_{\fe , \mk} -\bs{\beta}_{\omega, \bs{x}})^\top \bs{V}_{\omega,\bs{x}\bs{x}}(\tilde{\bs{\beta}}_{\fe , \mk} -\bs{\beta}_{\omega, \bs{x}}),
\end{align*}
and using Cauchy-Schwarz inequality, we have
\begin{align*}
    |\tilde{\mathcal{N}}_{\fe ,\mk}| \leq (\varepsilon_0^2 + \|\bs{\xi}_0\|_2^2)^{1/2}.
\end{align*}

Since $\ma_{\repfe}^\infty(\alpha_t) = \{ \|\bs{\sigma}(\bs{V}_{\omega,\bs{x}\bs{x}})^{-1}\bs{V}_{\omega,\bs{x}\bs{x}}^{1/2}\bs{\xi}_0\|_{\infty} \leq z_{1-\alpha_t/2}\}$, by letting $\bs{\xi}_0^\prime = \bs{D}(\bs{V}_{\omega,\bs{x}\bs{x}})^{-1/2} $ $ \cdot\bs{\sigma}(\bs{V}_{\omega,\bs{x}\bs{x}})^{-1}\bs{V}_{\omega, \bs{x}\bs{x}}^{1/2}\bs{\xi}_0$, we have
\begin{align*}
    &\Prob_{\infty} (\hpfe \leq \alpha \mid \ma_{\repfe}(\alpha_t)) \leq \Prob\Big((\varepsilon_0^2 + \|\bs{\xi}_0\|_2^2)^{1/2} \geq z_{1-\alpha/2}~\Big | ~\ma_{\repfe}^\infty(\alpha_t)\Big) \\
    =&\Prob\Big((\varepsilon_0^2 + \|\bs{\xi}_0^\prime\|_2^2)^{1/2}\geq z_{1-\alpha/2}~\Big | ~ \ma_{\repfe}^\infty(\alpha_t)\Big)\\
    =& \Prob\Big((\varepsilon_0^2 + \|\bs{\xi}_0^\prime\|_2^2)^{1/2}\geq z_{1-\alpha/2}~\Big | ~ \|\bs{D}(\bs{V}_{\omega,\bs{x}\bs{x}})^{1/2}\bs{\xi}_0^\prime\|_{\infty} \leq z_{1-\alpha_t/2}\Big).
\end{align*}

Since $\bs{\xi}_0^\prime$ is a standard Gaussian random vector independent of $\varepsilon_0$, the conclusion follows.
\end{proof}

\end{document}